\documentclass[12pt]{amsart}

\usepackage{amsmath}
\usepackage{amsthm}
\usepackage{amssymb}
\usepackage{amscd}
\usepackage{amsfonts}
\usepackage{amsbsy}
\usepackage[noadjust]{cite} 
\usepackage{epsfig,afterpage}
\usepackage{paralist}
\usepackage{psfrag}
\usepackage{mathrsfs}       
\usepackage{verbatim}

\newtheorem {theorem} {Theorem}
\newtheorem {definition} [theorem]{Definition}
\newtheorem {proposition} [theorem]{Proposition}
\newtheorem {corollary} [theorem]{Corollary}
\newtheorem {lemma}  [theorem]{Lemma}
\newtheorem {remark} [theorem]{\sc Remark}
\newtheorem {conjecture} [theorem]{\sc Conjecture}

\usepackage{xcolor}
\usepackage{float}

\usepackage{lpic}
\newcommand{\ve}{\varepsilon}

\title[Monodromic Brunella-Miari fields]{Monodromic singularities of Brunella-Miari vector fields with two edges in the Newton diagram}

\thanks{ The authors are partially supported by MICIN grants PID2020-113758GB-I00 and PID2022-136613NB-I00 and AGAUR grants 2021SGR 01618 and 2021SGR 01039.}

\author[Isaac A. Garc\'ia, Jaume Gin\'e and V\'{\i}ctor Ma\={n}osa]{Isaac A. Garc\'ia$^{(1)}$, Jaume Gin\'e$^{(1)}$ and V\'{\i}ctor Ma\={n}osa$^{(2)}$}

\address{$^{(1)}$Departament de Matem\`atica, Universitat de Lleida,
Avda. Jaume II, 69, 25001 Lleida, Spain}

\address{$^{(2)}$Departament de Matem\`atiques (MAT), Institut de Matem\`atiques de la UPC (IMTech), Universitat Polit\`ecnica de Catalunya (UPC),
Colom 11, 08222 Terrassa, Spain}

\email{isaac.garcia@udl.cat}
\email{jaume.gine@udl.cat}
\email{victor.manosa@upc.edu}
\subjclass[2000]{34Cxx, 37G15, 37G10}

\keywords{Center, periodic orbits, Poincar\'e map, cyclicity}

\begin{document}

\begin{abstract}
This work focuses on the study of monodromic singularities in planar analytic families of vector fields whose Newton diagram consists of exactly two edges. We begin by analyzing the desingularization scheme of a \emph{minimal model} of polynomial vector fields, denoted by $\mathcal{X}$, which includes only the monomials corresponding to the vertices of the Newton diagram. We then extend this minimal model to the so-called \emph{Brunella-Miari vector fields} $\mathcal{X} \subset \mathcal{X}^{[1]}$, incorporating all monomials associated with points lying on the edges of the Newton diagram. As a second extension, we consider vector fields $\mathcal{X}^{[1]} \subset \mathcal{X}^{[2]}$ that include higher-order terms corresponding to points located above the polygonal line in the Newton diagram. The key point of our approach is to preserve the desingularization geometry at each extension step. We provide explicit desingularization procedures, which enable the computation of the linear part of the return map $\Pi$ in cases where the desingularized singularity is associated with a hyperbolic polycycle. Several nontrivial examples are included to illustrate the method.
\end{abstract}

\maketitle

\section{Introduction}

A singular point of a planar vector field is said to be \emph{monodromic} if it admits a well-defined local return map. Centers and foci are examples of  monodromic points in which nearby trajectories either form closed curves or spiral around the point. The problem of distinguishing between centers and foci, which we shall refer to as \emph{the stability problem}, is a classical issue in the analytic theory of planar differential equations. There exist techniques that, except for computational obstructions, allow this problem to be addressed in some cases.
For instance, when the linear part of the vector field is non-degenerate, nilpotent, or identically zero but without characteristic directions, which are  singular points appearing in the exceptional divisor after the first blow-up, see for example \cite{CZZ,FLLL,GGL,G0,G,G1,GM,RS,TLZ}. However, a significant gap remains in the understanding of the degenerate case, when the differential matrix at the singularity vanishes and \emph{characteristic directions} are present.

The lack of a general methodology to study the degenerate case in the presence of characteristic directions has motivated the development of various geometric and algebraic tools, including desingularization via blow-ups, Newton polygon techniques, asymptotic analysis and the characterization of inverse integrating factors \cite{AGGMed,AGG,Ga-Gi,GaGi2,GaGi3,GG-complex,GGGrau,RS} and \cite{M20}. Unlike in the classical cases, one of the main difficulties in the degenerate setting, when there exist characteristic directions, is that the map is no longer analytic but semi-regular (or Dulac-type) which prevents the application of standard analytic techniques.
In fact, even though it is known that it admits a linear leading term, its global structure remains unknown in general.

In this work, we investigate the stability problem and focus on computing the linear term of the return map for planar analytic vector fields whose Newton diagram has exactly two edges. We first consider the case where the vector field coincides with the Brunella-Miari principal part of the vector field defined by these edges, and then we explore some extensions. Our analysis is restricted to the situation where, after desingularization, the system exhibits a monodromic polycycle composed of hyperbolic saddles.

The structure of the paper is as follows: In Section \ref{s:notation}, we introduce some of the notation that will be used throughout the article. In Section \ref{s:minimalmodel}, we define the notion of minimal models of Brunella-Miari fields and characterize conditions under which such fields have a monodromic point at the origin using the algorithm developed in the works \cite{AGR, AGR2}. These monodromic minimal models have a reversible center at the origin. The main result of this section is Theorem \ref{t:prop-mon-family}. In Section \ref{s:minimalhyp}, we study  three families of minimal models that can be analyzed through specific desingularization schemes leading to hyperbolic polycycles. As before, all these minimal models correspond to reversible centers. The main results are Theorems \ref{t:Teo-main1},   \ref{t:Teo-main2} and \ref{t:Teo-main2-new}. In Section \ref{s:linearterm}, we characterize the linear term of the return map associated with certain monodromic points, including some of the cases studied in the subsequent sections. The main results here are Theorems \ref{t:v1} and \ref{t:v2}. In Section \ref{s:extensions}, we introduce two types of extensions of the minimal models corresponding to Brunella-Miari fields with a monodromic point at the origin. These are desingularized using the blow-up schemes presented in Section \ref{s:minimalhyp}. In, Section \ref{s:examples}, we provide specific examples of such extensions and compute the linear part of their Poincar\'e maps. The main results are Propositions \ref{prop-1221}, \ref{p:novapropTh2} and \ref{p:propoexttipusb}. In Section \ref{s:literatura}, we analyze how certain well-known examples in the literature can be generalized and/or revisited in light of the results presented in this paper. Section \ref{S-ClasesMon} is devoted to show other monodromic classifications not based on the number of edges of the Newton diagram as we do. In Section \ref{s:conjuecture} we discuss how the results of this work relate to a conjecture formulated in \cite{GaGi3}. Finally, all proofs are provided in the concluding Sections \ref{S-Proofs1}, \ref{S-Proofs2} and \ref{S-Proofs3}.

\section{Background and some notations}\label{s:notation}

Through this work we will use the following notation. Let \begin{equation}\label{e:campZ}
\mathcal{Z}= \sum_{(i,j)\in\mathbb{N}^2} a_{ij} x^i y^{j-1}\,{\partial x} \, +
\sum_{(i,j)\in\mathbb{N}^2} b_{ij} x^{i-1} y^j\,{\partial y}
\end{equation}
be an analytic vector field in a neighborhood of the origin of $\mathbb{R}^2$.  The \emph{Newton diagram} $\mathbf{N}(\mathcal{Z})$ of $\mathcal{Z}$ is the polygonal line that appears when considering the boundary (modulus the two open rays) of the convex hull of the set $\bigcup_{(i,j)\in\text{supp}(\mathcal{Z})} \{(i, j) + \mathbb{R}^2_+\}$, where $\text{supp}(\mathcal{Z}) = \{(i, j) \in \mathbb{N}^2 : (a_{ij}, b_{ij}) \neq (0,0)\}$.
The endpoints of the edges of $\mathbf{N}(\mathcal{Z})$, which lie in  $\mathbb{N}^2$, are its \emph{vertices}.

To each edge of $\mathbf{N}(\mathcal{Z})$, we associate the \emph{weights} $(p,q) \in \mathbb{N}^2$, with $p$ and $q$ coprimes, given by the tangent $q/p$ of the angle between that edge and the ordinate axis. In what follows, we denote by $W(\mathbf{N}(\mathcal{Z})) $ the set of all weights associated with the edges of the Newton diagram $\mathbf{N}(\mathcal{Z})$ determined by:
\begin{equation}\label{e:WNZ}
W(\mathbf{N}(\mathcal{Z})) = \{(p_1, q_1), (p_2, q_2), \ldots, (p_\ell, q_\ell)\} \text{ ordered by } \frac{q_1}{p_1} < \frac{q_2}{p_2} < \cdots < \frac{q_\ell}{p_\ell}.
\end{equation}
In this setting, the edge labeled $i$ is assigned the weights $(p_i, q_i)$, and its adjacent upper edge corresponds to edge $i - 1$. For a vertex with coordinates $(i, j) \in \text{supp}(\mathcal{Z})$, we define its \emph{vector coefficient} as $(\mathfrak{a}, \mathfrak{b}) = (a_{ij}, b_{ij})$.

The \emph{Brunella-Miari vector field},  $\mathcal{Z}_{\Delta}$ associated with of $\mathcal{Z}$, also called the \emph{principal part}, is defined  as follows. For each $(p_i, q_i) \in W(\mathbf{N}(\mathcal{Z}))$ with $i = 1, \ldots, k$, we have the $(p_i, q_i)$-quasihomogeneous expansion $\mathcal{Z}= \mathcal{Z}_{r_i} + \cdots$, where $\mathcal{Z}_{r_i}$ is a quasi-homogeneous vector field with type $(p_i, q_i)$ and degree $r_i$, \cite{AGGMed}. We define
$$
\mathcal{Z}_{\Delta} = \bigoplus_{i=1}^k \mathcal{Z}_i,
$$
where the sum $\oplus$ of two vector fields acts as the ordinary sum but the common monomial terms, corresponding to vertices of $\mathbf{N}(\mathcal{Z})$, only appear once instead of twice.

As a consequence of the results in \cite{BM}, see also \cite[Theorem 6]{GaGi3}, if $\mathcal{Z}$ is an analytic planar vector field having an isolated singularity at the origin, then it is monodromic if and only if the origin is monodromic for $\mathcal{Z}_{\Delta}$, provided that this last field is not degenerated, that is, no quasi-homogeneous component
$\mathcal{Z}_{r_i}$ of $\mathcal{Z}_{\Delta}$ have singularities outside the coordinate axes.

Suppose that there exists $(m,n)\in\mathbb{N}^2$ such that the weighed polar coordinates change $x=\rho^m\cos\theta$ and $y=\rho^n\sin\theta$ defines an analytic system
\begin{equation}\label{e:polarspolars}
\dot{\rho} = R(\rho,\theta),\,
\dot{\theta}  = \Theta(\rho,\theta),
\end{equation}
for $\rho\geq 0$ and $\theta\in[0,2\pi)$. Then, any $\theta_*\in[0,2\pi)$ such that $\Theta(0,\theta_*)=0$ is called a \emph{characteristic direction}. Of course, if $\theta_*$ is a characteristic direction, then $\theta_*+\pi$ is as well. Notice also that if, in the above change, $(m,n)=(1,1)$, which corresponds to the classical polar coordinates, or if $(m,n)\in W(\mathbf{N}(\mathcal{Z}))$, then the system \eqref{e:polarspolars} is analytic. In this situation we will see that the forthcoming minimal model \eqref{e:minimal} only has, at most, two characteristic directions in $[0,\pi)$, see Propositions \ref{p:prop-dir.car-family} and ~\ref{p:prop-dir.car-family-2}.

\medskip

 The following definitions establish the types of characteristic directions involved in the main results of the paper.

\begin{definition}\label{d:typeA}
Let $\theta_* = 0$ be characteristic direction associated to the monodromic singularity at the origin of an analytic vector field $\mathcal{Z} = P(x,y) \partial_x + Q(x,y) \partial_y$. We say that $\theta_*=0$ is of type $\mathcal{A}$  with  \emph{weights} $(p,q)\in W(\mathbf{N}(\mathcal{Z})) =\{ (p_1, q_1), (p_2, q_2)\}$, if it can can be desingularized using the blow-ups: $(x, y) \mapsto (z,w)$ with $z=x^{q}/y^{p}$ and $w = y/x$ which gives a hyperbolic saddle at the origin of the system  \eqref{e:cordA} below; and $(x,y) \mapsto (x, v)$ with $v = y^{p}/x^{q}$ which gives a regular flow in a neighborhood of the regular solution $x=0$ of the system  \eqref{e:cordB}    below.
\end{definition}

One of the main ideas of this work is to desingularize the origin of certain systems using algebraic blow ups whose structure is encoded in $W(\mathbf{N}(\mathcal{Z}))$, and to compute the linear coefficient of the associated Poincar\'{e} map.  This requires studying the flow transitions both inside sectors containing characteristic directions and outside them.  The latter are obtained from first-order variational equations in weighted polar coordinates, with weights determined by $W(\mathbf{N}(\mathcal{Z}))$, that must be compatible with the chosen algebraic blow ups. Type $\mathcal{A}$ singularities are compatible with classical polar coordinates. The following definition deals with certain pairs of characteristic directions that are desingularized through blow ups compatible with polar coordinates with weights $(p_2,q_2-p_2)$.

\begin{definition}\label{d:typeB}

Let the origin be a monodromic singularity  of an analytic vector field $\mathcal{Z} = P(x,y) \partial_x + Q(x,y) \partial_y$, and
let $\{0,\pi/2\}$ be a $2$-tuple of characteristic directions associated appearing when using weighed polar coordinates $x=\rho^{p_2}\cos(\theta)$, $y=\rho^{q_2-p_2}\sin(\theta)$ where $(p_2,q_2)\in W(\mathbf{N}(\mathcal{Z}))$. We say that the 2-tuple is of type $\mathcal{B}$   if it can can be desingularized using the blow-ups.

\begin{enumerate}
\item[(a)] Direction $\theta=0$:  $(x, y) \mapsto (z,w)$ with $z=x^{q_2}/y^{p_2}$, $w = y^{p_2}/x^{q_2-p_2}$  which gives a hyperbolic saddle at the origin of the system  \eqref{e:cordA} below; and $(x,y) \mapsto (x, v)$ with $v = y^{p_2}/x^{q_2}$ which gives a regular flow in a neighborhood of the regular solution $x=0$ of the system  \eqref{e:cordB}  below.
\item[(b)] Direction $\theta=\pi/2$: $(x, y) \mapsto (z,w)$ with $z=y/x$, $w = x^{q_2-p_2}/y^{p_2}$  which gives a hyperbolic saddle at the origin of \eqref{e:cordA}; and $(x,y) \mapsto (u, y)$ with $u = x/y$ which gives a regular flow in a neighborhood of the regular solution $y=0$ of the system  \eqref{e:cordC} below.
\end{enumerate}
\end{definition}

See the Figures \ref{f:fig1}--\ref{f:fig5} for an illustration of the different local charts. Here, the systems \eqref{e:cordA}, \eqref{e:cordB} and \eqref{e:cordC} are given by:
\begin{equation}\label{e:cordA}
\dot{z}=Z(z,w)=-az+o(1), \ \ \dot{w} = W(z,w)=bw+o(1);\tag{A}
\end{equation}
with $a,b \neq 0$ and $\lambda = a/b > 0$ being the so-called \emph{hyperbolicity ratio} of the saddle at the origin of \eqref{e:cordA}; and
\begin{equation}\label{e:cordB}
\dot{x}=X(x,v), \ \ \dot{v} = V(x,v),\tag{B}
\end{equation} with $X(0,v) = 0$, $V(0,v) \neq 0$, and
\begin{equation}\label{e:cordC}
\dot{u} = U(u,y), \ \  \dot{y}=Y(u,y) \tag{C}
\end{equation} with $Y(u,0) = 0$, $U(u,0) \neq 0$.

Of course, if $\theta_*$ is an $\mathcal{A}$-type characteristic direction or it belongs to a $\mathcal{B}$-type $2$-tuple then so is its complementary characteristic directions $\theta_*+\pi$. This is because the use of algebraic blow-ups covers both directions.

\begin{remark} \label{remark-rotation-A}
{\rm The  $\mathcal{A}$-type definition can be straightforward generalized to the case $\# W(\mathbf{N}(\mathcal{X})) \geq 2$ and the set of characteristic directions $\Omega = \{\theta_1,\ldots,\theta_n\}$ with $n \geq 2$. Let $\theta_i$ be an $\mathcal{A}$-type characteristic direction with $i =1, \ldots, n$, after a rotation of angle $-\theta_i$, that may change $\mathbf{N}(\mathcal{X})$, we obtain a new vector field having $0$ as $\mathcal{A}$-type characteristic direction which, under the blow-ups described in Definition \ref{d:typeA}, transforms into the following systems:
\begin{equation}\label{e:sistemesi}
\dot{z}=Z_i(z,w), \, \dot{w} = W_i(z,w) \mbox{ and }
\dot{x}=X_i(x, v), \, \dot{v} = V_i(x,v),
\end{equation}
that corresponds with the fields \eqref{e:cordA} and \eqref{e:cordB}, respectively. }
\end{remark}

We stress that all the characteristic directions that appear in the forthcoming Theorems \ref{t:Teo-main1} and \ref{t:Teo-main2}, and in Propositions \ref{prop-1221} and \ref{p:novapropTh2}, are of type $\mathcal{A}$, while the characteristic directions that appear in Theorems \ref{t:Teo-main2-new} and Proposition \ref{p:propoexttipusb} form $2$-tuples of type $\mathcal{B}$.

\section{The minimal model}\label{s:minimalmodel}

We consider the 4-parameter family of polynomial planar vector fields
\begin{equation}\label{e:minimal}
\dot{x} = A y^{2s-1} + B x^{2 \alpha} y^{2 \beta-1}, \ \ \dot{y} = C x^{2r-1} + D x^{2 \alpha-1} y^{2 \beta},
\end{equation}
with arbitrary odd degree $\max\{ 2s-1, 2 r-1, 2 (\alpha+ \beta) -1\}$. Let $\mathcal{X}$ be its associated vector field. Clearly, when the origin of \eqref{e:minimal} is \emph{monodromic} it becomes a \emph{time-reversible center}. We call \eqref{e:minimal} the \emph{minimal model} since it is the simplest arbitrary degree family satisfying both $\# W(\mathbf{N}(\mathcal{X})) = 2$ and $\mathcal{X} = \mathcal{X}_\Delta$. The simplicity comes because \eqref{e:minimal} has the minimum number of monomials preserving the former conditions or, in other words, because $\mathbf{N}(\mathcal{X})$ has only points associated to the mononials of \eqref{e:minimal} at its vertices.

In this section, we characterize the monodromy of the origin of the vector field (see the Theorem  \ref{t:prop-mon-family} below), as well as its possible characteristic directions (see Proposition \ref{p:prop-dir.car-family} and \ref{p:prop-dir.car-family-2}). All the proof are written in Section \ref{S-Proofs1}.

\begin{theorem}\label{t:prop-mon-family}
The origin is a monodromic singular point of the polynomial vector field \eqref{e:minimal}, and therefore is a center, if and only if its exponents belong to the set
\begin{eqnarray*}\label{E-1}
\mathcal{E} &=& \{ r, s, \alpha, \beta \in \mathbb{N} \backslash \{0\} : s - \beta \geq 1, r-\alpha \geq 1, (r-\alpha)(s-\beta) > \alpha \beta \},
\end{eqnarray*}
and its coefficients lie in the monodromic parameter space
\begin{eqnarray*}\label{L-1}
\Lambda &=& \{ (A, B,C,D) \in \mathbb{R}^4 : A C < 0, A (D (s-\beta) - B \alpha) < 0, \\
 & & C (B (\alpha-r) + D \beta) > 0  \}. \nonumber
\end{eqnarray*}
Moreover, restricting to $\mathcal{E} \cap \Lambda$, the conditions $\# W(\mathbf{N}(\mathcal{X})) = 2$ and $\mathcal{X} = \mathcal{X}_\Delta$  hold.
\end{theorem}

\begin{remark}
{\rm Under the restrictions in  $\mathcal{E} \cap \Lambda$, the vector field \eqref{e:minimal} is the simplest vector field with a monodromic singular point at the origin such that $\mathcal{X} = \mathcal{X}_\Delta$ and $\# W(\mathbf{N}(\mathcal{X}) = 2$ in the sense that no monomial in $\mathcal{X}$ corresponds to a point in $\mathbf{N}(\mathcal{X})$ that is not a vertex of it. }
\end{remark}

The following proposition computes the characteristic directions $\theta_* \in[0, \pi)$ of the origin of \eqref{e:minimal} in polar coordinates that we will need to perform blow-ups in the proof of Theorem \ref{t:Teo-main1}. Recall that each characteristic direction $\theta_*$ has associated a symmetric one $\theta_*+\pi \in[\pi, 2\pi)$.

\begin{proposition}\label{p:prop-dir.car-family}
Let the origin be a monodromic singular point of the polynomial vector field \eqref{e:minimal}. Then, the characteristic directions $\theta_* \in[0, \pi)$ of the origin in polar coordinates are classified as follows: (i) if $r < \min\{s, \alpha + \beta\}$ then $\theta_*= \pi/2$; (ii) if $s < \min\{r, \alpha + \beta\}$ then $\theta_*= 0$; (iii) if $\alpha + \beta < \min\{r, s\}$ and $D \neq B$ then $\theta_* \in \{ 0, \pi/2\}$.
\end{proposition}

The next proposition computes the characteristic directions $\theta_* \in[0, \pi)$ of the origin of \eqref{e:minimal} in $(p,q)$-weighted polar coordinates with $(p,q) \in W(\mathbf{N}(\mathcal{X})$.

\begin{proposition}\label{p:prop-dir.car-family-2}
Let the origin be a monodromic singular point of the polynomial vector field \eqref{e:minimal}. Then, the characteristic directions $\theta_* \in[0, \pi)$ of the origin in $(p,q)$-weighted polar coordinates with $(p,q) \in W(\mathbf{N}(\mathcal{X}))$ are classified as follows: If $(p,q)=(p_1, q_1) = (s-\beta, \alpha)$ then $\theta_* =0$ while if $(p,q)=(p_2, q_2) = (\beta, r-\alpha)$ then $\theta_* = \pi/2$.
\end{proposition}

\section{minimal models with hyperbolic monodromic polycycle}\label{s:minimalhyp}

The main results of this section are Theorems \ref{t:Teo-main1}, \ref{t:Teo-main2}, and \ref{t:Teo-main2-new}, and are proved in Section \ref{S-Proofs1}. They characterize the monodromic singular points of the minimal models with the characteristic directions described in each case of Proposition \ref{p:prop-dir.car-family} or Proposition \ref{p:prop-dir.car-family-2} that are desingularized via a specific sequence of blow-ups that transforms the origin of the vector field into a monodromic polycycle composed of hyperbolic saddles.

We would like to point out that, the origin of any monodromic minimal model always desingularizes but different desingularization schemes appear in each particular case. We have not been able to explicitly determine the desingularization scheme of every minimal model.
Indeed, in this work, the desingularization schemes of minimal models lead to hyperbolic polycycles; however, we stress that there are examples of minimal models with non-hyperbolic monodromic polycycles, see Section \ref{ss:nonhyp}.

In all the cases described in the above mentioned results, and since they deal with particular cases of the minimal model \eqref{e:minimal}, the characterized singularities are reversible centers.

\subsection{The minimal model *}

In this section, we consider a general family of minimal models with a monodromic singularity at the origin, characterized by a specific geometry of its associated monodromic polycycle.

We define the subset of exponents $\mathcal{E}^* \subset \mathcal{E}$ as
\begin{equation}\label{exp-enteros+}
\mathcal{E}^* = \{ (\alpha, \beta, r, s) \in \mathcal{E} : e_i(\alpha, \beta, r, s) \in \mathbb{N} \cup \{0\}, i=1,\ldots, 8\},
\end{equation}
where
\begin{eqnarray} \label{exponents-ei}
e_1 &=& \frac{2 (s-\alpha-\beta)}{r - \alpha - \beta}, \ \ e_2 = \frac{2(-r s + s \alpha + r \beta)}{-r + \alpha + \beta}, \ \ e_3 = \frac{2 (r - \alpha-\beta)}{s - \alpha - \beta}, \nonumber \\
e_4 &=& \frac{2(r s - s \alpha - r \beta)}{s - \alpha - \beta}, \ \ e_5 = \frac{2 (r s - s \alpha - r \beta)}{\alpha}, \ \ e_6 = \frac{2r-\alpha}{\alpha}, \\
e_7 &=& \frac{2 (r s - s \alpha - r \beta)}{\beta}, \ \ e_8 = \frac{2 s-\beta}{\beta}. \nonumber
\end{eqnarray}
The conditions in $\mathcal{E}^*$ ensure that the characteristic directions are  of $\mathcal{A}$ type, and that after a specific sequence of blow-ups given below, the associated vector fields are still polynomial ones. Naturally, this set is non-empty: here we show a short list of some cases of $(\alpha, \beta, r, s) \in \mathcal{E}^*$: If $\alpha = \beta = 1$ then the smallest $(r, s)$ are $(r, s) \in \{ (3,3), (3,4) \}$. If $\alpha = 1$, $\beta = 2$ then the smallest $(r, s)$ are $(r, s) \in \{ (4,4), (4,5) \}$.

We also define the monodromic parameter subset $\Lambda^* \subset \Lambda$ by
\begin{equation}\label{mon-sillas-hyper}
\Lambda^* = \Lambda \cap \{ (B \alpha - D (s - \beta)) (B-D) > 0, (B (r -\alpha) - D \beta) (B-D) > 0 \}.
\end{equation}

\begin{theorem}\label{t:Teo-main1}
Let the origin be a monodromic singular point of the polynomial vector field \eqref{e:minimal} restricted to $\mathcal{E}^* \cap \Lambda^*$ with $\mathcal{E}^* \subset \mathcal{E}$ and $\Lambda^* \subset \Lambda$ defined in \eqref{exp-enteros+} and \eqref{mon-sillas-hyper}. Then, in polar coordinates there are two characteristic directions $\{0, \pi/2 \}$ and $(1,1) \not\in W(\mathbf{N}(\mathcal{X})) = \{ (p_1, q_1), (p_2, q_2)\}$. Moreover, both characteristic directions are of type $\mathcal{A}$  with weights $(p_2, q_2)$
 and $(p_1, q_1)$, respectively. Furthermore, the associated  blow-up polycycle is hyperbolic.
\end{theorem}

\begin{remark}\label{rem-1}
{\rm  In the proof of Theorem \ref{t:Teo-main1} we perform a blow-up $(x,y) \mapsto (w, z_2)$ with $w = y/x$ and $z_2=x^{q_2}/y^{p_2}$ and also the blow-up given by $(x,y) \mapsto (z, w_1)$ with $z = x/y$ and $w_1=y^{p_1}/x^{q_1}$ both designed to reach a hyperbolic saddle singularity associated to each characteristic direction. These changes only work in case that $p_i \neq q_i$ for $i=1,2$, which is, assured when we restrict the exponents to  $\mathcal{E}^*$. In short, two cases must be analyzed separately, namely $p_1=q_1$ and $p_2=q_2$. This is done in Theorem \ref{t:Teo-main2}. }
\end{remark}

\subsection{The minimal model $\dag$}

In this case, we confine ourselves in the specific case $(p_1, q_1) \in W(\mathbf{N}(\mathcal{X}))$ with $p_1=q_1$, equivalently we consider the parameter restriction
\begin{equation}\label{parametersp1q11}
s= \alpha+\beta.
\end{equation}
In this situation we define the exponents subset $\mathcal{E}^\dag \subset \mathcal{E}$ as
$$
\mathcal{E}^\dag = \{ (\alpha, \beta, r, s) \in \mathcal{E} : p_1 = q_1, \ e_i(\alpha, \beta, r, s) \in \mathbb{N} \cup \{0\}, i=7,8\},
$$
where $e_7$ and $e_8$ are defined in \eqref{exponents-ei}. Again, we notice that this set is non-empty: a short list of some $(\alpha, \beta, r, s) \in \mathcal{E}^\dag$ with $\alpha=1, 2$ and small exponents is:
$$
(\alpha, \beta, r, s) \in \{ (1, 1, 3, 2), (1, 1, 4, 2), (1, 2, 5, 3), (2, 2, 5, 4), (2, 4, 7, 6) \}.
$$
We also define the monodromic parameter subset $\Lambda^\dag \subset \Lambda$ by
\begin{equation}\label{mon-sillas-hyper-d}
\Lambda^\dag = \Lambda \cap \{ p_1 =  q_1, \ (B (r -\alpha) - D \beta) (B-D) > 0 \}.
\end{equation}

\begin{theorem}\label{t:Teo-main2}
Let the origin be a monodromic singular point of the polynomial vector field \eqref{e:minimal} where $W(\mathbf{N}(\mathcal{X})) = \{ (p_1, q_1), (p_2, q_2)\}$ with $p_1=q_1$. If  the field \eqref{e:minimal} is restricted to the parameter subset $\mathcal{E}^\dag \cap \Lambda^\dag$, then, in polar coordinates there is only one characteristic direction $\{0\}$, which is of type $\mathcal{A}$ with weights $(p_2, q_2)$, and the associated blow-up polycycle is hyperbolic.
\end{theorem}

\begin{remark}
{\rm We do not write an analogous result to Theorem \ref{t:Teo-main2} analyzing the case $p_2 = q_2$, hence with $p_1 \neq q_1$, because this case reduce to Theorem \ref{t:Teo-main2} taking into account the fact that when  $W(\mathbf{N}(\mathcal{X})) = \{ (p_1, q_1), (p_2, q_2) \}$ by taking change $(x,y) \mapsto \phi(x,y) = (y, x)$, one has  $W(\mathbf{N}(\phi_*\mathcal{X})) = \{ (q_2, p_2), (q_1, p_1) \}$.}
\end{remark}

\subsection{The minimal model $\ddag$}\label{ss:ddag}

The cases discussed in this section are desingularized using a different procedure than those in the previous sections. We take the exponents subset $\mathcal{E}^\ddag \subset \mathcal{E}$ as
\begin{equation}\label{exp-enteros-DD}
\mathcal{E}^\ddag = \{ (\alpha, \beta, r, s) \in \mathcal{E} : \alpha = s - \beta, \ \ n_i(\alpha, \beta, r, s) \in \mathbb{N} \cup \{0\}, i=1,2,3, 4\},
\end{equation}
where
\begin{eqnarray*}
n_1 &=& - \frac{2 (s \alpha + r (\beta-s))}{\beta^2}, \  n_2 =\frac{2 (r (s - \beta) + \beta^2 - s (\alpha + \beta))}{\beta^2},  \\
n_3 &=&  \frac{2 (r-s)}{\beta}, n_4 = \frac{2 (r - s) (r - s + \beta)}{\beta}.
\end{eqnarray*}
Also we define the subset $\Lambda^\ddag \subset \Lambda$ as
\begin{equation}\label{mon-sillas-hyper-d2}
\Lambda^\ddag = \Lambda \cap \{ \big( B r - B \alpha - D \beta \big) \big( D \beta + B (-r + \alpha + \beta) \big) < 0 \}.
\end{equation}
We notice that $\mathcal{E}^\ddag \neq \emptyset$, as for example
$$
(\alpha, \beta, r, s) \in \{ (1,1, 4, 2), (1,1, 5, 2), (1,1, 6, 2), (1,1,7, 2) \} \subset \mathcal{E}^\ddag.
$$

\begin{theorem}\label{t:Teo-main2-new}
Let the origin be a monodromic singular point of the polynomial vector field \eqref{e:minimal} where $W(\mathbf{N}(\mathcal{X})) = \{ (p_1, q_1), (p_2, q_2)\} =\{ (s-\beta, \alpha),  (\beta, r-\alpha) \}$ with $p_1=q_1$.  We   restrict family \eqref{e:minimal} to the parameter subset $\mathcal{E}^\ddag \cap \Lambda^\ddag$ defined in \eqref{exp-enteros-DD} and \eqref{mon-sillas-hyper-d2}. If the polar vector field associated to  weighted polar coordinates $x=\rho^{p_2}\cos(\theta)$, $y=\rho^{q_2-p_2}\sin(\theta)$ is analytic at $\rho=0$ and possesses the only characteristic directions  $\{0,\pi/2\}$ then the this $2$-tuple of characteristic directions are of type $\mathcal{B}$ and the associated polycycle is hyperbolic.
\end{theorem}

\subsection{Minimal model with non-hyperbolic monodromic polycycle}\label{ss:nonhyp}

In \cite{Ga-Li-Ma-Ma} it is proved that family
\begin{equation}\label{Ej2-Nonlinearity}
\dot{x}= y(\alpha x^2 + b x y+ c y^2),  \ \ \dot{y}= y^2(\alpha x + b y) + x^5,
\end{equation}
has a monodromic singularity at the origin when it is restricted to the parameter's set $\Lambda = \{(\alpha, b, c) \in \mathbb{R}^3 :  \alpha < 0, c < 0 \}$, and Poincar\'e map has the form $\Pi(x) = x + o(x)$.

The family \eqref{Ej2-Nonlinearity} has $W(\mathbf{N}(\mathcal{X})) = \{ (1, 1), (1,2) \}$, it is a minimal model \eqref{e:minimal} only when $b=0$ (just take $(\alpha, \beta, r, s)=(1, 1, 3,2)$ and $(A,B,C,D)=(c, \alpha, 1, \alpha)$) and $\mathcal{X} = \mathcal{X}_\Delta$ when $b \neq 0$. The origin of \eqref{Ej2-Nonlinearity} has associated a non-hyperbolic monodromic polycycle since it contains a saddle-node singularity. This follows after desingularizing the only characteristic direction $\theta_*=0$ with the two blow-ups used in the proof of Theorem \ref{t:Teo-main1}. More explicitly, after the blow-up $(x,y) \mapsto (x, w_2)$ with $w_2 = y/x^{2}$ and the time-rescaling removing the common factor $x^3$, we obtain that \eqref{Ej2-Nonlinearity} becomes
\begin{eqnarray*}
\dot{x} &=& w_2 x (\alpha + b w_2 x + c w_2^2 x^2),\\
\dot{w}_2 &=& 1 - b w_2^3 x - 2 c w_2^4 x^2 - \alpha w_2^2,
\end{eqnarray*}
so that $\dot{w}_2|_{x=0} = 1 - \alpha w_2^2$ has no real roots, since $\alpha < 0$, in $\Lambda$ giving rise to the regular part of the flow. Finally we perform the blow-up $(x, y) \mapsto (w, z_2)$ with $w = y/x$ and $z_2=x^{2}/y$, we remove the common factor $z_2^2 w^3$ and we obtain a differential system
\begin{eqnarray*}
\dot{z}_2 &=& \alpha z_2 + z_2(- z_2^2 + w (b + 2 c w)), \\
\dot{w} &=& w (z_2^2 - c w^2),
\end{eqnarray*}
that has a saddle-node at the origin.

\section{Linear term of the Poincar\'e map}\label{s:linearterm}

Since the pioneering work of N.B. Medvedeva and her collaborators, see \cite{Be-Me,Med1} for instance, it has been well established that the Poincar\'e map of a monodromic point, even in the most degenerate case, possesses a linear principal term, despite being only a semi-regular map. In this section we provide the necessary results to compute the linear term of the Poincar\'e map for the considered monodromic points by taking the appropriate extensions of the minimal models considered before. Our main result are Theorems \ref{t:v1} and \ref{t:v2} below and they are proved in Section \ref{S-Proofs2}.

The first result in this section is concerning degenerate monodromic points with $\mathcal{A}$-type characteristic directions. We use the straightforward generalization when we have $n \geq 2$ characteristic directions of type $\mathcal{A}$ and $\# W(\mathbf{N}(\mathcal{X})) \geq 2$ given in Remark \ref{remark-rotation-A}.

\begin{theorem}\label{t:v1}
Let the origin be a monodromic point of an analytic vector field such that all the characteristic directions $\theta_1,\ldots,\theta_n\in[0,\pi)$, obtained in any weighted polar coordinates, are of type $\mathcal{A}$. Let $\mathcal{A}^*=\{i_1<\cdots < i_k\}\subseteq\{1,\ldots,n\}$ such that $\{\theta_{i_j},\, j=1,\ldots, k\}$ is the set of characteristic directions with odd weights $p_{i_j}$ and $q_{i_j}$, and let
$X_{i_j}$ and $V_{i_j}$ be the components of the second differential system in \eqref{e:sistemesi}, associated with each characteristic direction.   Then, the Poincar\'e map defined for $x\gtrsim 0$ is $\Pi(x)=\eta x+o(x)$ with
\begin{equation}\label{e:tv1}
\eta=\exp\left\{ \mathrm{PV} \int_0^{2\pi} \mathcal{F}(\theta)d\theta+2\,\sum\limits_{_{i_j}\in\mathcal{A}^*} \frac{1}{\lambda_{i_j} p_{i_j}} \mathrm{PV} \int_{-\infty}^{\infty}
\frac{\partial}{\partial x} \left.\left(\frac{X_{i_j}(x,v)}{V_{i_j}(x,v)}\right)\right|_{x=0} dv\right\}
\end{equation}
where $\mathrm{PV}$ stands for the principal value of the integral, and where $\mathcal{F}(\theta)$ denotes the leading term of the differential equation $d\rho/d\theta = \mathcal{F}(\theta)\rho + o(\rho)$, derived from the original system using the initial weighted polar coordinates, and where
$$\lambda_{i_j}=-\left.\left(\frac{\partial W_{i_j}(z,w)/\partial w}{\partial Z_{i_j}(z,w)/\partial z}\right)\right|_{(z,w)=(0,0)}
$$
are the hyperbolicity ratios of the saddles appearing in the corresponding local charts.
\end{theorem}

Of course, if $\mathcal{A}^*$ is empty, then $
\eta=\exp\left\{ \mathrm{PV} \int_0^{2\pi} \mathcal{F}(\theta)d\theta\right\}$. Notice that, in formula \eqref{e:tv1}, we only count characteristic directions in $[0, \pi)$, not their complementary counterparts in $[\pi, 2\pi)$.

\medskip

The next result concerns degenerate monodromic points with $\mathcal{B}$-type $2$-tuples of  characteristic directions.

\begin{theorem}\label{t:v2}
Let the origin be a monodromic point of an analytic vector field  $\mathcal{Z}$ having $W(\mathbf{N}(\mathcal{Z})) = \{ (p_1,q_1), (p_2, q_2) \}$ and a $2$-tuple of $\mathcal{B}$-type characteristic directions $\theta_* \in \{0,\pi/2\}$ computed using the weighted polar coordinates $x = \rho^{p_2} \cos\theta$ and $y = \rho^{q_2-p_2} \sin\theta$. Then, the Poincar\'e map computed in these weighted polar coordinates with transversal section $\theta=\bar{\theta}\in (0,\pi/2)$ is $\Pi(\rho)=\eta \rho+o(\rho)$ with
\begin{eqnarray}\label{e:tv2}
\eta & = & \exp\left\{ \mathrm{PV} \int_0^{2\pi} \mathcal{F}(\theta)d\theta  +
\frac{2\mathcal{O}}{\lambda_{0}p_2^2}  \mathrm{PV} \int_{-\infty}^{\infty}
\frac{\partial}{\partial x} \left.\left(\frac{X(x,v)}{V(x,v)}\right)\right|_{x=0} dv\right.\\
 & & \left.  -\frac{2}{\lambda_{\frac{\pi}{2}} }  \mathrm{PV} \int_{-\infty}^{\infty}
\frac{\partial}{\partial y} \left.\left(\frac{Y(u,y)}{U(u,y)}\right)\right|_{y=0} d u  \right\}\nonumber
\end{eqnarray}
where $\mathcal{O}=1$ if $p_2$ and $q_2$ are odd and $\mathcal{O}=0$ otherwise; $\mathrm{PV}$ stands for the principal value of the integral;  where $\mathcal{F}(\theta)$ denotes the leading term of the differential equation $d\rho/d\theta = \mathcal{F}(\theta)\rho + o(\rho)$, derived from the original system using the weighted polar coordinates introduced above; and where
$$\lambda_{0,\frac{\pi}{2}}=-\left.\left(\frac{\partial W_{0,\frac{\pi}{2}}(z,w)/\partial w}{\partial Z_{0,\frac{\pi}{2}}(z,w)/\partial z}\right)\right|_{(z,w)=(0,0)}$$
are the hyperbolicity ratios of the saddles appearing in the corresponding local charts \eqref{e:cordA} for the directions $0$ and $\pi/2$, respectively.
\end{theorem}

We observe that $\mathrm{PV} \int_0^{2\pi} \mathcal{F}(\theta)d\theta = 0$ in \eqref{e:tv2} when either $p_2$ or $q_2-p_2$ is even, see \cite{GaGi3}. The proofs of the above results are technical and are provided in the proofs sections.

\section{The first and second extensions of the minimal models}\label{s:extensions}

Let $\mathcal{X}$ denote the vector field of the minimal model \eqref{e:minimal}.
We introduce two types of extensions of the vector field $\mathcal{X}$,  related to $\mathbf{N}(\mathcal{X})$. The \emph{first extension} $\mathcal{X}^{[1]}$ is the most general vector field such that $\mathcal{X} \subset \mathcal{X}^{[1]}$, meaning that $\mathcal{X}$ is a member of the family $\mathcal{X}^{[1]}$; $\mathbf{N}(\mathcal{X}^{[1]}) = \mathbf{N}(\mathcal{X})$; and $\mathcal{X}^{[1]} = (\mathcal{X}^{[1]})_\Delta$. That means we add the maximum number of monomials to $\mathcal{X}$ that are associated to points on the edges of $\mathbf{N}(\mathcal{X})$ but they are not vertices. In other words, since $W(\mathbf{N}(\mathcal{X})) = \{ (p_1, q_1), (p_2, q_2) \}$, then  $\mathcal{X}^{[1]}$ is the expansion in $(p_k, q_k)$-quasihomogeneous vector fields with the weighted degrees $r_1$ and $r_2$ given in
\begin{equation}\label{The-rk}
r_1 = s (2 \alpha - 1) + \beta - \alpha, \, r_2 = \alpha - \beta + r (2 \beta - 1)
\end{equation}
(see the proof of Proposition \ref{t:prop-mon-family}).
In other words, $\mathcal{X}^{[1]}=(\mathcal{X}^{[1]})_\Delta= \mathcal{X}_{r_1} \oplus  \mathcal{X}_{r_2}$.
In summary, the monomials $(a x^i y^{j - 1}) \partial_x + (b x^{i-1} y^{j}) \partial_y$ with $(i,j) \in \mathbb{N}^2$ are added if and only if $p_k i + q_k (j-1) = p_k + r_k$ and $p_k (i-1) + q_k j = q_k + r_k$ for some $k \in \{1,2\}$ and either $0 < i < 2 \alpha$ if $k=1$ or $2 \alpha < i < 2 r$ if $k=2$.
\medskip

Finally, we introduce the \emph{second extension}, $\mathcal{X}^{[2]}$, by adding to $\mathcal{X}^{[1]}$ arbitrary higher-order terms with respect to $\mathbf{N}(\mathcal{X})$. More precisely, we define $\mathcal{X}^{[2]} = \mathcal{X}^{[1]} + \mathcal{X}_*$ where
$$
\mathcal{X}_* = \Big(\sum_{(i, j) \in S} a_{i,j-1} x^i y^{j-1} \Big) \partial_x + \Big(\sum_{(i, j) \in S} b_{i-1,j} x^{i-1} y^{j} \Big) \partial_y
$$
being $(i,j) \in S \subset \mathbb{N}^2$ if and only if $(i, j)$ lies in the upper half plane with respect to $\mathbf{N}(\mathcal{X})$. In other words,
\begin{equation}\label{The-set-S-HOT}
S = \{ (i,j) \in \mathbb{N}^2 : p_k i + q_k (j-1) > p_k + r_k, p_k (i-1) + q_k j > q_k + r_k \}.
\end{equation}

Let $\Omega \subset [0,2\pi)$ be the set of characteristic directions of $\mathcal{X}$.
It is clear that if there exists a specific finite sequence of blow-ups desingularizing the origin of $\mathcal{X}^{[1]}$ (a \emph{blow-up scheme}), then the origin of $\mathcal{X}$ follows the same scheme, since $\mathcal{X} \subset \mathcal{X}^{[1]}$; but the converse is not generally true (see the example in Section \ref{S-Monstruo-1}).

We emphasize that the first extension is critical in the sense that it may modify the monodromy at the origin as well as the cardinality $\# \Omega$. In particular, when $\mathcal{X}^{[1]}$ has a different set of characteristic directions than $\mathcal{X}$, these vector fields cannot share the same blow-up desingularization process. On the contrary, the second extension preserves both the monodromy at the origin and the set $\Omega$ of $\mathcal{X}^{[1]}$. In our examples, we also observe that the second extension preserves the value of $\eta$; however, we do not currently have a proof of this fact in the general setting. This behaviour suggest to consider the following open problem: Set $\mathcal{X} = \mathcal{X}_\Delta + \cdots$, where the dots mean higher order terms with respect to $\mathbf{N}(\mathcal{X})$. Suppose that the Poincar\'e maps $\Pi(x) = \eta x + o(x)$ and $\Pi_\Delta(x) = \eta_\Delta x + o(x)$ of the origin of $\mathcal{X}$ and $\mathcal{X}_\Delta$, respectively, are well-defined. Then,  is it true that $\eta = \eta_\Delta$?

\subsection{Changes in the blow-up scheme for first extension fields} \label{S-Monstruo-1}

In this section, we provide a simple example illustrating how the first extension may alter the blow-up scheme at the origin. We consider a vector field $\mathcal{Y} = \mathcal{X} + (a x^i y^{j - 1}) \partial_x + (b x^{i-1} y^{j}) \partial_y$ with $(i,j) \in \mathbb{N}^2$ such that $\mathcal{Y} \subset \mathcal{X}^{[1]}$. In other words such that $(i,j)$ lies on a particular edge of $\mathbf{N}(\mathcal{X})$.

For the sake of simplicity, we assume that $p_1 = q_1$.  If we want to desingularize the characteristic direction $\theta_1^* = 0$ of $\mathcal{Y}$ by doing the blow-up $(x,y) \mapsto (w_2, x)$ with $w_2 = y^{p_2}/x^{q_2}$ and the corresponding time-rescaling, as we did in the proof of Theorem \ref{t:Teo-main2}, the outcome is that we get a polynomial differential system $\dot{x} = X(x, w_2)$, $\dot{w}_2 = W_2(x, w_2)$  if and only if $(i,j)$ is such that
\begin{equation}\label{rest-exp-regularpart-ext}
\left(-1 + \frac{j}{s-\alpha}, i-2 r + \frac{j(r-\alpha)}{s-\alpha}\right) \in \mathbb{N}^2 \cup \{0\},
\end{equation}
satisfies the conditions in $\mathcal{E}^\dag$. This condition on $(i,j)$ gives the allowed points on $\mathbf{N}(\mathcal{X})$ (different from the vertices) for which the geometry of the desingularizing blow-ups of $\mathcal{X}$ and $\mathcal{Y}$ coincide. In other words: in order that $\mathcal{X}$ and $\mathcal{Y}$ share the same desingularization scheme, the restriction \eqref{rest-exp-regularpart-ext} needs to be imposed.

For example, adding the point $(i,j)=(3, 9)$ that lies on the first edge of $\mathbf{N}(\mathcal{X})$ for the minimal model with exponents $(\alpha, \beta, r, s) = (2, 4 , 8, 6) \in \mathcal{E}^\dag$, the resulting vector field  does not satisfy \eqref{rest-exp-regularpart-ext}. The vector field $\mathcal{Y}$ of this example is $\dot{x} = B x^4 y^7 + A y^{11} + a x^3 y^8$, $\dot{y} = C x^{15}+ D x^3 y^{8} + b x^2 y^9$.

 The associated minimal model $\mathcal{X}$ fits within the framework of Theorem \ref{t:Teo-main2} when it is restricted to the parameter set \eqref{mon-sillas-hyper-d} that becomes $\Lambda^\dag = \{ A C < 0,  A (D-B) < 0, C(3 B -2 D) < 0 \}$. In this case, the characteristic direction $\{0\}$ of the associated minimal model is of type $\mathcal{A}$ and, therefore, can be  desingularized with the scheme that characterizes these type of directions. However, this scheme does not applies to  $\mathcal{Y}$.

In order to desingularize the characteristic direction $\{\theta=0\}$ of the origin of $\mathcal{Y}$, we perform the following blow-up sequence: We consider the change of coordinates $(x, y) \mapsto (x, z)$, where $z = y/x$, and rescale by dividing by $x^{10}$. To further desingularize point $(x,z)=(0,0)$, we first apply $(x, z) \mapsto (x, z_2)$ with $z_2 = z/x$, and rescale by dividing by $x^3$, so that the resulting polynomial vector field $\dot{x} = X(x, z_2)$, $\dot{z}_2 = Z_2(x, z_2)$ satisfies $Z_2(0, z_2) = C \neq 0$, providing a regular  flow. Secondly, we perform the transformation $(x, z) \mapsto (w, z)$ with $w = x/z$, followed by a division by $z^3$. This results in a polynomial vector field $\dot{w} = W(w, z)$, $\dot{z} = Z(w, z)$, whose origin is again degenerated.

To desingularize the point $(w,z)=(0,0)$, next, we apply $(w, z) \mapsto (w, u)$ with $u = z/w$ and divide by $w^4$, obtaining the polynomial vector field $\dot{w} = -C w + o(w,u)$, $\dot{u} = 2C u + o(w,u)$, so that the origin becomes a hyperbolic saddle. Finally, we apply the blow-up $(w, z) \mapsto (v, z)$ with $v = w/z$ and divide by $z^4$ to the system $\dot{w} = W(w, z)$, $\dot{z} = Z(w, z)$, resulting in the polynomial vector field $\dot{v} = (3B - 2D) v + o(v,z)$, $\dot{z} = (D - B) z + o(v,z)$, so that the origin becomes a hyperbolic saddle on $\Lambda^\dag$.

In summary, the sequence of elementary blow-ups is: $(x, y) \mapsto (x, z_2) = (x, y/x^2)$ that gives a regular part;  $(x, y) \mapsto (w, u) = (x^2/y, y^2/x^3)$ that produces a hyperbolic saddle; and, finally, $(x, y) \mapsto (v, z) = (x^3/y^2, y/x)$ that yields the last hyperbolic saddle. Therefore, we conclude that the blow-up scheme for $\mathcal{Y} \subset \mathcal{X}^{[1]}$ differs from that for $\mathcal{X}$, given by  Theorem \ref{t:Teo-main2}.  This example also shows that the minimal model $\mathcal{X}$ admits at least two different desingularizations.

\section{Extensions of particular minimal models}\label{s:examples}

In this section we compute the liner term of the Poincar\'e map of extensions of concrete examples of minimal models. The proofs are given in Section \ref{S-Proofs3}.

\subsection{An example of extension of minimal models within the framework of Theorem \ref{t:Teo-main1}}

We consider the family \eqref{e:minimal} with exponents given by the values $(\alpha, \beta, r, s) = (1, 1, 3, 3) \in \mathcal{E}^*$, that is, the vector field $\mathcal{X}$ given by
\begin{equation}\label{Ex-1}
\dot{x} = B x^2 y + A y^{5}, \ \ \dot{y} = D x y^2 + C x^{5},
\end{equation}
and we restrict the family to $\Lambda = \Lambda^* = \{ A C < 0, A (2 D-B) < 0, C (D-2 B) > 0 \}$. This family has a reversible center at the origin with $W(\mathbf{N}(\mathcal{X})) = \{(2, 1), (1,2) \}$ with characteristic directions $\Omega_{11} = \{0,\pi/2\}$ and a desingularization process like the one in Theorem \ref{t:Teo-main1}.

\begin{proposition}\label{prop-1221}
The first extension $\mathcal{X}^{[1]}$ of the minimal model  $\mathcal{X}$ given by \eqref{Ex-1}, with the assumptions of Theorem \ref{t:Teo-main1} is:
\begin{equation}\label{Ex-1-ext}
\dot{x} = B x^2 y + A y^{5} + a_1 x y^3 + a_2 x^4, \ \ \dot{y} = D x y^2 + C x^{5} +  b_1 y^4 + b_2 x^3 y.
\end{equation}
The origin of \eqref{Ex-1-ext} is  monodromic when restricted to the parameter space $\hat{\Lambda} = \Lambda \cap \{ \Delta_1 < 0, \Delta_2 < 0 \}$ with $\Lambda = \{ (A, B,C,D) \in \mathbb{R}^4 : A C < 0, A (2 D-B) < 0, C (D-2 B) > 0 \}$, $\Delta_1 = (2 a_2-b_2)^2 + 4 C (2 B-D)$, and $\Delta_2 = (2 b_1-a_1 ))^2 + 4 A (2 D-B)$. Moreover its Poincar\'e map is  $\Pi(x) = x + o(x)$.
The second extension $\mathcal{X}^{[2]}$, is given by
\begin{equation}\label{Ex-1-ext-HOT}
\begin{array}{l}
\dot{x} = B x^2 y + A y^{5} + a_1 x y^3 + a_2 x^4 + \sum_{(i, j) \in S} a_{i,j-1} x^i y^{j-1},  \\
\dot{y} = D x y^2 + C x^{5} + b_1 y^4 + b_2 x^3 y + \sum_{(i, j) \in S} b_{i-1,j} x^{i-1} y^{j},
\end{array}
\end{equation}
where $S = \{ (i,j) \in \mathbb{N}^2 : 2i + j > 6, i + 2j > 6 \}$. The origin of \eqref{Ex-1-ext-HOT}, when restricted to $\hat{\Lambda}$, is a monodromic singularity with the same linear coefficient in  its Poincar\'e map, that is $\Pi(x) = x + o(x)$.
\end{proposition}

\subsection{An example within the framework of Theorems \ref{t:Teo-main2} and \ref{t:v1}}

We consider the minimal model \eqref{e:minimal} with $(\alpha, \beta, r, s) = (1, 1, 3, 2) \in \mathcal{E}^\dag$ and associated vector field $\mathcal{X}$, that is,
\begin{equation}\label{Ex-2}
\dot{x} = B x^2 y + A y^3, \\ \dot{y} = C x^5 + D x y^2,
\end{equation} restricted to the monodromic space $\Lambda^\dag = \Lambda = \{ (A, B,C,D) \in \mathbb{R}^4 : A C < 0, A (D - B) < 0, C  (D- 2 B) > 0 \}$ associated to the origin and having only one characteristic direction $\theta_*=0$. The weights are $W(\mathbf{N}(\mathcal{X})) = \{ (1,1), (1,2) \}$. This vector field satisfies, therefore, the assumptions of Theorem \ref{t:Teo-main2}.

\begin{proposition}\label{p:novapropTh2}
The first extension $\mathcal{X}^{[1]}$ of the minimal model  $\mathcal{X}$ given by \eqref{Ex-2}, with the assumptions of Theorem \ref{t:Teo-main2} is:
\begin{equation}\label{Toy1-XD}
\dot{x} = B x^2 y + A y^3 + a_1 x y^2 + a_2 x^4, \ \ \dot{y} = C x^5 + D x y^2 + b_1 y^3 + b_2 x^3 y.
\end{equation}
The origin of \eqref{Toy1-XD} is a monodromic singularity when we restrict it to the parameter space $\hat{\Lambda} = \Lambda \cap \{ \delta_1 < 0, \Delta_1 < 0 \}$ with $\Lambda = \{ (A, B,C,D) \in \mathbb{R}^4 : A C < 0, A (D - B) < 0, C  (D- 2 B) > 0 \}$ and $\delta_1 = (a_1 - b_1)^2 + 4 A (D -B)$, $\Delta_1 = (2 a_2-b_2)^2 + 4 C (2B-D)$.  The second extension $\mathcal{X}^{[2]}$ is
\begin{equation}\label{Toy1-HOT}
\begin{array}{l}
\dot{x} = B x^2 y + A y^3 + a_1 x y^2 + a_2 x^4 + \sum_{(i, j) \in S} a_{i,j-1} x^i y^{j-1},  \\
\dot{y} = C x^5 + D x y^2 + b_1 y^3 + b_2 x^3 y + \sum_{(i, j) \in S} b_{i-1,j} x^{i-1} y^{j},
\end{array}
\end{equation}
where $S = \{ (i,j) \in \mathbb{N}^2 : i+j > 4, i + 2 j > 6 \}$. Then the origin of \eqref{Toy1-HOT} restricted to $\hat{\Lambda}$ is a monodromic singularity.

Moreover if we fix $D-B>0$, to ensure that the orbits turn around in counterclockwise direction, the Poincar\'e map of both extensions $\mathcal{X}^{[1]}$ and  $\mathcal{X}^{[2]}$ is $\Pi(x) = \eta x + o(x)$,
with $$
\eta=\exp\left\{
2 \, \mathrm{PV} \int_{0}^{2\pi}
\frac{(b_1 +   \cot\theta (A + D +  \cot\theta (a_1 +  B \cot\theta))) \sin^2\theta }{ D-A - B  + (A - B +  D) \cos(2 \theta) + (b_1-a_1 ) \sin^2 \theta)} d\theta\right\}.
$$
\end{proposition}

\subsection{An example within the framework of Theorems \ref{t:Teo-main2-new} and \ref{t:v2}}

We consider the minimal model \eqref{e:minimal} with exponents $(\alpha, \beta, r, s) = (1,1, 4, 2) \in \mathcal{E}^\ddag$. So, the vector field $\mathcal{X}$ is given by
\begin{equation}\label{Toy-3}
\dot{x} = B x^2 y + A y^{3}, \ \ \dot{y} = D x y^2 + C x^{7},
\end{equation}
restricted to $\Lambda = \Lambda^\ddag = \{ A C < 0, A D < A B, C D > 3 B C \}$. This minimal family has weights $W(\mathbf{N}(\mathcal{X})) = \{(1, 1), (1, 3) \}$, and $\Omega_{11} = \{0\}$. To compute the coefficient $\eta$ of the linear part of the Poincar\'e map using  Theorem \ref{t:v2}, we must employ weighted polar coordinates $x=\rho^{p_2}\cos\theta=\rho\cos\theta$ and
$x=\rho^{q_2-p_2}\sin\theta=\rho^2\sin\theta$, in which new characteristic directions $\Omega_{p_2, q_2-p_2} = \{0, \pi/2\}$ of $\mathcal{B}$-type arise, both for $\mathcal{X}$, $\mathcal{X}^{[1]}$, and $\mathcal{X}^{[2]}$.

\begin{proposition}\label{p:propoexttipusb}
The first extension $\mathcal{X}^{[1]}$ of the minimal model  $\mathcal{X}$ given by \eqref{Toy-3}, with the assumptions of Theorem \ref{t:Teo-main2-new} is:
\begin{equation}\label{Toy-3-ext}
\dot{x} = B x^2 y + A y^{3} +a_1 x y^2 + a_2 x^5, \ \ \dot{y} = D x y^2 + C x^{7} + b_1 y^3 + b_2 x^4 y.
\end{equation}
The origin of \eqref{Toy-3-ext} becomes  monodromic if we restrict it to the parameter space $\hat{\Lambda} = \Lambda \cap \{ \Delta_1 < 0, \Delta_2 < 0 \}$ with $\Lambda = \{ (A, B,C,D) \in \mathbb{R}^4 : AC<0, (3B-D)(D-2B)<0, (B-D)(D-2B)<0 \}$, $\Delta_1 = (b_2-3 a_2)^2 + 4 C (3 B - D)$ and $\Delta_2 = (a_1 - b_1)^2 + 4 A (D-B)$. Moreover if we fix $D-2B>0$, to ensure that the orbits turn around in counterclockwise direction, the Poincar\'e map is $\Pi(\rho) = \eta \rho + o(\rho)$ with
\begin{align}\label{eta-Toy-3-ext}
\eta=&\exp\left\{ 2\,\frac{3B-D}{2B-D}\mathrm{PV}\int_{-\infty}^{\infty}\frac{Bv+a_2}{(D-3B)v^2+(b_2-3a_2)v+C}dv\right.\\
&\left.-2\,\frac{B-D}{2B-D}  \mathrm{PV}\int_{-\infty}^{\infty}\frac{Du+b_1}{(B-D)u^2+(a_1-b_1)u+A}du\right\}.\nonumber
\end{align}
The second extension $\mathcal{X}^{[2]}$ is
\begin{equation}\label{Toy-3-ext-HOT}
\begin{array}{l}
\dot{x} = B x^2 y + A y^{3} +a_1 x y^2 + a_2 x^5+ \sum_{(i, j) \in S} a_{i,j-1} x^i y^{j-1},  \\
\dot{y} = D x y^2 + C x^{7} + b_1 y^3 + b_2 x^4 y + \sum_{(i, j) \in S} b_{i-1,j} x^{i-1} y^{j},
\end{array}
\end{equation}
where $S = \{ (i,j) \in \mathbb{N}^2 : i + j > 4, i + 3 j  > 8\}$. The origin of \eqref{Toy-3-ext-HOT} restricted to $\hat{\Lambda}$ is a monodromic singularity with Poincar\'e map $\Pi(x) =\eta x + o(x)$ where $\eta$ is given in \eqref{eta-Toy-3-ext}.
\end{proposition}

\section{Some difficulties after the first extension}\label{s:literatura}

We consider the simplest family of Example B given in \cite{Ma}:
\begin{equation}\label{ejemplo-manyosa}
\dot{x} = x y^2 - y^3 + a x^5, \ \ \dot{y} = 2 x^7 - x^4 y + 4 x y^2 + y^3,
\end{equation}
that has a monodromic singularity at the origin when $\Delta(a) := 32 - (1+3 a)^2 >0$ and $W(\mathbf{N}(\mathcal{X})) = \{(1, 1), (1, 3) \}$. In \cite{Ma} it is proved that, when $\Delta(a)>0$, the coefficient of the linear part of the Poincar\'e map of the origin of \eqref{ejemplo-manyosa} is
\begin{equation}\label{eta1-manyosa}
\eta = \exp\left( \pi + \frac{4 \pi a}{\sqrt{\Delta(a)}} \right).
\end{equation}
In particular $\eta \neq 1$ if $a \neq -31/25$, the origin is an stable focus when $-(4 \sqrt{2}+1)/3 < a < -31/25$ while it is an unstable focus when $-31/25 < a < (4 \sqrt{2}-1)/3$. The nature of the singularity when $a = -31/25$ is not known because $\eta = 1$.
\medskip

In order to generalize the family \eqref{ejemplo-manyosa} we take $(A, B, C, D)=(-1, 0, 2,4)$ and  $(\alpha, \beta, r, s) = (1, 1, 4, 2) \in \mathcal{E}^\dag$ into the minimal model \eqref{e:minimal}. This choice of parameters and exponents lie in $\Lambda^\dag \cap  \mathcal{E}^\dag$ and \eqref{e:minimal} becomes the vector field $\mathcal{\mathcal{X}}$ given by
\begin{equation}\label{toy-manyosa-Famili-I}
\dot{x} = - y^3, \ \ \dot{y} = 2 x^7 + 4 x y^2,
\end{equation}
a part of \eqref{ejemplo-manyosa} with the same Newton diagram and the origin monodromic. The next proposition is proved in Section \ref{S-Proofs3}.

\begin{proposition}\label{Prop-manyosa-Famili-I-ext}
The first extension $\mathcal{X}^{[1]}$ of the minimal model  $\mathcal{X}$ given by  \eqref{toy-manyosa-Famili-I} is:
\begin{equation}\label{manyosa-ext}
\dot{x} = - y^3 + c_1 x y^2 + a x^5, \ \ \dot{y} = 2 x^7 + 4 x y^2 + d_1 y^3 + d_2 x^4 y.
\end{equation}
The origin of \eqref{manyosa-ext} is  monodromic  when we restrict it to the parameter space $\hat{\Lambda} = \{ \delta_1 < 0, \Delta_1 > 0 \}$ with $\delta_1 = -16 + (c_1 - d_1)^2$, $\Delta_1 = 32 - (d_2 -3 a)^2$. Moreover the Poincar\'e map is $\Pi(x) = \eta x + o(x)$ with

\begin{equation}\label{eta-Toy1-XD_ManyosaFamilyI}
\eta = \exp\left\{\frac{2(c_1+d_1) \pi}{\sqrt{-\delta_1}}+\frac{4a\pi}{\sqrt{\Delta_1(a,d_2)}}\right\}.
\end{equation}
The second extension $\mathcal{X}^{[2]}$
\begin{equation}\label{ManyosaFamilyI-HOT}
\begin{array}{l}
\dot{x} = - y^3 + c_1 x y^2 + a x^5 + \sum_{(i, j) \in S} a_{i,j-1} x^i y^{j-1}, \\
\dot{y} = 2 x^7 + 4 x y^2+ d_1 y^3 + d_2 x^4 y + \sum_{(i, j) \in S} b_{i-1,j} x^{i-1} y^{j},
\end{array}
\end{equation}
where $S = \{ (i,j) \in \mathbb{N}^2 : i + j  > 4, i + 3 j > 8\}$. Then the origin of \eqref{ManyosaFamilyI-HOT} restricted to $\hat{\Lambda}$ is a monodromic singularity with the same linear coefficient \eqref{eta-Toy1-XD_ManyosaFamilyI} of its Poincar\'e map.
\end{proposition}

\begin{remark}\label{Remar-Monster2}
{\rm Along the proof of Proposition \ref{Prop-manyosa-Famili-I-ext}, we observed that when extending the field $\mathcal{X}$ given in \eqref{toy-manyosa-Famili-I} to $\mathcal{X}^{[1]}$ given in \eqref{manyosa-ext}, new characteristic directions may appear while the origin of both $\mathcal{X}$ and $\mathcal{X}^{[1]}$ remains monodromic. More specifically there is only one characteristic direction $\theta_1^* = 0$ when $\delta_1 < 0$ but two characteristic directions for the origin of $\mathcal{X}^{[1]}$ appear when $\delta_1=0$, namely $(\theta_1^*, \theta_2^*) = (0, \pm \arctan(2))$ according to whether $d_1 = \mp 4 + c_1$. We emphasize that this last parameter restriction is compatible with the monodromy of the origin. Since $p_1 = q_1$, the geometry of the desingularizing blow-ups of $\mathcal{X}$ and $\mathcal{X}^{[1]}$ coincide only when the extension is made preserving the number of characteristic directions, which is what states Proposition \ref{Prop-manyosa-Famili-I-ext} when $\delta_1 < 0$.}
\end{remark}

\subsection{The action of linear changes of variables}

Let $\mathcal{Z}$ be an analytic vector field with a monodromic singular point at the origin with Newton diagram $\mathbf{N}(\mathcal{Z})$ and a set of characteristic directions $\Omega \subset [0,2\pi)$. We consider the effect of a diffeomorphism $\Phi$ around the origin on the transformed vector field $\phi_* \mathcal{Z}$ with Newton diagram $\mathbf{N}(\phi_* \mathcal{Z})$ and characteristic directions $\Omega^* \subset [0,2\pi)$. Clearly, $\phi$ preserves certain invariants, such as the monodromic character at the origin and the equality $\#\Omega = \#\Omega^*$. However, in general, we have $\Omega \neq \Omega^*$ and $\mathbf{N}(\mathcal{Z}) \neq \mathbf{N}(\phi_* \mathcal{Z})$ because  $W(\mathbf{N}(\mathcal{Z})) \neq W(\mathbf{N}(\phi_* \mathcal{Z}))$ and even $\# W(\mathbf{N}(\mathcal{Z})) \neq \# W(\mathbf{N}(\phi_* \mathcal{Z}))$ may occur.
\medskip

We are going to see how linear changes of variables modify $\Omega$, $\mathbf{N}(\mathcal{X}^{[1]})$, and the monodromy of the associated minimal model $\mathcal{X}$. According to Remark \ref{Remar-Monster2} and Proposition \ref{Prop-manyosa-Famili-I-ext} we extend the minimal model $\mathcal{X}$ given in \eqref{toy-manyosa-Famili-I} to $\mathcal{X}^{[1]}$ given by \eqref{manyosa-ext} but imposing the condition $\delta_1 = 0$ so that one more characteristic direction appears in the first extension. More specifically we have $(\theta_1^*, \theta_2^*) = (0, \pm \arctan(2))$ when $d_1 = \mp 4 + c_1$.

Some computations in the specific case $d_1 = - 4 + c_1$ reveal that the origin of $\mathcal{X}^{[1]}$ is monodromic if and only if the parameters lie in $\hat{\Lambda}_1 =  \{ c_1=2, \Delta_1 > 0\}$ with $\Delta_1 = 32 - (d_2 -3 a)^2$ as in Proposition \ref{Prop-manyosa-Famili-I-ext}. Now we do a linear change of variables $(x,y) \mapsto L(x,y) = (-2 x + y, y)$ so that the transformed vector field $L_* \mathcal{X}^{[1]}$ has a singularity at the origin with new characteristic directions $(\theta_1^*, \hat{\theta}_2^*) = (0, \pi/2)$. Moreover, $W(\mathbf{N}(\mathcal{X}^{[1]})) = \{(1, 1), (1, 3) \}$ is transformed into $W(\mathbf{N}(L_*\mathcal{X}^{[1]})) = \{(2, 1), (1, 3) \}$.

Since $\# W(\mathbf{N}(L_*\mathcal{X}^{[1]})) = 2$, we can find the minimal model \eqref{e:minimal} associated to $L_* \mathcal{X}^{[1]}$ which turns out to be
\begin{equation}\label{toy-transformed}
\dot{x} = \frac{d_2-a}{16} y^5, \ \ \dot{y} = -\frac{1}{64} x^7 -2 x y^2.
\end{equation}
Using Proposition \ref{t:prop-mon-family} it follows that for this transformed minimal model \eqref{toy-transformed} the origin is monodromic if and only if the new condition $a < d_2$ holds. We emphasize that this parameter condition  must be added to guarantee that the desingularization process of $L_* \mathcal{X}^{[1]}$ used in Theorem \ref{t:Teo-main1} success, that is we need to work on $\hat{\Lambda}_1 \cap \{a < d_2\}$. We have checked that the desingularization of the origin of $L_*\mathcal{X}^{[1]})$ without the condition $a < d_2$ is quite involved.

\subsection{Two different desingularizations of the first extension}\label{ss:twodifferent}

We consider the minimal model $\mathcal{X}$ given in \eqref{toy-manyosa-Famili-I}. First we do the desingularization of the characteristic direction $\theta_*=0$ as in  Theorem \ref{t:Teo-main1}, that is, thinking it as a type $\mathcal{A}$ characteristic directions, by doing $(x, y) \mapsto (z_2, w)$ and next $(x,y) \mapsto (x, w_2)$ with
\begin{equation}\label{first-desingularization-toy-M}
z_2= x^{3}/y, \ \  w = y/x, \ \ w_2 = y/x^{3}.
\end{equation}
These changes desingularize the origin of both $\mathcal{X}$ and its first extension $\mathcal{X}^{[1]}$ given in \eqref{manyosa-ext} but it doesn't do it with the origin of $\mathcal{X}^{[2]}$ when we add the higher order terms defined in \eqref{ManyosaFamilyI-HOT}. The reason is that, although the first transformation in \eqref{first-desingularization-toy-M} gives us a regular part, the second one trying to get a hyperbolic saddle singularity at the origin brings $\mathcal{X}^{[2]}$ to a non-polynomial vector field. This phenomenon was shown in the proof of Proposition \ref{Prop-manyosa-Famili-I-ext}.

This problem is fixed when we perform different blow-ups. More specifically, the origin of \eqref{toy-manyosa-Famili-I} desingularizes using the $\mathcal{B}$-type $2$-tuples scheme, that is, we do first $(x,y) \mapsto (z_2, z)$ and $(x,y) \mapsto (x, w_2)$ with
\begin{equation}\label{second-desingularization-toy-M}
z_2 =x^{3}/y, \ \ z=y/x^{2}, \ \ w_2 = y/x^{3},
\end{equation}
for the characteristic direction $\theta_*=0$ and finally we perform $(x, y) \mapsto (w_1, z_1)$ and $(x,y) \mapsto (u, y)$ with
\begin{equation}\label{FIRST-desingularization-toy-M}
z_1 =y/x, \ \ w_1=x^2/y, \ \ u = x/y,
\end{equation}
for the characteristic direction $\theta_*=\pi/2$.

In summary we notice that the first extension \eqref{manyosa-ext} of the minimal model \eqref{toy-manyosa-Famili-I} can be desingularized in two essentially different ways. At this point we recall that the desingularization process \cite{Dum} obtained composing the elemental blow-ups $\phi_1(x,y) = (x, y/x)$ and $\phi_2(x,y) = (x/y, y)$ is unique and it corresponds to the scheme \eqref{second-desingularization-toy-M} since the first change in \eqref{second-desingularization-toy-M} is written as the composition $(x, w_2) = \phi_1 \circ \phi_1 \circ \phi_1(x,y)$ while the second change as  $(z_2, z) = \phi_2 \circ \phi_1 \circ \phi_1(x,y)$. Of course the blow ups in \eqref{FIRST-desingularization-toy-M} are also composition of elemental ones since $(w_1, z_1) = \phi_2 \circ \phi_1(x,y)$ and $(u, y) = \phi_2(x,y)$.

Notice that we fail looking for a similar composition scheme associated to the second transformation in \eqref{first-desingularization-toy-M} because the change $(x, y) \mapsto (w, z_2) \neq \phi_i \circ \phi_j \circ \phi_k(x,y)$ for any combination of $i, j, k \in \{1,2\}$. We also point out that the origin of the second extension $\mathcal{X}^{[2]}$ can be desingularized only using the blow-up \eqref{second-desingularization-toy-M} while the blow-up \eqref{first-desingularization-toy-M} fails.

This example shows that, although clearly any desingularization of $\mathcal{X}^{[2]}$ also desingularizes $\mathcal{X}^{[1]}$, the reciprocal is not true so that $\mathcal{X}^{[1]}$ and  $\mathcal{X}^{[2]}$ can have two different desingularizations. It is well-known, see \cite{Dum} for details, that any analytic monodromic singularity of $\mathcal{X}$ admits a finite desingularization but, as far as we know, there is no method to determine the value $k \in \mathbb{N}$ such that the finite jet $J^{k} \mathcal{X}$ of $\mathcal{X}$ at the origin has the same desingularization that $\mathcal{X}$ itself. In all the examples of this work we obtain that $k$ is given by the minimum $k \in \mathbb{N}$ such that $\mathcal{X}^{[1]} \subset J^{k} \mathcal{X}$.

\section{Some monodromic classifications without $\# W(\mathbf{N}(\mathcal{Z}))$ fixed} \label{S-ClasesMon}

\subsection{The monodromic class $S_{k \omega}$}\label{s:sk}

The monodromic class $S_{k \omega}$ was defined in \cite{Ga-Ma-Ma}. A real analytic planar vector field defined in a neighborhood of the origin  whose first non-zero terms have odd degree $k$, polar differential system
$\dot{\rho}=\sum_{i\geq 0}F_{k+i}(\theta)\rho^{i+1},$ $
\dot{\theta}=\sum_{i\geq 0}G_{k+i}(\theta)\rho^i;$
with $G_k \not\equiv 0$ and $G_k(\theta) \geq 0$ (without loss of generality) and for any characteristic direction $\theta_*$ the following conditions hold:
\begin{enumerate}
  \item[(a)] $F_k(\theta_*) = 0$;
  \item[(b)] $G_{k+1}(\theta_*) = 0$;
  \item[(c)] $[(G''_{k}-2 F'_{r}) G''_{k}](\theta_*) > 0$;
  \item[(d)] $[(G'_{k+1}- F_{k+1})^2 -2 (G''_{k} - 2 F'_{k}) G_{k+2}](\theta_*) < 0$.
\end{enumerate}
Next we see that the class $S_{3 \omega}$ can be embedded the framework of Theorems \ref{t:Teo-main1} and \ref{t:Teo-main2}. Indeed,
in Proposition 13 of \cite{GaGi3} it is proved that the monodromic class $S_{3 \omega}$ has either $W(\mathbf{N}(\mathcal{Z})) = \{ (2,1), (1,2) \}$ and $\eta = 1$ or $W(\mathbf{N}(\mathcal{Z})) = \{ (1,1), (1,2) \}$. It is a tedious but straightforward to check that each one of the above cases fall into the desingularization processes established in Theorems \ref{t:Teo-main1} and \ref{t:Teo-main2}, respectively (in other words, the characteristic directions are $\mathcal{A}$-type). Indeed, following \cite{GaGi3} it follows that there exists a linear change of coordinates such that the associated polar differential system has the angular component $\dot{\theta} = G_3(\theta) + O(\rho)$ with (excluding the trivial case $G_3(\theta) > 0$ in $[0,2\pi)$) only two different cases: (I) either $G_{3}(\theta) = \alpha \sin^2\theta \cos^2\theta$ with $\alpha > 0$ and $W(\mathbf{N}(\mathcal{Z})) = \{ (2,1), (1,2) \}$ or (II) $G_{3}(\theta) = \sin^2\theta S(\theta)$ where $S(\theta) > 0$ in $[0,2\pi)$ and $W(\mathbf{N}(\mathcal{Z})) = \{ (1,1), (1,2) \}$.

In the first case (I) the family $S_{3 \omega}$ has two characteristic directions $\theta_1^*=0$ and $\theta_2^*= \pi/2$ and the associated minimal model \eqref{e:minimal} has exponents $(\alpha, \beta, r, s) = (1, 1, 3, 3)$ so that the origin desingularizes using the blow-ups of Theorem \ref{t:Teo-main1}.

In the second case (II) the family $S_{3 \omega}$ has only the characteristic direction $\theta_1^*=0$, the associated minimal model \eqref{e:minimal} has exponents $(\alpha, \beta, r, s) = (1, 1, 3, 2)$, and the origin desingularizes with the blow-ups described in Theorem \ref{t:Teo-main2}. In summary, the $S_{3,\omega}$ class falls within the scope of this work and, therefore, can be studied with the tools developed here.

Now, we show that vector fields $\mathcal{Z} \in S_{k \omega}$ with $k > 3$ can have $\# W(\mathbf{N}(\mathcal{Z})) > 2$, so that it lies beyond the scope of this work. Indeed, as an example,  consider the differential system
$$
\dot{x} = -2 x^4 y + x^3 y^2 - 3 x^2 y^3 - y^7, \ \ \dot{y} = x^7 - x^3 y^2 + 2 x^2 y^3 - x y^4.
$$ It has a monodromic singular point at the origin, it belongs to the class $S_{5 \omega}$, and moreover $\# W(\mathbf{N}(\mathcal{Z})) = 3$.

\subsection{The $\mathcal{G}$-monodromic class}\label{s:G}

The work \cite{Ga-Li-Ma-Ma} defines the $\mathcal{G}$-monodromic class for vector fields $\mathcal{Z} = P(x,y) \partial_x + Q(x,y) \partial_y =  \mathcal{Z}_{m-1} + \mathcal{Z}_{M-1}$ constructed from the sum of two homegeneous vector fields of degrees $m$ and $M$ with $1 \leq m < M$.
Let $\mathcal{Z}_{m-1} = P_m(x,y) \partial_x + Q_m(x,y) \partial_y$ and $\theta_j$ with $j = 1, \ldots, k$ the characteristic directions associated to the origin of $\mathcal{Z}$. We set $a_j = \cos\theta_j$, $b_j = \sin\theta_j$ and taking $z_j(z) = (a_j - b_j z, b_j+ a_j z)$ we define the polynomials
$$
P_m^j(z) = a_j P_m(z_j(z)) + b_j Q_m(z_j(z)), \ \ Q_m^j(z) = -b_j P_m(z_j(z)) + a_j Q_m(z_j(z)).
$$
from where we compute the derivatives $\alpha_j = (P_m^j)'(0)$ and $\beta_j = (Q_m^j)''(0)/2$.

Then $\mathcal{Z} \in \mathcal{G}$ if either the origin has no characteristic directions or $\alpha_j^2 + \beta_j^2 \neq 0$ for $j = 1, \ldots, k$. The vector field $\mathcal{Z}$ satisfies condition (a) if $x Q(x,y)-y P(x,y) \neq 0$ for all $(x,y)$ in a punctured neighborhood of the origin. Additionally, $\mathcal{Z}$ is said to satisfy condition (b) if either the origin has no characteristic directions or $P_m(\cos\varphi_j^*, \sin\varphi_j^*) = Q_m(\cos\varphi_j^*, \sin\varphi_j^*) = 0$ for all $j = 1, \ldots, k$. In \cite{Ga-Li-Ma-Ma} it is proved the following statements:
\begin{itemize}
  \item[(i)] If the origin is monodromic for $\mathcal{Z}$ then $m$ is odd, and conditions (a) and (b) hold. Moreover, if there are characteristic directions then $M$ is also odd.

  \item[(ii)] If $\mathcal{Z} \in \mathcal{G}$ then the origin is monodromic if and only if conditions (a) and (b) hold and $(2 + M - m) \alpha_j - 2 \beta_j \leq 0$ (resp. $\geq 0$) for all $j = 1, \ldots, k$ when the flow rotates counterclockwise (resp. clockwise)..
\end{itemize}
The origin of the vector field $\mathcal{Z}$ belongs to the $\mathcal{G}$-monodromic class when the conditions in (ii) are satisfied. We notice that, by construction, this monodromic class satisfies that $1 \leq \# W(\mathbf{N}(\mathcal{Z})) \leq 3$.

We stress that the $\mathcal{G}$-monodromic class with cubic lower order terms can be embedded into the scheme of Theorems \ref{t:Teo-main1} and \ref{t:Teo-main2}. The proof of this fact follows exactly the same lines than those shown in Section \ref{s:sk} for the class $S_{3 \omega}$. The connection is again Proposition 13 of \cite{GaGi3} where it is proved that the $\mathcal{G}$-monodromic class with cubic lower order terms shares some properties with the class $S_{3 \omega}$. In particular it shares all the features indicated  for the class $S_{3 \omega}$, namely either $W(\mathbf{N}(\mathcal{Z})) = \{ (2,1), (1,2) \}$ and $\eta = 1$ or $W(\mathbf{N}(\mathcal{Z})) = \{ (1,1), (1,2) \}$. Also shares the two different monodromic cases (I) and (II) so that we get an analogous proof of the fact that the origin of the $\mathcal{G}$-monodromic class with cubic lower order terms can be desingularized either with the processes established in Theorem \ref{t:Teo-main1} or in Theorem \ref{t:Teo-main2}.

However, there are examples of $\mathcal{G}$-monodromic singularities with $\# W(\mathbf{N}(\mathcal{Z})) > 2$ and, therefore, cannot be studied with the tools developed in this work. For example the quintic differential system
$$
\dot{x} = -x^4 y - 2 x^2 y^3 - y^7, \ \ \dot{y} = - x^3 y^2 - x y^4 + x^7,
$$ corresponds to a vector field $\mathcal{Z}$ in the $\mathcal{G}$-monodromic class with quintic lower order terms with $\# W(\mathbf{N}(\mathcal{Z})) = 3$.

\section{About a conjecture}\label{s:conjuecture}

In this section we discuss a conjecture formulated in \cite{GaGi3} and we relate it with the contents of this paper.  Following the notation in Section \ref{s:notation},  we consider an analytic family of vector fields $\mathcal{Z}$ given in \eqref{e:campZ} with a monodromic singularity at the origin and Newton diagram $\mathbf{N}(\mathcal{Z})$  with fixed weights $W(\mathbf{N}(\mathcal{Z})$, given in  \eqref{e:WNZ}. The vertex of $\mathbf{N}(\mathcal{Z})$ connecting the edges $i-1$ and $i$ with vector coefficient $(\mathfrak{a}_i, \mathfrak{b}_i)$ is called a {\it nondegenerate vertex} if $p_{i} \mathfrak{b}_i -q_{i} \mathfrak{a}_i \neq 0$ and $p_{i-1} \mathfrak{b}_i - q_{i-1} \mathfrak{a}_i \neq 0$.
\medskip

Given a weight $(p_i, q_i) \in W(\mathbf{N}(\mathcal{Z}))$ we may order the monomials of $\mathcal{Z}$ to write the $(p_i ,q_i)$-quasihomogeneous expansions $\mathcal{Z} = \mathcal{Z}_{r_i} + \cdots$, a sum of $(p_i ,q_i)$-quasihomogeneous vector fields of increasing weighted degrees where the leading vector field $\mathcal{Z}_{r_i} = P_{p_i+r_i}(x,y) \partial_x + Q_{q_i+r_i}(x,y) \partial_y \not\equiv 0$ is a $(p_i ,q_i)$-quasihomogeneous vector field of degree $r_i$ (the minimum degree in that expansion). After doing the weighted $(p_i, q_i)$-polar blow-up $(x,y) \mapsto (\rho, \varphi)$ given by
\begin{equation}\label{weighted-polar}
(x,y) = (\rho^{p_i} \cos\varphi, \rho^{q_i} \sin\varphi),
\end{equation}
and removing the maximum common power of $\rho$ times $\mathcal{D}_i(\varphi) = p_i \cos^2\varphi + q_i \sin^2\varphi > 0$ by a time-rescaling, $\mathcal{Z}$ is transformed into the polar vector field $\dot{\rho} = R_{i}(\varphi, \rho) = F_{r_i}(\varphi) \rho + O(\rho^2)$, $\dot{\varphi} = \Theta_{i}(\varphi, \rho) = G_{r_i}(\varphi) + O(\rho)$ for $i=1, \ldots, \ell$, where
\begin{eqnarray*} \label{def-GjFjD}
F_{r_i}(\varphi) &=& P_{p_i+r_i}(\cos\varphi, \sin\varphi) \cos\varphi + Q_{q_i+r_i}(\cos\varphi, \sin\varphi) \sin\varphi, \nonumber \\
G_{r_i}(\varphi) &=& p_i \, Q_{p_i+r_i}(\cos\varphi, \sin\varphi) \cos\varphi -q_i \, P_{q_i+r_i}(\cos\varphi, \sin\varphi) \sin\varphi.
\end{eqnarray*}
are trigonometric polynomials. Now we are in position to define the quantities
$$
\xi_{p_i q_i} = \mathrm{PV} \int_0^{2 \pi} \mathcal{F}_{p_i q_i}(\varphi) \, d \varphi, \ \ i=1, \ldots, \ell,
$$
with $\mathcal{F}_{p_i q_i}(\varphi) := F_{r_i}(\varphi)/G_{r_i}(\varphi)$, in case they exits. Under some additional conditions, the linear part of the Poincar\'e map $\Pi(x) = \eta x + o(x)$ associated to the origin of $\mathcal{Z}$ has leading coefficient
\begin{equation}\label{main-formula-eta}
\eta = \exp\left( \pm \sum_{i=1}^\ell \lambda_i \, \xi_{p_i q_i} \right),
\end{equation}
with
$$
\lambda_1 = 1, \ \ \ \ \lambda_i = \frac{p_{i} \mathfrak{b}_i -q_{i} \mathfrak{a}_i}{p_{i-1} \mathfrak{b}_i-q_{i-1} \mathfrak{a}_i}, \ \ i=2, \ldots, \ell.
$$
In the formula \eqref{main-formula-eta} the positive sign is taken when the flow rotates counterclockwise, in other words when $G_{r_i}(\varphi) \geq 0$ and the negative sign otherwise.

In \cite{GaGi3} it is formulated the following conjecture and it was checked to be true for a wide family of monodromic singularities.

\begin{conjecture}\label{Conject-xi}
{\rm  Given any analytic vector field $\mathcal{Z}$ with a monodromic singular point, there are analytic coordinates such that the leading coefficient $\eta$ of the asymptotic Dulac expansion $\Pi(x) = \eta x + o(x)$ of the Poincar\'e map $\Pi$ has the form \eqref{main-formula-eta} provided $\xi_{p_i q_i}$ exists for all $(p_i, q_i) \in \mathbf{N}(\mathcal{Z})$ and all the vertices of $\mathbf{N}(\mathcal{Z})$ are nondegenerate. }
\end{conjecture}

Let $(p, q) \in W(\mathbf{N}(\mathcal{Z}))$ and consider the integral $\xi_{pq}$, if it exists. We assume that the monodromic polycycle associated to the origin of $\mathcal{Z}$ is hyperbolic and that it has a characteristic direction such that, in the desingularization process, the blow-up associated to the regular part of the flow around a hyperbolic saddle is given by either $(x,y) \mapsto (x, w)$ with $w= y^{m}/x^{n}$ or $(x, y) \mapsto (z, y)$ with $z = x^{m}/y^{n}$ for some positive integers $m$ and $n$. Associated to each blow-up, and after a convenient time-rescaling removing the common factor, $\mathcal{Z}$ is transformed into an analytic vector field $\dot{x} = X(x, w)$, $\dot{w} = W(x, w)$ or $\dot{z} = Z(z, y)$, $\dot{y} = Y(z, y)$ with $W(0, w)$ and $Z(z, 0)$ without real roots, respectively. Now we associate to each vector field the functions
$$
g_1(w) = \frac{\partial}{\partial x} \left. \left(\frac{X(x, w)}{W(x, w)}\right) \right|_{x=0}, \ \ g_2(z) = \frac{\partial}{\partial y} \left. \left(\frac{Y(z, y)}{Z(z, y)}\right) \right|_{y=0},
$$
and the quantities
$$
t_1 = \int_{-\infty}^{\infty} g_1(w) \, dw, \ \ t_2 = \int_{-\infty}^{\infty} g_2(z) \, dz.
$$
The key point is to relate $\xi_{pq}$ with either $t_1$ or $t_2$, respectively. To do it we consider the trigonometric transformations $\varphi \mapsto w$ or $\varphi \mapsto z$ associated to each blow-up given by
\begin{equation}\label{changes-w-z}
w = \sin^{m}\varphi/\cos^{n}\varphi, \ \ z = \cos^{m}\varphi/\sin^{n}\varphi.
\end{equation}

We define the function $\mathcal{F}_{pq}(\varphi) = F_{r}(\varphi)/G_{r}(\varphi)$ so that $\xi_{p q} = \mathrm{PV} \int_0^{2 \pi} \mathcal{F}_{pq}(\varphi) d \varphi$ and we define the functions $\hat{g}_j(\varphi)$ by equating the differential 1-forms $g_1(w) \, dw = \hat{g}_1(\varphi) \, d\varphi$ and $g_2(z) \, dz = \hat{g}_2(\varphi) \, d\varphi$, respectively.  We take the differences
$$
h_1(\varphi) = \mathcal{F}_{pq}(\varphi) - \hat{g}_1(\varphi), \ \ h_2(\varphi) = \mathcal{F}_{pq}(\varphi) - \hat{g}_2(\varphi),
$$
respectively. We emphasize that the maps \eqref{changes-w-z} are not diffeomorphisms on $\mathbb{S}^1$ and when we apply it to some integral over the interval $[0, 2 \pi]$ in the variable $\varphi$ like $\xi_{pq}$ then the outcome can be a linear combination of integrals of either $g_1(w)$ or $g_2(z)$ over unbounded intervals in the variables $w$ or $z$, respectively.

If the additional condition
$$
\mathrm{PV} \int_0^{2 \pi} h_j(\varphi) d \varphi = 0
$$
holds then $\xi_{pq} = k_{pq} \, t_j$ for $j=1$ or $j=2$, respectively, for some non-negative integer $k_{pq}$. Actually, if $m$ or $n$ are even then $k_{pq}=0$ whereas $k_{pq}=2$ otherwise. This is in good agreement with the results stated in Theorems \ref{t:v1} and \ref{t:v2}.

\section{Proofs on the minimal model} \label{S-Proofs1}

\subsection{Proof of Theorem \ref{t:prop-mon-family}}

\begin{proof}
We have obtained the sets $\mathcal{E}$ and $\Lambda$ computing the necessary and sufficient monodromic conditions for the origin of \eqref{e:minimal} according to the algorithm developed in \cite{AGR, AGR2} based in the Newton diagram $\mathbf{N}(\mathcal{X})$ of \eqref{e:minimal}. More specifically, $\mathbf{N}(\mathcal{X})$ has two exterior vertices $(0, 2s)$ and $(2 r, 0)$ and one interior vertex $(2 \alpha, 2 \beta)$ all of them with even coordinates. The first condition in $\Lambda$ appears by imposing the necessary monodromic condition $A C < 0$ where $(A, 0)$ and $(0, C)$ are the vector coefficients of the exterior vertices of $\mathbf{N}(\mathcal{X})$. Up to know conditions (i) and (ii) of Theorem 3 in \cite{AGR} hold.

Notice that from the conditions in $\Lambda$ we get that  $B^2+D^2 \neq 0$, so that the coefficient vector $(B, D)$ of the interior vertex never vanishes. In particular $\mathbf{N}(\mathcal{X})$ contains the three former vertices and it is made by two edges whose slopes have associated weights $W(\mathbf{N}(\mathcal{X})) = \{(s-\beta, \alpha), (\beta, r-\alpha) \} \subset \mathbb{N}^2$. It is also worth to emphasize that the vector field $\mathcal{X}$ associated to system  \eqref{e:minimal} satisfies $\mathcal{X} = \mathcal{X}_\Delta$.

\medskip

Let $(p_1, q_1) = (s-\beta, \alpha)$ and $(p_2, q_2) = (\beta, r-\alpha)$ be the weights associated to the upper and lower edge of $\mathbf{N}(\mathcal{X})$, respectively. Then, the last condition in $\Lambda$ corresponds to the necessary convex hull condition $p_1/q_1 > p_2/q_2$ of $\mathbf{N}(\mathcal{X})$.

Consider the  $(p_i, q_i)$-quasihomogeneous expansion of $\mathcal{X}$ is $\mathcal{X} = \mathcal{X}_{r_i} + \cdots$ with leading $(p_i, q_i)$-quasihomogeneous vector field $\mathcal{X}_{r_i}$ of  weighted degree $r_i \in \mathbb{N}$. Then, we apply the conservative-dissipative decomposition $\mathcal{X}_{r_i} = \mathcal{X}_{H_i} + \mu_i \, \mathcal{D}_i$, being $\mathcal{X}_{H_i} = - \partial_y H_i(x,y) \partial_x + \partial_x H_i(x,y) \partial_y$ a Hamiltonian vector field and $\mathcal{D}_i = p_i x \partial_x + q_i y \partial_y$ the Euler field. After some computations we get
$$
\mathcal{X}_{r_1} = (A y^{2s-1} + B x^{2 \alpha} y^{2 \beta-1}) \partial_x + D x^{2 \alpha-1} y^{2 \beta} \partial_y
$$
with weighted degree $r_1=  s (2 \alpha-1) + \beta- \alpha \in \mathbb{N}$ where
$$
\mu_{r_1}(x,y) = \frac{B \alpha + D \beta}{s \alpha} x^{2 \alpha-1} y^{2 \beta -1}
$$
and Hamiltonian
$$
H_1(x,y) = \frac{1}{2 s \alpha} \left( (D s - B \alpha - D \beta) x^{2 \alpha} y^{2 \beta} - A \alpha y^{2 s} \right),
$$
whereas
$$
\mathcal{X}_{r_2} = B x^{2 \alpha} y^{2 \beta-1} \partial_x + (C x^{2r-1} + D x^{2 \alpha-1} y^{2 \beta}) \partial_y
$$
with weighted degree $r_2 = \alpha - \beta + r (2 \beta -1) \in \mathbb{N}$ where
$$
\mu_{r_2}(x,y) = \frac{B \alpha + D \beta}{r \beta} x^{2 \alpha-1} y^{2 \beta -1}
$$
and Hamiltonian
$$
H_2(x,y) = \frac{-1}{2 r \beta} \left( (B(r-\alpha - D \beta) x^{2 \alpha} y^{2 \beta} -  C \beta x^{2 r} \right).
$$
Now we are in position to compute the constant $\beta_V = -A C /(4 r s)$ associated to the interior vertex as described in  \cite{AGR}. So the necessary monodromic condition $\beta_V > 0$ yields the first condition in $\Lambda$ again.

We, now, factorize in $\mathbb{R}[x,y]$ with prime factors the Hamiltonian as $H_1(x,y) = k_1 y^{2 \beta} (y^{p_1} - \lambda_1 x^{q_1}) (y^{p_1} - \lambda_2 x^{q_1})$ with constant $k_1 = -A/(2 s)$ and coefficients $\lambda_2 = - \lambda_1 = ((D s - B \alpha - D \beta)/ (A \alpha))^{1/2}$. The penultimate condition in $\Lambda$ assures that $\lambda_i$ are non-real so that $H_1$ has no strong factor.

Similarly, $H_2(x,y)$ factorizes into the form $H_2(x,y) = k_2 x^{2 \alpha} (y^{p_2} - \gamma_1 x^{q_2}) (y^{p_2} - \gamma_2 x^{q_2})$ with constant $k_2 = (B (\alpha - r) + D \beta)/(2 r \beta)$ and coefficients $\gamma_2 = -\gamma_1 = (-\beta C /(B (\alpha - r) + D \beta))^{1/2}$. So the last condition in $\Lambda$ implies that $\gamma_i$ are non-real so that $H_2$ has no strong factor.

In short we have proved that on $\mathcal{E} \cap \Lambda$ family \eqref{e:minimal} has a monodromic singularity at the origin by Theorem 3 in \cite{AGR}.
~\end{proof}

\subsection{Proof of Proposition \ref{p:prop-dir.car-family}}

\begin{proof}
We take polar coordinates $(\theta, \rho)$ and we obtain the polar vector field associated to family \eqref{e:minimal} given by
\begin{eqnarray*}
\dot{\rho} &=& A_1(\theta) \rho^{2r-1} + A_2(\theta) \rho^{2s-1} + A_3(\theta) \rho^{2(\alpha+\beta)-1}, \\
\dot{\theta}  &=& B_1(\theta) \rho^{2r-2} + B_2(\theta) \rho^{2s-2} + B_3(\theta) \rho^{2(\alpha+\beta)-2},
\end{eqnarray*}
where all the exponents are positive integers in $\mathcal{E} \cap \Lambda$, the functions $A_j(\theta)$ are certain trigonometric polynomials and
$$
B_1(\theta) = C \cos^{2 r}\theta, \ \  B_2(\theta) = -A \sin^{2 s}\theta, \ \  B_3(\theta) = (D-B) \sin^{2 \beta}\theta \cos^{2 \alpha}\theta.
$$
We divide both components of the polar vector field by the common factor $\rho^{\gamma}$ with $\gamma = \min\{2r-2, 2s-2, 2(\alpha + \beta)-2\}$. So we split the analysis in three subcases according with the minimum value of these exponents.

If $r < \min\{s, \alpha + \beta\}$ then $\gamma = 2r-2$ so that the rescaled polar vector field has the component $\dot{\theta} = B_1(\theta) + O(\rho)$.
Conditions $\mathcal{E} \cap \Lambda$ imply that $C \neq 0$,
hence there is just one characteristic direction given by $\theta= \pi/2$.

If $s < \min\{r, \alpha + \beta\}$ then $\gamma = 2s-2$ and the rescaled polar vector field has the component $\dot{\theta} = B_2(\theta) + O(\rho)$.
Conditions $\mathcal{E} \cap \Lambda$ imply that $A \neq 0$, so we get one characteristic direction given by $\theta= 0$.

If $\alpha + \beta < \min\{r, s\}$ then $\gamma = 2(\alpha + \beta)-2$ and the rescaled polar vector field has the component $\dot{\theta} = B_3(\theta) + O(\rho)$. If $D = B$ then  $B_3(\theta) \equiv 0$ and the origin is no longer monodromic. If $D \neq B$ then we get two characteristic direction given by $\theta = 0$ and $\theta = \pi/2$.
\end{proof}

\subsection{Proof of Proposition \ref{p:prop-dir.car-family-2}}

\begin{proof}
Take $(p,q)$-weighted polar coordinates with $(p,q)=(p_1, q_1)$, we obtain the polar vector field associated to family \eqref{e:minimal} given by \eqref{e:polarspolars} with
$$
\Theta(0, \theta) =\frac{\sin^{2 \beta}\theta \, P_1(\theta)}{(s -\beta) \cos^2 \theta + \alpha \sin^2 \theta}
$$
being $P_1(\theta) = -A \alpha \sin^{2 (s - \beta)} \theta  + (\alpha-B + D (s - \beta)\cos^{2 \alpha} \theta$ a trigonometric polynomial without real roots when we restrict the parameters to  $\mathcal{E} \cap \Lambda$
 given in Theorem \ref{t:prop-mon-family}, hence $\theta_* =0$.

Analogously, using $(p,q)$-weighted polar coordinates with $(p,q)=(p_2, q_2)$ the associated polar vector field to \eqref{e:minimal} is \eqref{e:polarspolars} where now
$$
\Theta(0, \theta) = \frac{\cos^{2 \alpha}\theta \, P_2(\theta)}{\beta \cos^2 \theta + (r-\alpha) \sin^2 \theta}
$$
with $P_2(\theta) = C \beta \cos^{2 (r - \alpha)} + ((B (\alpha -r) + D \beta) \sin^{2 \beta} \theta$ being a trigonometric polynomial without real roots on $\mathcal{E} \cap \Lambda$. Therefore we get $\theta_* = \pi/2$.
\end{proof}

\subsection{Proof of Theorem \ref{t:Teo-main1}}

\begin{proof}
We restrict the parameters and exponents of \eqref{e:minimal} to $\mathcal{E}^* \cap \Lambda^*$. Then the possibilities (i) and (ii) of Proposition \ref{p:prop-dir.car-family} are not allowed. Regarding condition (iii) of Proposition \ref{p:prop-dir.car-family} we check that when $\alpha + \beta < \min\{r, s\}$ all the conditions are compatible and give $D \neq B$ in that case. In summary, \eqref{e:minimal} restricted to $\mathcal{E}^* \cap \Lambda^*$ always possesses two characteristic directions $\{\theta_1^*, \theta_2^*\} = \{0, \pi/2 \}$.

Let $(p_1, q_1) = (s-\beta, \alpha)$ and $(p_2, q_2) = (\beta, r-\alpha)$ be the weights associated to the upper and lower edge of $\mathbf{N}(\mathcal{X})$, respectively. We also can check that $p_i = q_i$ for some $i=1, 2$, is not compatible on $\mathcal{E}^* \cap \Lambda^*$ and therefore $(1,1) \not\in W(\mathbf{N}(\mathcal{X}))$.

\medskip

Next, we see that the desingularization process of the characteristic directions consists in applying the following sequence of blow-ups.
\begin{itemize}
  \item[(i)] For the characteristic direction $0$: $(x, y) \mapsto (z_2,w)$ with $z_2=x^{q_2}/y^{p_2}$ and $w = y/x$; and $(x,y) \mapsto (x, w_2)$ with $w_2 = y^{p_2}/x^{q_2}$.
  \item[(ii)] For the characteristic direction $\pi/2$: $(x,y) \mapsto (z_1, y)$ with $z_1 = x^{q_1}/y^{p_1}$ and $(x, y) \mapsto (z, w_1)$ with $z = x/y$ and $w_1=y^{p_1}/x^{q_1}$.
\end{itemize}

Indeed, associated to the characteristic direction $\theta_1^* = 0$ ($y=0$), we first consider the blow-up $(x, y) \mapsto (z_2,w)$ with $z_2=x^{q_2}/y^{p_2}$ and  $w = y/x$. We remove the common factor $z_2^{\gamma_1} w^{\gamma_2}$ with $\gamma_1=-1+(-2 + r + \alpha + \beta)/$ $(r - \alpha - \beta)$ and $\gamma_2 ={(r - \alpha + \beta - 2 r \beta)/(-r + \alpha + \beta)}$. Observe that this step can only be performed if and only if $r\neq \alpha+\beta$, that is, if and only if $p_2\neq q_2$.
 and we obtain the differential system
\begin{eqnarray*}
\dot{z}_2 &=& z_2 \left[ B (r - \alpha)- D \beta  - \beta C z_2^2 +   A (r-\alpha) z_2^{e_1} w^{e_2}  \right], \\
\dot{w} &=& w \left[ D-B + C  z_2^2 - A z_2^{e_1} w^{e_2} \right],
\end{eqnarray*}
with exponents $e_1, e_2 \in \mathbb{N} \cup \{0\}$ defined in \eqref{exponents-ei}. Taking into account the conditions in $\Lambda^*$, we obtain that the origin is a hyperbolic saddle. Assuming, without loss of generality, a counterclockwise direction of rotation, then
$D-B<0$ and $B (r - \alpha)- D \beta>0$, and the hyperbolicity ratio is
$$
\lambda_1 = \frac{B-D}{B (r - \alpha) - D \beta} > 0,
$$
on $\Lambda^*$.

Finally, we perform the blow-up $(x,y) \mapsto (x, w_2)$ with $w_2 = y^{p_2}/x^{q_2}$. After a time-rescaling obtained by removing the common factor $w_2^{1-1/\beta} \, x^{(-r + \alpha - \beta + 2 r \beta)/\beta}$ (having non-negative exponents on $\mathcal{E}^*$), we obtain a differential system
\begin{eqnarray*}
\dot{x} &=& B x w_2 + A x^{e_7+1} w_2^{e_8},\\
\dot{w}_2 &=& \beta C -(B r - B \alpha - D \beta) w_2^2  -  A (r-\alpha) x^{e_7} w_2^{e_8+1},
\end{eqnarray*}
with exponents $e_7, e_8 \in \mathbb{N} \cup \{0\}$ defined in \eqref{exponents-ei}. Moreover $\dot{w}_2|_{x=0} = \beta C -(B r - B \alpha - D \beta) w_2^2$ has no real roots on $\Lambda^*$.
So, in this local chart of the blow-up, the flow is regular.

\medskip

Associated to the characteristic direction $\theta_2^* = \pi/2$ ($x=0$), we first perform the blow-up $(x, y) \mapsto (z, w_1)$ with $z = x/y$ and $w_1=y^{p_1}/x^{q_1}$. After a time-rescaling that removes the common factor  $w_1^{\gamma_1}z^{\gamma_2}$ with $\gamma_1=
(2 (-1 + \alpha$ $ + \beta))/(s - \alpha - \beta)$ and $\gamma_2={(s + \alpha - 2 s \alpha - \beta)/(-s + \alpha + \beta)}$ (this step can only be performed if $s\neq \alpha + \beta$ or, equivalently, if $p_1\neq q_1$), and a time-reversion to preserve the previously defined direction of rotation, and we obtain the differential system
\begin{eqnarray*}
\dot{z} &=& z \left[ (D-B) - A w_1^2 + C  w_1^{e_3}  z^{e_4}  \right],\\
\dot{w}_1 &=& w_1 \left[ (B \alpha - D (s - \beta)) +A \alpha w_1^2 - C(s-\beta) w_1^{e_3}  z^{e_4}   \right] ,
\end{eqnarray*}
with exponents $e_3, e_4 \in \mathbb{N} \cup \{0\}$ given in \eqref{exponents-ei}. We conclude that the origin of the last system is a hyperbolic saddle with hyperbolicity ratio
$$
\lambda_2 = \frac{B-D}{B \alpha - D (s - \beta)} > 0,
$$
on $\Lambda^*$.

Finally, we perform the blow-up $(x,y) \mapsto (z_1, y)$ with $z_1 = x^{q_1}/y^{p_1}$. After the time-rescaling obtained by removing the common factor $z_1^{(\alpha-1)/\alpha} \, y^{(-s - \alpha + 2 s \alpha + \beta)/\alpha}$ (again notice that the exponents are non-negative on $\mathcal{E}$) we obtain the differential system
\begin{eqnarray*}
\dot{z}_1 &=& \alpha A - (D s - B \alpha - D \beta) z_1^2 - C (s - \beta) z_1^{1+e_6} y^{e_5}, \\
\dot{y} &=& D y z_1 +  C y^{1+e_5} z_1^{e_6},
\end{eqnarray*}
with exponents $e_5, e_6 \in \mathbb{N} \cup \{0\}$, see its expressions in \eqref{exponents-ei}. Since $\dot{z}_1|_{y=0} = \alpha A - (D s - B \alpha - D \beta) z_1^2$ has no real roots on $\mathcal{E}^*$, in this local chart the flow is regular.
\end{proof}

\subsection{Proof of Theorem \ref{t:Teo-main2}}

\begin{proof}
We take $p_1=q_1$, or equivalently we take into account the parameters' restriction  \eqref{parametersp1q11}. Therefore we get $p_2 \neq q_2$ because $\# \mathbf{N}(\mathcal{X}) = 2$. The polar vector field, after removing the common factor $r^{(2 (s-1))}$ has the angular component $\dot{\theta} = \sin^{2(s-\alpha)}(\theta) ((D-B) \cos^{2 \alpha}(\theta) -  A \sin^{2\alpha}(\theta)) + O(\rho)$. Thus the only characteristic direction is $\theta_* = 0$ because $A (D-B) < 0$ on $\Lambda^\dag$.

Now, it is only necessary to follow the same steps as in  the proof of Theorem \ref{t:Teo-main1} associated to the characteristic direction $\theta_* = 0$, but just taking into account the condition \eqref{parametersp1q11}, which leads to the fact that the  desingularization process is $(x,y) \mapsto (x, w_2)$ with $w_2 = y^{p_2}/x^{q_2}$ and $(x, y) \mapsto (z_2,w)$ with $z_2=x^{q_2}/y^{p_2}$ and $w = y/x$.~\end{proof}

\subsection{Proof of Theorem \ref{t:Teo-main2-new}}

\begin{proof} The proof is analogous to the ones of  Theorems \ref{t:Teo-main1} and \ref{t:Teo-main2}.
Proceeding as in these proofs, and assuming that the parameters belong to $\mathcal{E}^\ddag \cap \Lambda^\ddag$, one readily obtains that the blow-ups described in Definition \ref{d:typeB} desingularize the characteristic directions $0$ and $\pi/2$. This procedure yields hyperbolic saddle points at the origin for the corresponding systems \eqref{e:cordA}, as well as regular flow boxes surrounding the solutions ${x=0}$ and ${y=0}$ for systems \eqref{e:cordB} and \eqref{e:cordC}, respectively.~
\end{proof}

\section{Proof of Theorems \ref{t:v1} and \ref{t:v2}} \label{S-Proofs2}

The following result characterizes the linear term of the transition maps for characteristic directions of type $\mathcal{A}$.

\begin{proposition}\label{p:propotocha}
Let the origin of an analytic differential system $\dot{x}=P(x,y)$, $\dot{y}=Q(x,y)$, be a monodromic point. Assume that the orbits rotate in a counterclockwise direction.
Let $\theta = 0,\pi$ be  $\mathcal{A}$-type characteristic directions with weights $(p,q)$. For a fixed $\ve\gtrsim 0$, consider the transition maps of the flow $\Delta_{0,\pi}^\ve:\{y=-\ve x\}\rightarrow \{y=\ve x\}$  for  $x\gtrsim 0$ and  $x\lesssim 0$, respectively. Then,
$$
\Delta_{0,\pi}^\ve(x,-\ve x)=(D_{0,\pi}^\ve x+o(x), \ve\,D_{0,\pi}^\ve x+o(x))
$$
where
\begin{enumerate}[(a)]
\item If $p$ is even, then $
D_{0,\pi}^\ve=\exp\left\{R(\ve)\right\}.$
\item If $p$ is odd and $q$ is even, then
$$
D_{0,\pi}^\ve=\exp\left\{R(\ve)+ \frac{\nu}{\lambda p} \mathrm{PV} \int_{-\infty}^{\infty}
\left.\frac{\partial}{\partial x} \left(\frac{X(x,v)}{V(x,v)}\right)\right|_{x=0} dv\right\},
$$  where $\nu=1$ when $\theta=0$ and $\nu=-1$ when $\theta=\pi$.
\item If both $p$ and $q$ are odd, then
$$
D_{0,\pi}^\ve=\exp\left\{R(\ve)+ \frac{1}{\lambda p} \mathrm{PV} \int_{-\infty}^{\infty}
\left.\frac{\partial}{\partial x} \left(\frac{X(x,v)}{V(x,v)}\right)\right|_{x=0} dv\right\}.
$$
\end{enumerate}
In all the expressions above,
\begin{equation}\label{e:hypratio1}
\lambda=-\left.\left(\frac{\partial W(z,w)/\partial w}{\partial Z(z,w)/\partial z}\right)\right|_{(z,w)=(0,0)}
\end{equation} is the hyperbolicity ratio of the saddles appearing in the associated with the first blow-up described in Definition \ref{d:typeA}, where $Z$ and $W$ denote the components of the differential system \eqref{e:cordA}; and
$X$ and $V$ are the components of the differential system \eqref{e:cordB} associated with the second blow-up described in Definition \ref{d:typeA}; and $R(\ve)$ is a continuous function satisfying $\lim\limits_{\ve \to 0^+} R(\ve) =~0$.
\end{proposition}

\begin{proof}
We consider the characteristic directions $\theta_* \in \{0, \pi\}$ and the systems \eqref{e:cordA} and \eqref{e:cordB} in Definition \ref{d:typeA}, obtained via the blow-ups
$$
(z, w)_A = \left(\frac{x^q}{y^p}, \frac{y}{x}\right)_A \quad \text{and} \quad (x, v)_B = \left(x, \frac{y^p}{x^q}\right)_B,
$$
where the subscripts $(\cdot)_A$ and $(\cdot)_B$ indicate the corresponding local coordinate charts.
We decompose the transition maps $\Delta_0^\ve:\{y=-\ve x\}\rightarrow \{y=\ve x\}$ for $x\gtrsim 0$; and $\Delta_\pi^\ve:\{y=-\ve x\}\rightarrow \{y=\ve x\}$ for $x\lesssim 0$, into the local transition maps
$$
\Delta_0^\ve=\sigma_2\circ\tau_1\circ\sigma_1\mbox{ and }
\Delta_\pi^\ve=\sigma_4\circ\tau_2\circ\sigma_3,
$$
see Figure~\ref{f:fig1}. The analysis in local coordinates depends on the parity of $p$ and $q$. For reasons of space, we only will provide a detailed proof for the case where both $p$ and $q$ are odd, and we will comment on the differences for the remaining cases at the end of the proof.

\begin{figure}[h]
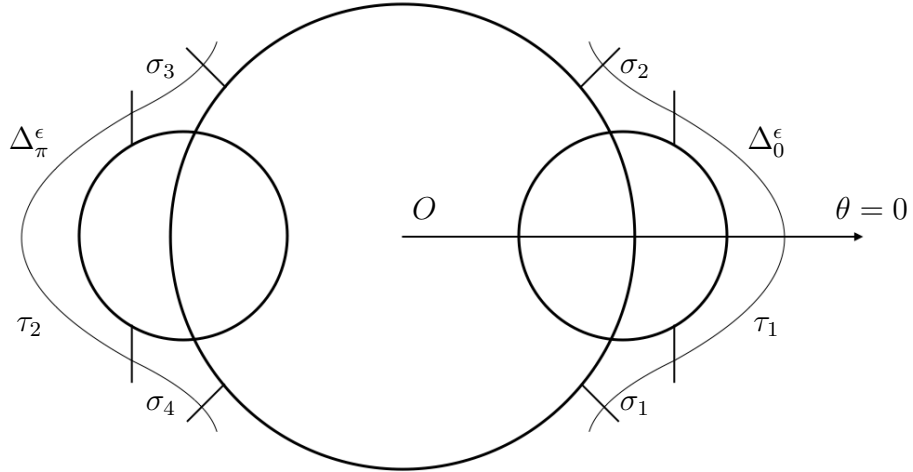


	\centering
	
	\begin{lpic}[l(2mm),r(2mm),t(2mm),b(2mm)]{typeA2(0.35)}
		
		\lbl[c]{350,125; $\theta=0$}
		
		\lbl[c]{310,150; $\Delta_{0}^\epsilon$}
		\lbl[c]{30,150; $\Delta_{\pi}^\epsilon$}

		\lbl[c]{310,80; $\tau_1$}
		\lbl[c]{30,80; $\tau_2$}		
		
	    \lbl[c]{260,50; $\sigma_1$}
		\lbl[c]{80,50; $\sigma_4$}	
		
	    \lbl[c]{260,178; $\sigma_2$}
		\lbl[c]{80,178; $\sigma_3$}
		
		\lbl[c]{180,125; $O$}
		
	\end{lpic}
	
	\caption{Transition maps scheme for  $\mathcal{A}$-type  characteristic directions.}\label{f:fig1}
\end{figure}

Assume that $p$ and $q$ are odd. The transversal sections $\{y = -\ve x\}$ and $\{y = \ve x\}$ correspond to $\{w = -\ve\}$ and $\{w = \ve\}$, respectively, in the local coordinates of system \eqref{e:cordA}. The transition maps $\sigma_1$ and $\sigma_3$ are determined by the flow of \eqref{e:cordA} from $\{w = -\ve\}$ to $\{z = -\delta\}$ where $\ve$ and $\delta$ are positive and small enough; while $\sigma_2$ and $\sigma_4$ arise from the reversed-time flow of \eqref{e:cordA} from $\{z = \delta\}$ to $\{w = \ve\}$. In the local coordinates of system \eqref{e:cordB}, the maps $\tau_1$ and $\tau_2$ are given by the flow of \eqref{e:cordB} from $\{v = -1/\delta\}$ to $\{v = 1/\delta\}$ for $x\gtrsim 0$ and $x\lesssim 0$ respectively; see Figure~\ref{f:fig2}.

\begin{figure}[h]
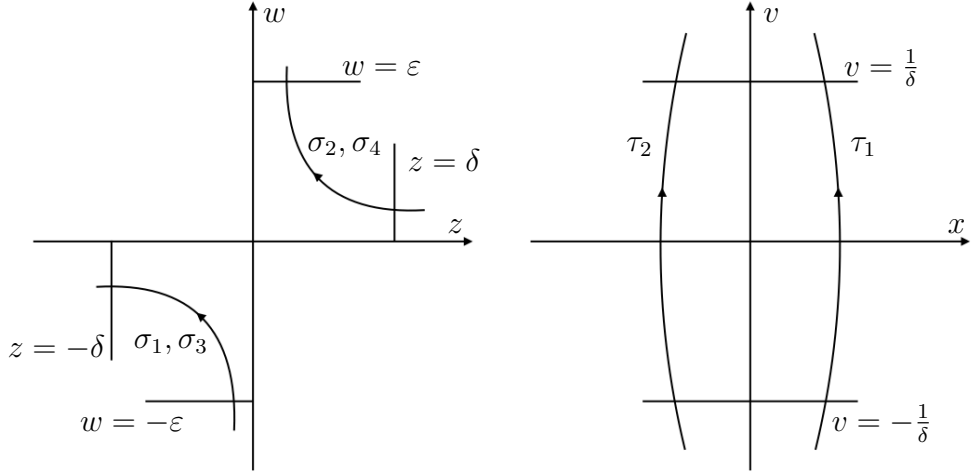

	
	\centering
	
	\begin{lpic}[l(2mm),r(2mm),t(2mm),b(2mm)]{cartalocal(0.35)}
		
\lbl[l]{105,195; $w$}
\lbl[l]{175,115; $z$}

\lbl[r]{75,40; $w=-\ve$}
\lbl[r]{45,69; $z=-\delta$}	

\lbl[l]{135,175; $w=\ve$}
\lbl[l]{160,140; $z=\delta$}	

\lbl[c]{136,145; $\sigma_{2},\sigma_4$}
\lbl[c]{70,70; $\sigma_{1},\sigma_3$}

\lbl[l]{295,195; $v$}
\lbl[l]{365,115; $x$}

\lbl[c]{340,175; $v=\frac{1}{\delta}$}
\lbl[c]{340,40; $v=-\frac{1}{\delta}$}

\lbl[c]{248,145; $\tau_2$}
\lbl[c]{333,145; $\tau_1$}

	\end{lpic}
	
	\caption{Transition maps in local coordinates for $\mathcal{A}$-type characteristic directions in the case $p$ and $q$ odd.}\label{f:fig2}
\end{figure}

Set $$F(\kappa)\,=\,\displaystyle{\int\limits_{0}^{\kappa}\left.\left(\frac{W(z,w)}{w\,
Z(z,w)}\right)\right|_{w=0}\,dz} \mbox{ and }
G(\kappa)\,=\,\displaystyle{\int\limits_{0}^{\kappa}\left.\left(\frac{Z(z,w)}{z\,
W(z,w)}\right)\right|_{z=0}\, dw}.$$

By using Lemma \ref{l:lema8} we obtain that for $i=1,3$:
$$
\sigma_{i}\left((z,-\ve)_A\right)=\left(-\delta,-A_i(\ve,\delta)|z|^\lambda+o(|z|^\lambda)\right)_A
$$ where
$$
A_i(\ve,\delta)=\exp\{F(-\delta)-\lambda G(-\ve)\},
$$ being $\lambda$ the hyperbolicity ratio of the saddle, noting that at this stage the stable and unstable manifolds are $W^s(O) = \{z = 0\}$ and $W^u(O) = \{w = 0\}$, respectively so that $\lambda$ is the one given in \eqref{e:hypratio1}. For, $i=2,4$:
$$
\sigma_{i}\left((\delta,w)_A\right)=\left(A_i(\ve,\delta)w^\frac{1}{\lambda}+o(w^\frac{1}{\lambda}),\ve\right)_A
$$ where
$$
A_i(\ve,\delta)=\exp\left\{G(\ve)-\frac{1}{\lambda}F(\delta)\right\}.
$$
Hence, for $i=1,3$ we obtain
\begin{multline}\label{e:mercuri}
A_{i+1}(\ve,\delta)A_i^{\frac{1}{\lambda}}(\ve,\delta)= \\
=\exp\left\{
\displaystyle{\int\limits_{-\ve}^{\ve}\left.\left(\frac{Z(z,w)}{z
W(z,w)}\right)\right|_{z=0}dw
-\frac{1}{\lambda}\int\limits_{-\delta}^{\delta}\left.\left(\frac{W(z,w)}{w
Z(z,w)}\right)\right|_{w=0}dz}\right\}.
\end{multline}
The regular maps $\tau_1$ and $\tau_2$ are given by:
$$
\tau_i\left(\left(x,-\frac{1}{\delta}\right)_B\right)=
\left(t_i(\delta)\,x+o(x),\frac{1}{\delta}\right)_B
$$
where, by direct integration of the first variational equation of system \eqref{e:cordB},
\begin{equation}\label{e:venus}
t_i(\delta)=\exp\left\{\int_{-1/\delta}^{1/\delta}
\frac{\partial}{\partial x} \left.\left(\frac{X(x,v)}{V(x,v)}\right)\right|_{x=0} dv\right\}.
\end{equation}

In order to compute the transition map $\Delta_{0}^\ve$, set $i,j=1$; and set $i=3$ and $j=2$ to compute $\Delta_{\pi}^\ve$. Set $\nu\in\{1,-1\}$ according to $\theta\in\{0,\pi\}$. To simplify the notation we write $A_i=A_i(\ve,\delta)$ and $t_j=t_j(\delta)$. Also to simplify the notation, and with a slight abuse of notation, we will write only the leading terms of the successive images computed.  Then, the transition maps are:
$$
(x,-\ve x)=\left(-\ve^{-p}|x|^{q-p},-\ve\right)_A
\xrightarrow{\sigma_i}\left(-\delta,-A_i{\ve^{-\lambda p}}{|x|^{\lambda(q-p)}}  \right)_A
$$
$$=\left(\nu\delta^{\frac{1}{q-p}},A_i^{\frac{p}{q-p}}{\ve^{^{-\frac{\lambda p^2}{q-p}}}}{|x|^{\lambda p}} , -\frac{1}{\delta}\right)_B
\xrightarrow{\tau_j}\left(\nu t_j\delta^{\frac{1}{q-p}},A_i^{\frac{p}{q-p}}{\ve^{^{-\frac{\lambda p^2}{q-p}}}}{|x|^{\lambda p}} , \frac{1}{\delta}\right)_B
$$
$$
=\left(\delta, t_j^{\frac{q-p}{p}} A_i {\ve^{^{-\lambda p}}}{|x|^{\lambda (q-p)}} \right)_A\xrightarrow{\sigma_{i+1}}
\left( t_j^{\frac{q-p}{\lambda p}} A_{i+1}A_i^{\frac{1}{\lambda}}{\ve^{^{-p}}} {|x|^{ q-p}} ,\ve\right)_A
$$
$$
=\left( t_j^{\frac{1}{\lambda p}} \left(A_{i+1}A_i^{\frac{1}{\lambda}}\right)^{\frac{1}{q-p}} x,\ve
t_j^{\frac{1}{\lambda p}} \left(A_{i+1}A_i^{\frac{1}{\lambda}}\right)^{\frac{1}{q-p}} x\right).
$$
Hence
$$
D_{0,\pi}^\ve=t_j^{\frac{1}{\lambda p}}(\delta) \left(A_{i+1}(\ve,\delta)A_i^{\frac{1}{\lambda}}(\ve,\delta)\right)^{\frac{1}{q-p}}
$$
for $(i,j)=(1,1)$ and $(i,j)=(3,2)$ respectively. Since $D_{0,\pi}^\ve$ do not depend on $\delta$, and taking into account that the expressions in \eqref{e:mercuri} are continuous functions, we can take $\delta\to 0$. Taking into account \eqref{e:venus}, we obtain:
$$
D_{0,\pi}^{\ve}=\exp\left\{R(\ve)+\frac{1}{\lambda p}\mathrm{PV}\int_{-\infty}^{\infty}
\left.\frac{\partial}{\partial x} \left(\frac{X(x,v)}{V(x,v)}\right)\right|_{x=0} dv\right\}
$$
where
$$
R(\ve)=\frac{1}{q-p} \int\limits_{-\ve}^{\ve}\left.\left(\frac{Z(z,w)}{z
W(z,w)}\right)\right|_{z=0}dw,
$$ which proves statement (c).

The remaining cases, corresponding to statements (a) and (b), are obtained analogously. As we have seen in the preceding case and, in fact, it is  guaranteed by Lemma \ref{l:lema8}(d), the analysis of the terms associated with the hyperbolic saddles is irrelevant, since  $D_{0,\pi}^\ve=\exp\left\{R_1(\ve)+R_2(\delta)\right\}$ where $\lim\limits_{\kappa\to 0}R_i(\kappa)=0$. The difference in these cases is that the coefficients $t_i(\delta)$ of the maps $\tau_i$ vary with respect the one given in \eqref{e:venus}. To compute them, denote $Q_i$ for $i=1,\ldots,4$ the different quadrants in $(x, y)$-coordinates.

In the case where both $p$ and $q$ are even, the regions $Q_1$ and $Q_4$ in $(x, y)$-coordinates (where $\Delta_0^\ve$ is defined) map to the first quadrant in coordinates A, while $Q_2$ and $Q_3$ (where $\Delta_{\pi}^\ve$ is defined) map to the second quadrant. A straightforward calculation, using the first variational equations of system \eqref{e:cordA},
shows that:
\begin{equation}\label{e:ti}
t_i(\delta)=\exp\left\{\int_{1/\delta}^{0} + \int_{0}^{1/\delta}
\left.\frac{\partial}{\partial x} \left(\frac{X(x,v)}{V(x,v)}\right)\right|_{x=0} dv
\right\}=1.
\end{equation}

When $p$ is even and $q$ is odd, the regions $Q_1$ and $Q_4$ in $(x, y)$-coordinates are mapped to the first quadrant in coordinates A via the blow up, but $Q_2$ and $Q_3$ are mapped to the third quadrant. Hence, $t_1(\delta)$ is given by \eqref{e:ti} but
$$
t_2(\delta)=\exp\left\{\int_{-1/\delta}^{0} + \int_{0}^{-1/\delta}
\left.\frac{\partial}{\partial x} \left(\frac{X(x,v)}{V(x,v)}\right)\right|_{x=0} dv
\right\}=1,
$$which proves statement (a).

When $p$ is odd and $q$ is even, then the regions $Q_1$ and $Q_4$ in $(x, y)$-coordinates map to the first and fourth quadrant in coordinates A, hence we can compute $t_1(\delta)$ using \eqref{e:venus}. However,  $Q_2$ and $Q_3$ map to the second and fourth quadrant, respectively. This means that,
$$
t_2(\delta)=\exp\left\{\int_{1/\delta}^{-1/\delta}
\left.\frac{\partial}{\partial x} \left(\frac{X(x,v)}{V(x,v)}\right)\right|_{x=0} dv\right\}=t_1^{-1}(\delta),
$$which completes the proof of statement (b).
\end{proof}

A direct consequence of the above Proposition \ref{p:propotocha} is:

\begin{corollary}\label{c:producte} Under the assumptions of Proposition \ref{p:propotocha}, the following statements hold:
\begin{enumerate}[(a)]
\item If either $p$ or $q$ are even, then
$$\lim\limits_{\ve\to 0} D_\pi^\ve\,D_0^\ve=1.$$
\item  If $p$ and $q$ are odd, then
$$
\lim\limits_{\ve\to 0} D_\pi^\ve\,D_0^\ve=\exp\left\{ \frac{2}{\lambda p} \mathrm{PV} \int_{-\infty}^{\infty}
\left.\frac{\partial}{\partial x} \left(\frac{X(x,v)}{V(x,v)}\right)\right|_{x=0} dv\right\}
$$
\end{enumerate}
\end{corollary}

The linear terms of the transition maps for characteristic directions in a $\mathcal{B}$-type $2$-tuple are given in the following results:

\begin{proposition}\label{p:propotochb}
Let the origin of an analytic differential system $\dot{x}=P(x,y)$, $\dot{y}=Q(x,y)$, be a monodromic point. Assume that the orbits rotate in a counterclockwise direction.
Let  $\theta_* = 0,\pi$, belong  to a $2$-tuple of $\mathcal{B}$-type characteristic directions computed with $(p,q)$-weighted polar coordinates.  For a fixed $\ve\gtrsim 0$, consider the transition maps of the flow $\Delta_{0,\pi}^\ve:\{y=\mp\ve^\frac{1}{p} |x|^\frac{q-p}{p}\}\rightarrow \{y=\pm \ve^\frac{1}{p} |x|^\frac{q-p}{p}\}$ for $x\gtrsim 0$ and  $x\lesssim 0$, respectively. Then,
$$
\Delta_{0,\pi}^\ve(x,\mp \ve^\frac{1}{p} |x|^\frac{q-p}{p})=(D_{0,\pi}^\ve x+o(x),\pm  \ve^\frac{1}{p}\,\left(D_{0,\pi}^\ve  |x|\right)^\frac{q-p}{p}+o(|x|^\frac{q-p}{p}))
$$ respectively,
where
\begin{enumerate}[(a)]
\item If $p$ is even, then $
D_{0,\pi}^\ve=\exp\left\{R(\ve)\right\}.$
\item If $p$ is odd and $q$ is even, then
$$
D_{0,\pi}^\ve=\exp\left\{R(\ve)+ \frac{\nu}{\lambda_0 p} \mathrm{PV} \int_{-\infty}^{\infty}
\left.\frac{\partial}{\partial x} \left(\frac{X(x,v)}{V(x,v)}\right)\right|_{x=0} dv\right\},
$$where $\nu\in\{1,-1\}$ when $\theta\in\{0,\pi\}$, respectively.
\item If $p$ is odd and $q$ is odd, then
\begin{equation}\label{e:novadzero}
D_{0,\pi}^\ve=\exp\left\{R(\ve)+ \frac{1}{\lambda_0 p} \mathrm{PV} \int_{-\infty}^{\infty}
\left.\frac{\partial}{\partial x} \left(\frac{X(x,v)}{V(x,v)}\right)\right|_{x=0} dv\right\}.
\end{equation}
\end{enumerate}

In the expressions above, $\lambda_0$ has the same form as in Equation \eqref{e:hypratio1} and represents the hyperbolicity ratio of the saddle arising in the blow-up described in Definition \ref{d:typeB} for $\theta_*=0$. In this case, $Z$ and $W$, and additionally $X$ and $V$, denote the components of the differential systems \eqref{e:cordA} and \eqref{e:cordB}, respectively. Finally, $R(\ve)$ is a continuous function satisfying $\lim\limits_{\ve \to 0^+} R(\ve) =~0$.
\end{proposition}

\begin{proof}
We consider the systems \eqref{e:cordA} and \eqref{e:cordB}, obtained via the blow-ups
$
(z, w)_A = \left({x^q}/{y^p}, {y^p}/{x^{q-p}}\right)_A \quad \text{and} \quad (x, v)_B = \left(x, {y^p}/{x^q}\right)_B.$ Following the scheme of Figure~\ref{f:fig3},
the transition maps $\Delta_{0,\pi}^\ve:\{y=\mp \ve^\frac{1}{p} |x|^\frac{q-p}{p}\}\rightarrow \{y=\pm \ve^\frac{1}{p} |x|^\frac{q-p}{p}\}$ for $x\gtrsim 0$ and $x\lesssim 0$ can be written as:
$
\Delta_0^\ve=\sigma_2\circ\tau_1\circ\sigma_1$ and $
\Delta_\pi^\ve=\sigma_4\circ\tau_2\circ\sigma_3.
$

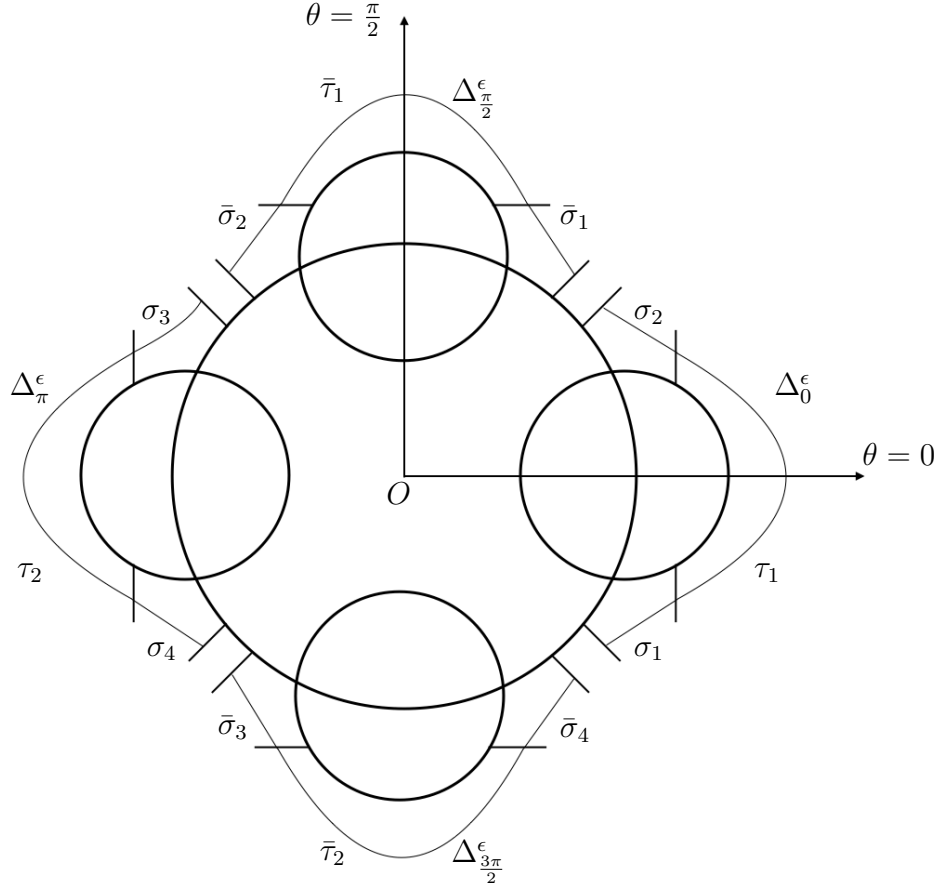
\begin{figure}[h]

	\centering
	
	\begin{lpic}[l(2mm),r(2mm),t(2mm),b(2mm)]{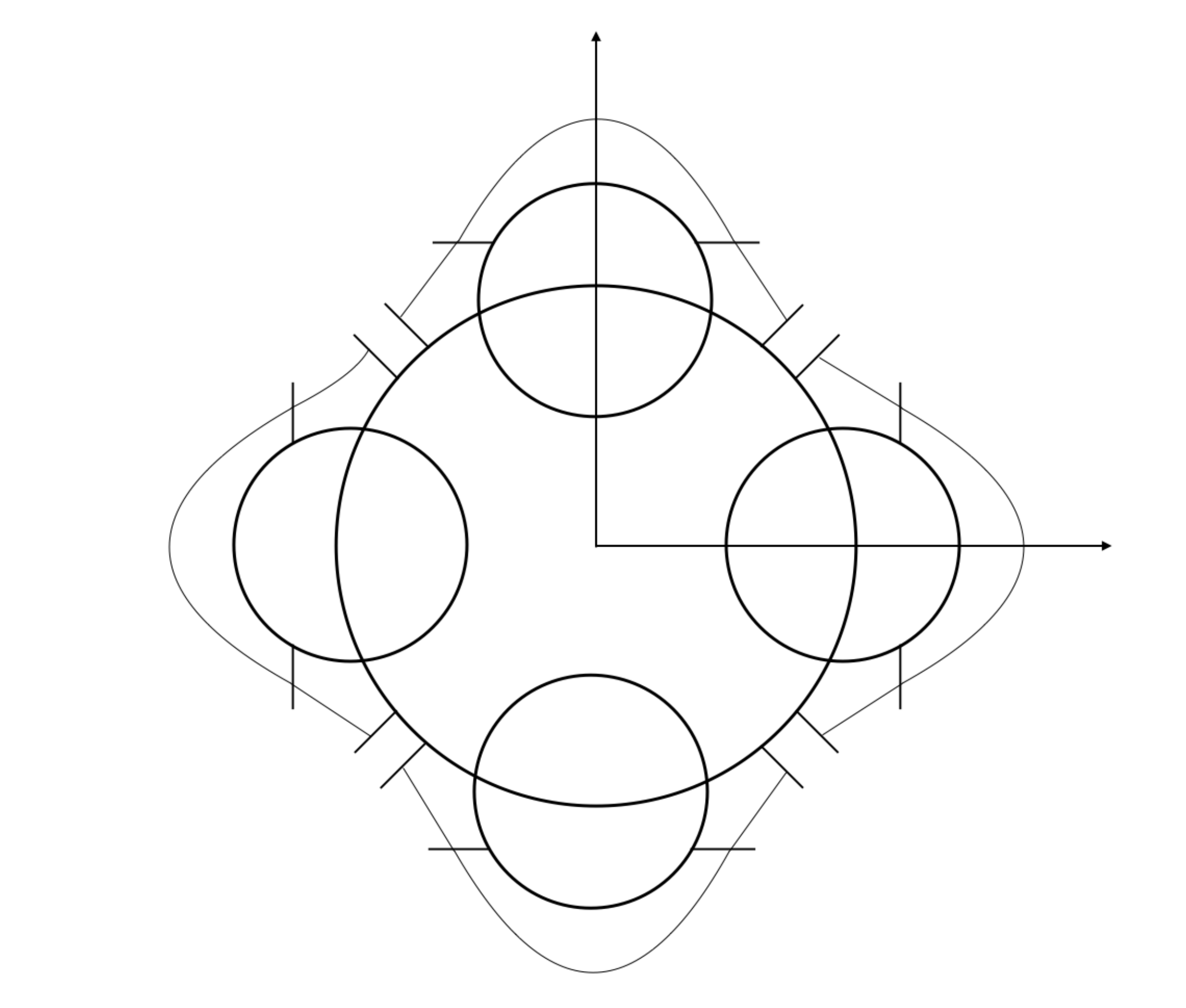(0.35)}
		
		\lbl[c]{390,165; $\theta=0$}
		
		\lbl[c]{350,190; $\Delta_{0}^\epsilon$}
		\lbl[c]{60,190; $\Delta_{\pi}^\epsilon$}

		\lbl[c]{340,120; $\tau_1$}
		\lbl[c]{60,120; $\tau_2$}		
		
	    \lbl[c]{295,90; $\sigma_1$}
		\lbl[c]{110,90; $\sigma_4$}	
		
	    \lbl[c]{295,218; $\sigma_2$}
		\lbl[c]{108,218; $\sigma_3$}

	\lbl[r]{194,330; $\theta=\frac{\pi}{2}$}
		
		\lbl[l]{220,300; $\Delta_{\frac{\pi}{2}}^\epsilon$}
		\lbl[l]{220,10; $\Delta_{\frac{3\pi}{2}}^\epsilon$}

		\lbl[l]{170,303; $\bar{\tau}_1$}
		\lbl[l]{170,13; $\bar{\tau}_2$}		
		
	    \lbl[c]{267,255; $\bar{\sigma}_1$}
		\lbl[c]{267,60; $\bar{\sigma}_4$}	
		
	    \lbl[c]{137,255; $\bar{\sigma}_2$}
		\lbl[c]{137,61; $\bar{\sigma}_3$}

		\lbl[c]{200,150; $O$}

	\end{lpic}
	
	\caption{Transition maps scheme for $\mathcal{B}$-type $2$-tuples of characteristic directions.}\label{f:fig3}
\end{figure}

The analysis in local coordinates depends on the parity of $p$ and $q$. We only provide a detailed proof for the case where both $p$ and $q$ are odd.  In this case, in the local coordinates of system \eqref{e:cordA}:
$
\sigma_1:\{w = -\ve\}\rightarrow \{z = -\delta\}$;  $
\sigma_2:\{z = \delta\}\rightarrow \{w = \ve\}$; $\sigma_3:\{w = \ve\}\rightarrow \{z = -\delta\}
$; $\sigma_4:\{z = \delta\}\rightarrow \{w = -\ve\}$.
where $\ve$ and $\delta$ are positive and small enough, see Figure~\ref{f:fig4}.

\begin{figure}[h]
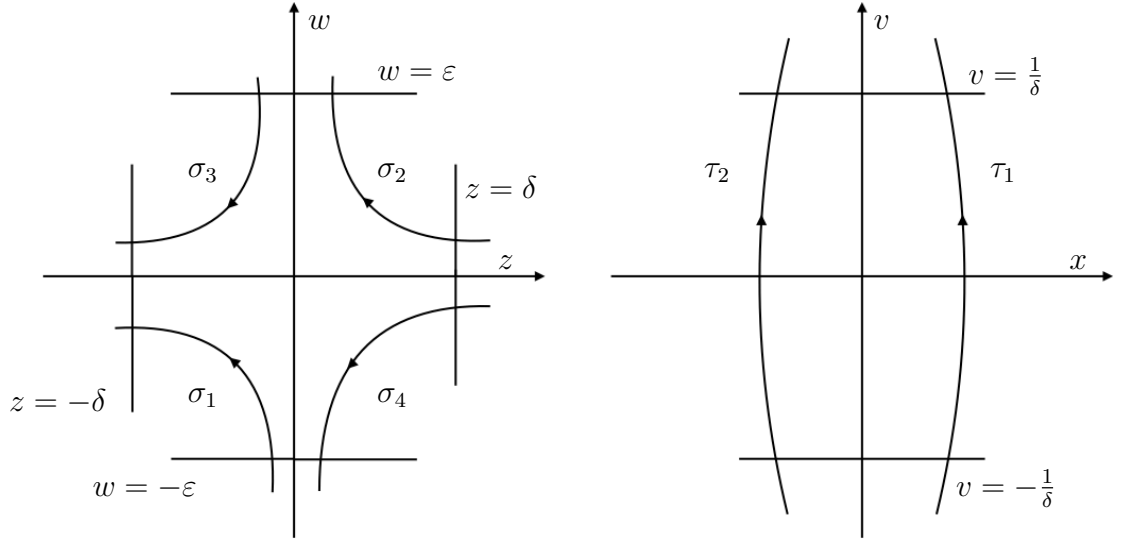

	
	\centering
	
\begin{lpic}[l(2mm),r(2mm),t(2mm),b(2mm)]{cartalocalb(0.40)}
		
\lbl[l]{112,195; $w$}
\lbl[l]{175,115; $z$}

\lbl[r]{75,40; $w=-\ve$}
\lbl[r]{45,69; $z=-\delta$}	

\lbl[l]{135,178; $w=\ve$}
\lbl[l]{164,140; $z=\delta$}	

\lbl[c]{140,145; $\sigma_{2}$}
\lbl[c]{77,70; $\sigma_{1}$}

\lbl[c]{140,70; $\sigma_{4}$}
\lbl[c]{77,145; $\sigma_{3}$}

\lbl[l]{300,195; $v$}
\lbl[l]{365,115; $x$}

\lbl[c]{344,177; $v=\frac{1}{\delta}$}
\lbl[c]{344,40; $v=-\frac{1}{\delta}$}

\lbl[c]{248,145; $\tau_2$}
\lbl[c]{343,145; $\tau_1$}

	\end{lpic}
	
	\caption{Transition maps in local coordinates for $\theta_*\in\{0,\pi\}$ belonging to a $\mathcal{B}$-type $2$-tuple of characteristic directions in the case $p$ and $q$ odd.}\label{f:fig4}
\end{figure}

Proceeding as in the proof of Proposition \ref{p:propotocha} we have that for $i=1,3$:
$$
\sigma_{i}\left((z,\mp\ve)_A\right)=\left(-\delta,-\nu A_i(\ve,\delta)|z|^\lambda+o(|z|^\lambda)\right)_A
$$ where
$\lambda$ is the hyperbolicity ratio of the saddle, and $\nu\in\{1,-1\}$ depending on whether $\theta_*$ is $0$ or $\pi$. Note
that  the stable and unstable manifolds are $W^s(O) = \{z = 0\}$ and $W^u(O) = \{w = 0\}$, respectively, hence the hyperbolicity ratio is given by the same expression as in \eqref{e:hypratio1}. For, $i=2,4$:
$$
\sigma_{i}\left((\delta,w)_A\right)=\left(A_i(\ve,\delta)|w|^\frac{1}{\lambda}+o(|w|^\frac{1}{\lambda}),\pm \ve\right)_A.
$$
In this case, $W^s(O) = \{w = 0\}$ and $W^u(O) = \{z = 0\}$.

Observe that, from Lemma \ref{l:lema8}(d) for $i=1,3$ we obtain $
A_{i+1}(\ve,\delta)A_i^{\frac{1}{\lambda}}(\ve,\delta)=\exp\left(R_1(\ve)+R_2(\delta)\right) $ where $\lim_{\kappa\to 0} R_i(\kappa)=0$.
The regular maps $\tau_1$ and $\tau_2$ are given by:
$
\tau_i\left(\left(x,-\frac{1}{\delta}\right)_B\right)=
\left(t_i(\delta)\,x+o(x),\frac{1}{\delta}\right)_B
$ with $
t_i(\delta)$ given by the same expression that in
\eqref{e:venus}.

To compute the transition map $\Delta_{0}^\ve$, we set $i,j=1$; and to compute $\Delta_{\pi}^\ve$ we set $i=3$ and $j=2$. Also we set $\nu\in\{1,-1\}$ according to $\theta\in\{0,\pi\}$. Again, with an abuse of notation, we will write only the leading terms of the successive images computed.  Then, the transition maps are:
$$
\left(x,-\nu \ve^\frac{1}{p} |x|^\frac{q-p}{p}\right)=\left(-\ve^{-1}|x|^{p},-\nu \ve\right)_A
\xrightarrow{\sigma_i}\left(-\delta,-\nu A_i {\ve^{-\lambda }} {|x|^{\lambda p}} \right)_A
$$
$$=\left(\nu\delta^{\frac{1}{p}}A_i^{\frac{1}{p}}{\ve^{^{-\frac{\lambda }{p}}}}{|x|^{\lambda }} , -\frac{1}{\delta}\right)_B
\xrightarrow{\tau_j}\left(\nu t_j\delta^{\frac{1}{p}}A_i^{\frac{1}{p}}{\ve^{^{-\frac{\lambda }{p}}}}{|x|^{\lambda }} , \frac{1}{\delta}\right)_B
$$
$$
=\left(\delta, \nu t_j^{p} A_i {\ve^{^{-\lambda }}}{|x|^{\lambda p}} \right)_A \xrightarrow{\sigma_{i+1}}
\left( t_j^{\frac{p}{\lambda }} A_{i+1}A_i^{\frac{1}{\lambda}} {\ve^{^{-1}}}{|x|^{ p}} ,\nu\ve\right)_A
$$
$$
=\left( t_j^{\frac{1}{\lambda }} \left(A_{i+1}A_i^{\frac{1}{\lambda}}\right)^{\frac{1}{p}} x,
\nu \ve^\frac{1}{p}\left(t_j^{\frac{1}{\lambda }} \left(A_{i+1}A_i^{\frac{1}{\lambda}}\right)^{\frac{1}{p}}\right)^\frac{q-p}{p} x^\frac{q-p}{p}\right),
$$
where we have used $\nu|x|=x$.
Hence
$
D_{0,\pi}^\ve=t_j^{\frac{1}{\lambda p}}(\delta) \left(A_{i+1}(\ve,\delta)A_i^{\frac{1}{\lambda}}(\ve,\delta)\right)^{\frac{1}{q-p}}
$
for $(i,j)=(1,1)$ and $(i,j)=(3,2)$ respectively.
Proceeding as in Proposition \ref{p:propotocha},  we obtain
that $D_{0,\pi}^{\ve}$ is given by \eqref{e:novadzero}, which proves statement (c).

The remaining cases, corresponding to statements (a) and (b), follow using the same  arguments as in the proof of Proposition \ref{p:propotocha}.
\end{proof}

\begin{proposition}\label{p:propotochb2}
Let the origin of an analytic differential system $\dot{x}=P(x,y)$, $\dot{y}=Q(x,y)$, be a monodromic point. Assume that the orbits rotate in a counterclockwise direction.
Let  $\theta_* = \pi/2, 3\pi/2$ belong  to a $2$-tuple of $\mathcal{B}$-type characteristic directions computed with $(p,q)$-weighted polar coordinates. For a fixed $\ve\gtrsim 0$, consider the transition maps of the flow $\Delta_{\frac{\pi}{2},\frac{3\pi}{2}}^\ve:\{y=\pm\ve^{-\frac{1}{p}} |x|^\frac{q-p}{p}\}\rightarrow \{y=\pm \ve^{-\frac{1}{p}} |x|^\frac{q-p}{p}\}$ for $x\gtrsim 0$ and  $x\lesssim 0$, respectively. Then,
$$
\Delta_{\frac{\pi}{2},\frac{3\pi}{2}}^\ve(x,\pm\ve^{-\frac{1}{p}} |x|^\frac{q-p}{p})=(-D_{\frac{\pi}{2},\frac{3\pi}{2}}^\ve x+o(x),\pm  \ve^{-\frac{1}{p}}\,\left(D_{\frac{\pi}{2},\frac{3\pi}{2}}^\ve  |x|\right)^\frac{q-p}{p}+o(|x|^\frac{q-p}{p}))
$$ respectively,
where
\begin{equation}\label{e:novadpimitjos}
D_{\frac{\pi}{2},\frac{3\pi}{2}}^\ve=\exp\left\{R(\ve)- \frac{p}{\lambda_{\frac{\pi}{2}} } \mathrm{PV} \int_{-\infty}^{\infty}
\left.\frac{\partial}{\partial y} \left(\frac{Y(u,y)}{U(u,y)}\right)\right|_{y=0} du\right\}.
\end{equation}
In the above expression, $\lambda_{\frac{\pi}{2}}$ has the form as in Equation \eqref{e:hypratio1} and represents the hyperbolicity ratio of the saddle arising in the blow-up described in Definition \ref{d:typeB} for $\theta_*=\pi/2$. In this case, $Z$ and $W$, and additionally $U$ and $Y$, denote the components of the differential systems \eqref{e:cordA} and \eqref{e:cordC}, respectively. Again, $R(\ve)$ is a continuous function satisfying $\lim\limits_{\ve \to 0^+} R(\ve) =~0$.
\end{proposition}

\begin{proof}
We consider the systems \eqref{e:cordA} and \eqref{e:cordC}, obtained via the blow-ups
$
(z, w)_A = \left(y/x, {x^{q-p}/y^p}\right)_A \quad \text{and} \quad (u,y)_C = \left(x/y,y\right)_C.$ Using the notation displayed in Figure   \ref{f:fig3},
the transition maps $\Delta_{\frac{\pi}{2},\frac{3\pi}{2}}^\ve:\{y=\pm \ve^{-\frac{1}{p}} |x|^\frac{q-p}{p}\}\rightarrow \{y=\pm \ve^{-\frac{1}{p}} |x|^\frac{q-p}{p}\}$ for $x\gtrsim 0$ and $x\lesssim 0$ respectively, decompose as:
$
\Delta_\frac{\pi}{2}^\ve=\bar{\sigma}_2\circ\bar{\tau}_1\circ\bar{\sigma}_1$ and $
\Delta_\frac{3\pi}{2}^\ve=\bar{\sigma}_4\circ\bar{\tau}_2\circ\bar{\sigma}_3.
$
As in the proof of Proposition \ref{p:propotochb}, the analysis of the transition maps $\bar{\sigma}_i$, in local coordinates, depends on the parity of $p$ and $q$. The main difference in this case is that, in local coordinates, the analysis of the maps  $\bar{\tau}_i$ do not depend on the parity of $p$ and $q$.

As in the proof of the previous results, we only provide a detailed proof for the case where both $p$ and $q$ are odd.  In this case, in the local coordinates of system \eqref{e:cordA}:
$
\bar{\sigma}_1:\{w = \ve\}\rightarrow \{z = \delta\}$;  $
\bar{\sigma}_2:\{z = -\delta\}\rightarrow \{w = \ve\}$; $\bar{\sigma}_3:\{w = -\ve\}\rightarrow \{z = \delta\}
$; $\bar{\sigma}_4:\{z = -\delta\}\rightarrow \{w = -\ve\}$, being $\ve$ and $\delta$ are positive and small enough, see Figure~\ref{f:fig5}. Again, we have that for $i=1,3$:
$
\sigma_{i}\left((z,\pm\ve)_A\right)=\left(\delta,\nu A_i(\ve,\delta)|z|^\lambda+o(|z|^\lambda)\right)_A
$ where
$\lambda$ is the hyperbolicity ratio of the saddle, and $\nu\in\{-1,1\}$ according to $\theta_*\in\{\frac{\pi}{2},\frac{3\pi}{2}\}$. The hyperbolicity ratio $\lambda_{\frac{\pi}{2},\frac{3\pi}{2}}$ has the same expression as in \eqref{e:hypratio1}. For, $i=2,4$:
$
\sigma_{i}\left((-\delta,w)_A\right)=(-A_i(\ve,\delta)|w|^\frac{1}{\lambda}+o(|w|^\frac{1}{\lambda}),\nu \ve)_A.
$

\begin{figure}[h]
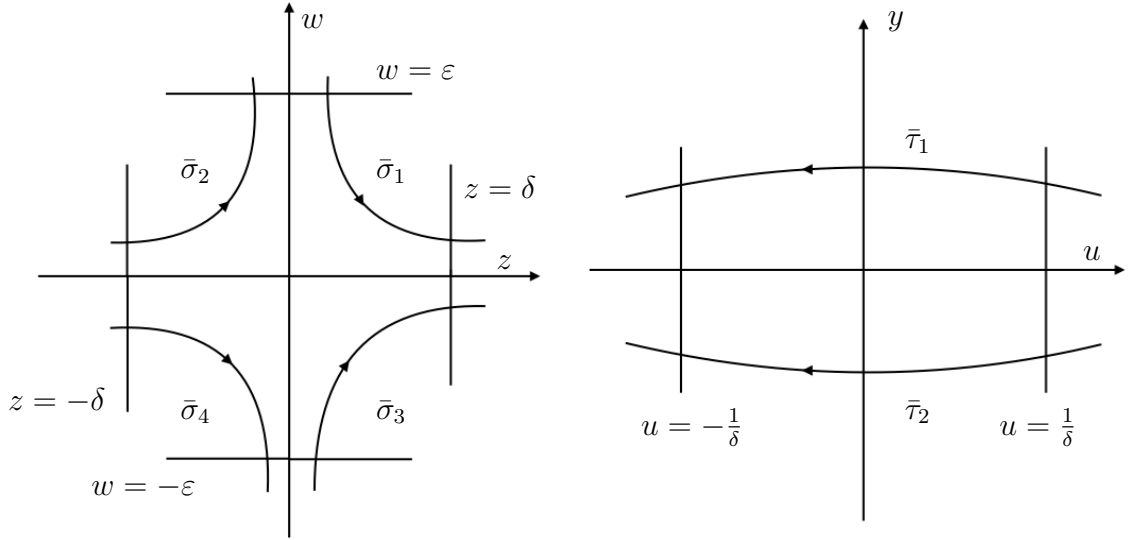

	
	\centering
	
	\begin{lpic}[l(2mm),r(2mm),t(2mm),b(2mm)]{locc(0.40)}
		
\lbl[l]{95,180; $w$}
\lbl[l]{160,100; $z$}

\lbl[r]{60,25; $w=-\ve$}
\lbl[r]{30,54; $z=-\delta$}	

\lbl[l]{120,163; $w=\ve$}
\lbl[l]{149,125; $z=\delta$}	

\lbl[c]{125,130; $\bar{\sigma}_{1}$}
\lbl[c]{60,50; $\bar{\sigma}_{4}$}

\lbl[c]{125,50; $\bar{\sigma}_{3}$}
\lbl[c]{60,130; $\bar{\sigma}_{2}$}

\lbl[l]{290,180; $y$}
\lbl[l]{355,103; $u$}

\lbl[c]{225,45; $u=-\frac{1}{\delta}$}
\lbl[c]{340,45; $u=\frac{1}{\delta}$}

\lbl[c]{300,140; $\bar{\tau}_1$}
\lbl[c]{300,50; $\bar{\tau}_2$}

	\end{lpic}
	
	\caption{Transition maps in local coordinates for $\theta_*\in\{\frac{\pi}{2},\frac{3\pi}{2}\}$ belonging to a $\mathcal{B}$-type $2$-tuple of characteristic directions in the case $p$ and $q$ odd. While the transition maps in the 
	$(z,w)$-coordinates associated with systems \eqref{e:cordA} depend on the parity of $p$ and~$q$, the transition maps in the $(u,y)$-coordinates associated with systems \eqref{e:cordC} are independent on the parity of $p$ and~$q$.}\label{f:fig5}
\end{figure}

The regular maps $\tau_1$ and $\tau_2$ are given by:
$
\tau_i\left(\left(x,\frac{1}{\delta}\right)_C\right)=
\left(t_i(\delta)\,x+o(x),-\frac{1}{\delta}\right)_C
$ with $
t_i(\delta)$ given by
\begin{equation}\label{e:novati}
t_i(\delta)=\exp\left\{ \int_{1/\delta}^{-1/\delta}
\left.\frac{\partial}{\partial y} \left(\frac{Y(u,y)}{U(u,y)}\right)\right|_{y=0} du\right\}.
\end{equation}

As in the proof of Proposition \ref{p:propotochb}, to compute $\Delta_{\frac{\pi}{2}}^\ve$, we set $i,j=1$; and to compute $\Delta_{\frac{3\pi}{2}}^\ve$, we set $i=3,j=2$. Also we set $\nu\in\{1,-1\}$ according to $\theta_*\in\{{\pi}/{2},{3\pi}/{2}\}$.   Then, the transition maps are:
$$
\left(x,\nu \ve^{-\frac{1}{p}} |x|^\frac{q-p}{p}\right)=\left(\ve^{-\frac{1}{p}}|x|^\frac{q-2p}{p},\nu \ve\right)_A
\xrightarrow{\bar{\sigma}_i}\left(\delta,\nu \ve^{-\frac{\lambda}{p}}A_i  {|x|^{\lambda\frac{q-2p}{p}}} \right)_A
$$
$$=\left(\frac{1}{\delta},\nu\delta^{\frac{q-p}{q-2p}}
\ve^{^{-\frac{\lambda }{(q-2p)p}}}
A_i^{\frac{1}{q-2p}}{{|x|^{\frac{\lambda}{p} }} }\right)_C
\xrightarrow{\bar{\tau}_j}\left(-\frac{1}{\delta},t_j\nu\delta^{\frac{q-p}{q-2p}}
\ve^{^{-\frac{\lambda }{(q-2p)p}}}
A_i^{\frac{1}{q-2p}}{{|x|^{\frac{\lambda}{p} }}} \right)_C
$$
$$
=\left(-\delta, \nu t_j^{q-2p} {\ve^{^{-\frac{\lambda}{p} }}}A_i {|x|^{\lambda \frac{q-2p}{p}}} \right)_A \xrightarrow{\bar{\sigma}_{i+1}}
\left( \nu t_j^{\frac{q-2p}{\lambda }}  {\ve^{^\frac{-1}{p}}}A_{i+1}A_i^{\frac{1}{\lambda}}{|x|^{ \frac{q-2p}{p}}} ,\nu\ve\right)_A
$$
$$
=\left(-t_j^{\frac{p}{\lambda }} \left(A_{i+1}A_i^{\frac{1}{\lambda}}\right)^\frac{p}{q-2p}x ,\nu \ve^{-\frac{1}{p}}t_j^{\frac{q-p}{\lambda }} \left(A_{i+1}A_i^{\frac{1}{\lambda}}\right)^\frac{q-p}{q-2p}|x|^{\frac{q-p}{p}}\right),
$$
where only the leading terms of the successive images are written. In summary,
$
D_{\frac{\pi}{2},\frac{3\pi}{2}}^\ve=t_j(\delta)^{\frac{p}{\lambda }} \left(A_{i+1}(\ve,\delta)A_i(\ve,\delta)^{\frac{1}{\lambda}}\right)^\frac{p}{q-2p}
$
for $(i,j)=(1,1)$ and $(i,j)=(3,2)$ respectively.
As in the previous proofs,  using  Lemma \ref{l:lema8}(d), we obtain that $D_{\frac{\pi}{2},\frac{3\pi}{2}}^\ve$ is given by Equation \eqref{e:novadpimitjos}. When either $p$ or $q$ are even we obtain the same result.
\end{proof}

\begin{corollary}\label{c:producteb}
Consider the assumptions of Proposition \ref{p:propotochb} and \ref{p:propotochb2}.
\begin{enumerate}[(a)]
\item If either $p$ or $q$ are even, then
$\lim\limits_{\ve\to 0} D_\pi^\ve\,D_0^\ve=1.$
If $p$ and $q$ are odd, then
$$
\lim\limits_{\ve\to 0} D_\pi^\ve\,D_0^\ve=\exp\left\{ \frac{2}{\lambda_0 p} \mathrm{PV} \int_{-\infty}^{\infty}
\left.\frac{\partial}{\partial x} \left(\frac{X(x,v)}{V(x,v)}\right)\right|_{x=0} dv\right\}.
$$
\item For every natural $p$ and $q$: $$
\lim\limits_{\ve\to 0} D_{\frac{3\pi}{2}}^\ve\,D_{\frac{\pi}{2}}^\ve=
\exp\left\{- \frac{2p}{\lambda_{\frac{\pi}{2}} } \mathrm{PV} \int_{-\infty}^{\infty}
\left.\frac{\partial}{\partial y} \left(\frac{Y(u,y)}{U(u,y)}\right)\right|_{y=0} du\right\}.
$$
\end{enumerate}
\end{corollary}

To prove the above results we will use the following straightforward generalization of Lemma 8 in \cite{Ga-Li-Ma-Ma} (we omit its proof):

\begin{lemma}\label{l:lema8}
Consider system
\begin{equation}\label{e:sistomega}
\begin{array}{ccl}
\dot{x} &=& -x(a+f(x,y))\,=\,P(x,y),\\
\dot{y} &=& y(b+g(x,y))\,=\,Q(x,y),
\end{array}
\end{equation}
where  $f$ and $g$ begin with first order terms and $a$ and $b$ are
positive.
\begin{enumerate}[(a)]
\item Let $s(y)$ be the transition map of the
flow from $\{x=\ve\}$ to $\{y=\delta\}$ being $\ve$ and $\delta$
positive and small enough,  then the leading term of the map is given by $$
s(y)=A(\ve,\delta)\, y^{\lambda}\, +\, o(y^{\lambda}),
\quad \mbox{with}\quad A(\ve,\delta)=
\frac{\exp\{F(\delta)\}}{\exp\{\lambda G(\ve)\}},$$ where $\lambda=a/b$ is the hyperbolicity ratio and
$$F(\delta)\,=\,\displaystyle{\int\limits_{0}^{\delta}\left.\left(\frac{P(x,y)}{x\,
Q(x,y)}\right)\right|_{x=0}\,dy}\mbox{ and }
G(\ve)\,=\,\displaystyle{\int\limits_{0}^{\ve}\left.\left(\frac{Q(x,y)}{y\,
P(x,y)}\right)\right|_{y=0}\, dx}.$$
\item[(b)] Let $\bar{s}$ be the transition map from $\{y=\delta\}$  to $\{x=\ve\}$ associated with system \eqref{e:sistomega} after a time reversion. Then
$$
\bar{s}(x)=\bar{A}(\ve,\delta)\, x^{\frac{1}{\lambda}}\, +\, o(x^\frac{1}{\lambda}),
\quad \mbox{with}\quad \bar{A}(\ve,\delta)=
\frac{\exp\{G(\ve)\}}{\exp\{\frac{1}{\lambda} F(\delta)\}},$$

\item[(c)] Let $A_{Q_i}(\ve,\delta)$ be the coefficient of the leading terms of the transition maps of the saddle \eqref{e:sistomega} associated with each quadrant $Q_i$. Then, the leading terms of the transition maps for $i = 2, 3, 4$, corresponding to the second, third, and fourth quadrants (that is, from $\{x = -\ve\}$ to $\{y = \delta\}$; from $\{x = -\ve\}$ to $\{y = -\delta\}$; and from $\{x = \ve\}$ to $\{y = -\delta\}$, respectively), are obtained by substituting the appropriate values of $\pm \ve$ and/or $\pm \delta$ into the expression for $A(\ve, \delta)$ given above.

\item[(d)] For any $i\neq j$ we have $\bar{A}_{Q_j}(\ve,\delta)A_{Q_i}^{\frac{1}{\lambda}}(\ve,\delta)=\exp\left(R_1(\ve)+R_2(\delta)\right)$, where
$\lim\limits_{\ve\to 0} R_1(\ve)=\lim\limits_{\delta\to 0} R_2(\delta)=0$. Of course, $\bar{A}_{Q_i}(\ve,\delta)A_{Q_i}^{\frac{1}{\lambda}}(\ve,\delta)=1$.
\end{enumerate}
\end{lemma}

\subsection{Proof of Theorem \ref{t:v1}}

\begin{proof}
If $\theta_i$ is an $\mathcal{A}$-type characteristic direction with weights $(p,q)$, then the transition maps $\Delta_{\theta_i,\theta_i+\pi}^\ve$
from the section $\{y=\tan(\theta_i-\bar{\ve}) x\}$ to
$\{y=\tan(\theta_i+\bar{\ve}) x\}$, where $\bar{\ve}=\arctan(\ve)$, can be computed
via Proposition \ref{p:propotocha} by using the corresponding systems \eqref{e:sistemesi}  obtained taking a rotation of angle $-\theta_i$. This yields to a transition map $\Delta_{0,\pi}^{\theta_i,\ve}(x,-\ve x)=(D_{0,\pi}^{\theta_i,\ve} x+o(x),\ve D_{0,\pi}^{\theta_i,\ve} x+o(x))$. A straightforward computation gives that
 the first component of this map, written in polar coordinates  $(\rho,\theta)$, is
$
\tilde{\Delta}_{0,\pi}^{\theta_i,\ve}(\rho)=D_{0,\pi}^{\theta_i,\ve} \rho+o(\rho).
$
Hence,
$$
\tilde{\Delta}_{\theta_i,\theta_i+\pi}^\ve(\rho)=D_{\theta_i,\theta_i+\pi}^\ve \rho+o(\rho) \mbox{ where } D_{\theta_i,\theta_i+\pi}^\ve=D_{0,\pi}^{\theta_i,\ve}.
$$
From Corollary \ref{c:producte} we have:
$$
\lim\limits_{\ve\to 0} D_{\theta_i+\pi}^\ve D_{\theta_i}^\ve=
\exp\left\{ \frac{2}{\lambda_i p_i} \mathrm{PV} \int_{-\infty}^{\infty}
\left.\frac{\partial}{\partial x} \left(\frac{X_i(x,v)}{V_i(x,v)}\right)\right|_{x=0} dv\right\}
$$ 
if the weights associated with the blow-up of the direction $p_i$ and $q_i$ are odd; and $\lim\limits_{\ve\to 0} D_{\theta_i+\pi}^\ve D_{\theta_i}^\ve=1$ otherwise.

Without loss of generality, assume that $\{\theta=0\}$ is not a
characteristic direction, so we can take $\{x\gtrsim 0\}$ as transversal section or, in polar coordinates, $\{\rho\gtrsim 0;\theta=0\}$.

Let $\ve>0$ be small enough such that
$\{(\theta_j-\bar{\ve},\theta_j+\bar{\ve})\}_{j\in\{1,\cdots,n\}}$, where $\bar{\ve}=\arctan(\ve)$, is a
collection of disjoint intervals. Set
$$S_\ve=\bigcup\limits_{j=1}^{n}\left((\theta_j-\bar{\ve},\theta_j+\bar{\ve})\cup (\theta_j+\pi-\bar{\ve},\right. \theta_j \left.+\pi+\bar{\ve})\right)\mbox{ and } I_\ve=[0,2\pi)\setminus S_{\ve}.$$

Consider the polar equation associted with the initial differential system
$$
\frac{d\rho}{d\theta}=\mathcal{F}(\theta)\rho+o(\rho).
$$
Then the return map $\Pi(\rho)$ decomposes as:
$$
\Pi(\rho)=T_{2n+1}^\ve\circ\tilde{\Delta}_{\theta_n+\pi}^\ve\circ T_{2n}^\ve\circ\cdots\circ\tilde{\Delta}_{\theta_1+\pi}^\ve\circ T_{n+1}^\ve\circ
\tilde{\Delta}_{\theta_n}^\ve\circ T_{n}^\ve\circ\cdots\circ
T_{2}^\ve\circ
\tilde{\Delta}_{\theta_1}^\ve\circ T_{1}^\ve(\rho),
$$
where $T_1^\ve$ denotes the regular transition map from
$\{\theta=0\}$ to $\{\theta=\theta_1-\bar{\ve}\}$;  $T_j^\ve$ denotes the
regular transition from $\{\theta=\theta_{j-1}+\bar{\ve}\}$ to
$\{\theta=\theta_j-\bar{\ve}\}$; and $T_{2n+1}^\ve$ the transition map from $\{\theta=\theta_{n}+\pi+\bar{\ve}\}$ to $\{\theta=2\pi\}$.
A computation shows that for all $i=1,\ldots 2n+1$, the maps $T_j^\ve$ have the form
$$ T_j^\ve(\rho)\, =\,t_j(\ve)\,\rho \,+\,
o(\rho)$$ where $$ t_1^\ve=
\exp\left\{\int_{0}^{\theta_{1}-\bar{\ve}}
\mathcal{F}(\theta) d\theta \right\},\,  t_j^\ve=
\exp\left\{\int_{\theta_{j-1}+\bar{\ve}}^{\theta_{j}-\bar{\ve}}
\mathcal{F}(\theta) d\theta \right\} $$	and  $$
t_{2n+1}^\ve=
\exp\left\{\int_{\theta_{n}+\bar{\ve}}^{2\pi}
\mathcal{F}(\theta) d\theta \right\}.$$

Since $\Pi(\rho)$ is a composition of  maps with non-vanishing linear leading terms, we can write
$\Pi(\rho) = \eta\, \rho+o(x),$ where

\begin{equation}\label{e:final}
\eta\, =\, \left(\prod_{j=1}^{n} D_{\theta_j+\pi}^{{\ve}}D_{\theta_j}^{{\ve}} \right)\cdot \exp
\left\{\int_{I_\ve} \mathcal{F}(\theta) d\theta
\right\}.
\end{equation}

The terms depending on ${\ve}$ in the integrals appearing in
each $D_{\theta_j+\pi}^{\ve}D_{\theta_j}^{\ve}$ (given in Proposition \ref{p:propotocha}) are
non--singular. Therefore, taking $\ve\rightarrow 0$ in
equation (\ref{e:final}), and using Corollary \ref{c:producte} we obtain equation \eqref{e:tv1}.
\end{proof}

\subsection{Proof of Theorem \ref{t:v2}}

\begin{proof}
Let $\{0,\pi/2\}$ be a $\mathcal{B}$-type tuple of characteristic directions appearing using weighted polar coordinates $x=\rho^{m}\cos\theta$ and $y=\rho^{n}\sin\theta$ with $m=p_2$ and $n=q_2-p_2$. Let $\theta_j \in \{0,\pi/2,\pi,3\pi/2\}$, $j=1,2,3,4$, denote the characteristic directions. Now we establish the relations between the transversal sections in Cartesian and weighted polar coordinates just taking into account that a constant angle $\theta = \theta^\dag$ is written as $y^m/x^n = \sin^m\theta^\dag / \sin^n\theta^\dag$.

We begin with the transversal sections corresponding to the characteristic directions $\theta_j$, for $j = 1,3$.  Given $\ve> 0$ sufficiently small, let $\tilde{\ve} > 0$ be the smallest positive solution of the equation $\sin^m(\tilde{\ve}) = \ve \cos^{n} (\tilde{\ve})$. Then, using the above relation and applying the trigonometric angle addition formulas, we easily obtain that, for $j=1,3$, the transversal sections $\{\theta = \theta_j \pm \tilde{\ve}\}$, in weighted polar coordinates, coincide with the local transversal sections near the origin contained in the corresponding branches of the curves $\{ y = \pm \ve^{1/m} |x|^{n/m} \}$ taking a counterclockwise sense of the rotation flow.

For the transversal sections corresponding to $\theta_j$, for $j=2,4$, we proceed similarly.
Given the value $\tilde{\ve}>0$ defined above, let $\bar{\ve}>0$ be the smallest positive solution of $\bar{\ve} \cos^m(\tilde{\ve}) = \sin^{n} (\tilde{\ve})$. Again, using the above relation and the trigonometric angle addition formulas we straightforwardly obtain that for $j=2,4$, the sections $\{\theta = \theta_j \pm \tilde{\ve}\}$ give rise to the local transversal sections near the origin lying in appropriate branches of the curves  $\{ y = \pm \bar{\ve}^{-1/m} |x|^{n/m} \}$. Observe that, as $\tilde{\ve} \to 0$, we have $\ve \to 0$ and $\bar{\ve}\to 0$.

With this notation and assuming the flow rotates counterclockwise, the transition maps $\Delta_{0,\pi}^\ve(x, \mp \ve^{\frac{1}{m}}|x|^\frac{n}{m})$, as stated in Proposition \ref{p:propotochb}, correspond with the transition maps, in weighted polar coordinates,
$\tilde{\Delta}_{0}^{\tilde{\ve}}:\{\theta=-\tilde{\ve} \} \to \{\theta= \tilde{\ve}\}$ and $\tilde{\Delta}_{\pi}^{\tilde{\ve}}:\{\theta=\pi-\tilde{\ve}\} \to \{\theta=\pi+\tilde{\ve}\}$.  In the same way, the transition maps $\Delta_{\pi/2,3\pi/2}^{\bar{\ve}}(x,\pm \bar{\ve}^{-\frac{1}{m}}|x|^\frac{n}{m})$,  defined in Proposition \ref{p:propotochb2}, correspond with the transition maps, in weighted polar coordinates,
$\tilde{\Delta}_{\pi/2}^{\tilde{\ve}} : \{\theta= \pi/2-\tilde{\ve}\} \to \{\theta = \pi/2+\tilde{\ve}\}$ and $\tilde{\Delta}_{3\pi/2}^{\tilde{\ve}}:\{\theta = 3\pi/2-\tilde{\ve} \} \to \{\theta = 3\pi/2 + \tilde{\ve} \}$.

We claim that
$$
\tilde{\Delta}_{0,\pi}^{\tilde{\ve}}(\rho)=\left(D_{0,\pi}^{\ve}\right)^\frac{1}{m} \rho+o(\rho)
\mbox{ and }
\tilde{\Delta}_{\frac{\pi}{2},\frac{3\pi}{2}}^{\tilde{\ve}}(\rho) = \left(D_{\frac{\pi}{2},\frac{3\pi}{2}}^{\bar{\ve}}\right)^\frac{1}{m} \rho+o(\rho).
$$
We only prove the claim for $\tilde{\Delta}_{0}^{\tilde{\ve}}$, the other cases follow analogously. The                                                          first component of the map $\Delta_{0}^\ve$ is $\Delta_{0}^\ve(x,- \ve^{\frac{1}{m}}|x|^\frac{n}{m})_1 = D_0^\ve\,x+o(x)$ and, by definition $x=\rho^m\cos(-\tilde{\ve})=\rho^m\cos(\tilde{\ve})$ so that
$$
D_0^\ve \, x + o(x)=\left(\tilde{\Delta}_{0}^{\tilde{\ve}}(\rho) \right)^m \cos(\tilde{\ve}).
$$
Hence
$$
D_0^\ve \, \rho^m \cos(\tilde{\ve}) + o(\rho^m) = \left(\tilde{\Delta}_{0}^{\tilde{\ve}}(\rho) \right)^m \cos(\tilde{\ve}),
$$
and therefore $\tilde{\Delta}_{0}^{\tilde{\ve}}(\rho)= \left(D_0^\ve\right)^\frac{1}{m}\,\rho+o(\rho)$, proving the claim.

\medskip

We take the transversal section $\{\theta=\bar{\theta}\}$, with $\bar{\theta}\in(0,\pi/2)$.
Then, the return map $\Pi(\rho)$ decomposes as:
$$
\Pi(\rho)=T_{5}^{\tilde{\ve}}\circ\tilde{\Delta}_{0}^{\tilde{\ve}}
\circ T_{4}^{\tilde{\ve}} \circ \tilde{\Delta}_{\frac{3\pi}{2}}^{\tilde{\ve}}
\circ T_{3}^{\tilde{\ve}} \circ
\tilde{\Delta}_{\pi}^{\tilde{\ve}} \circ T_{2}^{\tilde{\ve}} \circ
\tilde{\Delta}_{\frac{\pi}{2}}^{\tilde{\ve}}\circ T_{1}^{\tilde{\ve}}(\rho),
$$
where $T_{1}^{\tilde{\ve}}$ denotes the regular transition map from
$\{\theta=\bar{\theta}\}$ to $\{\theta=\pi/2-\tilde{\ve}\}$; and the rest of maps $T_{j}^{\tilde{\ve}}$ are the obvious regular transition maps connecting the transition maps $\tilde\Delta^{\tilde{\ve}}_{\theta_j}$,  in polar coordinates.  These maps have the form $T_j^{\tilde{\ve}}(\rho)\, =\,t_j({\tilde{\ve}})\,\rho \,+\, o(\rho)$, where the coefficients $t_j$ can be obtained integrating the first variation equations of the $(m,n)$-weighted polar differential equation $\frac{d\rho}{d\theta}=\mathcal{F}(\theta)\rho+o(\rho)$ associated with the initial differential system as done in the the proof of Theorem~\ref{t:v1}.

Using the above decomposition, one gets $\Pi(\rho) = \eta\, \rho+o(\rho)$ with
\begin{align}
\eta\, &=\, \left(D_{0}^{\ve}\,D_{\frac{\pi}{2}}^{\bar{\ve}}
\,D_{{\pi}}^{\ve}\,D_{\frac{3\pi}{2}}^{\bar{\ve}}\right)^\frac{1}{m}
\,t_5(\tilde{\ve})\cdots t_1(\tilde{\ve})\label{e:final2}\\
&=\, \left(D_{0}^{\ve}\,D_{\frac{\pi}{2}}^{\bar{\ve}}
\,D_{{\pi}}^{\ve}\,D_{\frac{3\pi}{2}}^{\bar{\ve}}\right)^\frac{1}{m}
\cdot \exp
\left\{\int_{I_{\tilde{\ve}}} \mathcal{F}(\theta) d\theta
\right\}, \nonumber
\end{align}
where the last integral is computed over $I_{\tilde{\ve}}=[0,2\pi)\setminus S_{\tilde{\ve}}$ with
$$
S_{\tilde{\ve}} = \left(-\tilde{\ve}, \tilde{\ve} \right) \cup \left(\frac{\pi}{2}- \tilde{\ve}, \frac{\pi}{2} + \tilde{\ve} \right)\cup\left(\pi- \tilde{\ve}, \pi + \tilde{\ve} \right) \cup \left(\frac{3\pi}{2}-\tilde{\ve},\frac{3\pi}{2}+ \tilde{\ve} \right).
$$
By Corollary \ref{c:producteb}, letting $\tilde{\ve} \to 0$ (and thus both $\ve$ and $\bar{\ve}$ tend to zero) in equation \eqref{e:final2}, we obtain the expression for $\eta$ in \eqref{e:tv2}.
\end{proof}

\section{Proofs on the extensions of particular minimal models} \label{S-Proofs3}

\subsection{Proof of Proposition \ref{prop-1221}}

\begin{proof}
An straightforward verification shows that the first extension $\mathcal{X}^{[1]}$ of the minimal model  $\mathcal{X}$ given by \eqref{Ex-1} is just the family \eqref{Ex-1-ext}. We notice that the family \eqref{Ex-1-ext} breaks the time-reversible symmetry with respect to both axis but it keeps the set $\Omega_{11} =\{0,\pi/2\}$ of characteristic directions in polar coordinates. Moreover, in polar coordinates
$$
\dot{\theta}=(D-B)	\cos^2\theta\sin^2\theta+o(\rho).
$$
Hence, to ensure that the orbits turn around in counterclockwise direction we fix $D-B>0$. As explained below, we will restrict ourselves to the parameter values of \eqref{Ex-1-ext} that maintain the monodromic nature of the origin.

We claim that in the parameter space $\hat{\Lambda}$ both $\mathcal{X}$ and $\mathcal{X}^{[1]}$ share the blow-up desingularization geometry. In details, we take the characteristic direction $\theta_1^* = 0$ of \eqref{Ex-1-ext} and we do the blow-up $(x, y) \mapsto (z_2,w)$ with $z_2=x^{2}/y$ and $w = y/x$. We remove the common factor $z_2^2 w^3$ and we obtain the polynomial differential system $\dot{z}_2 = z_2[ 2B-D + o(1)]$, $\dot{w} = w[D-B+o(1)]$,
 Since $(2B-D) (D-B) < 0$ on $\Lambda^*$ it follows that the origin now becomes a hyperbolic saddle.

Next, we perform the blow-up  $(x,y) \mapsto (x, w_2)$ with $w_2 = y/x^2$ and a time-rescaling obtained by removing the factor $x^3$ and a reversion of time in order to preserve the fixed sense of rotation. Then \eqref{Ex-1-ext} is transformed into the polynomial differential system $\dot{x} = X(x, w_2)$, $\dot{w}_2 = W_2(x, w_2) = P_2(w_2) + x Q_2(x, w_2)$  where $Q_2$ is a polynomial and where $\dot{w}_2|_{x=0} =  P_2(w_2) := -C +(2 a_2-b_2) w_2 + (2B-D) w_2^2 $, that has no real roots and it is positive since  $(2 B-D)>0$ and provided that $\Delta_1 := (2 a_2-b_2)^2 + 4 C (2 B-D) < 0$. This gives us a regular part of the flow.

Associated to the characteristic direction $\theta_2^* = \pi/2$ of \eqref{Ex-1-ext} we can either perform the rotation of $-\pi/2$ and do the procedure explained in Remark \ref{remark-rotation-A} or just do equivalently first the blow-up $(x, y) \mapsto (z, w_1)$ with $z = x/y$ and $w_1=y^{2}/x$. After a time-rescaling that removes the common factor $w_1^{2} z^{3}$ and ensure the assumed sense of rotation, we obtain the polynomial differential system $\dot{z} = z[D-B+o(1)]$,
$\dot{w}_1 = w_1[B-2D+o(1)]$.  Again, the fact that $(D-B)(B-2D)< 0$ on $\Lambda^*$ guarantees that the origin is a hyperbolic saddle singularity.

Finally, we perform the blow-up $(x,y) \mapsto (z_1, y)$ with $z_1 = x/y^2$.  After a pertinent rescalling to remove the common factor $y^3$ and ensures the sense of rotation we obtain the system $\dot{z}_1 = Z_1(z_1, y) = P_1(z_1) + y Q_1(z_1, y)$, $\dot{y} = Y(z_1, y)$, with $\dot{z}_1|_{y=0} = P_1(z_1) = -A + (2 b_1-a_1 ) z_1 + (2D-B) z_1^2$ that has no real roots provided its discriminant $\Delta_2 := (2 b_1-a_1 ))^2 + 4 A (2 D-B) < 0$. This gives another regular part of the flow.

The former computations prove the claim and, in particular, they establishes the monodromy of the origin of \eqref{Ex-1-ext} on $\hat{\Lambda}$.

\medskip

Observe that the characteristic directions $\theta_* = 0$ and $\theta_* = \pi/2$ are of type $\mathcal{A}$.  Moreover, since
$W(\mathbf{N}(\mathcal{X}))=\{(2,1),(1,2)\}$, it follows that $\mathcal{A}_* = \emptyset$. Therefore, according to Theorem \ref{t:v1}, the linear term $\eta$ of the Poincaré map is given by:
$$
\displaystyle{\eta=\exp\left\{\mathrm{PV}\int_0^{2\pi} \mathcal{F}(\theta)d\theta  \right\}=\exp\left\{\mathrm{PV} \int_0^{2 \pi} \frac{B \cot\theta + D \tan\theta}{D-B} d \theta \right\}= 1.}
$$
where $\mathcal{F}(\theta)$ is the leading term of the polar coordinates equation associated with \eqref{Ex-1-ext}, $dr/d\theta=\mathcal{F}(\theta)r+o(r)$.

\medskip

Finally, we consider the second extension $\mathcal{X}^{[2]}$, by adding to $\mathcal{X}^{[1]}$ arbitrary higher order terms with respect to $\mathbf{N}(\mathcal{X})$, that is, $\mathcal{X}^{[2]} = \mathcal{X}^{[1]} + \mathcal{X}_*$ where $\mathcal{X}_* = \left(\sum_{(i, j) \in S} a_{i,j-1} x^i y^{j-1} \right) \partial_x + \left(\sum_{(i, j) \in S} b_{i-1,j} x^{i-1} y^{j} \right) \partial_y$ being $(i,j) \in S$ with $S = \{ (i,j) \in \mathbb{N}^2 : 2 i + (j - 1) > 2 + 3, i + 2 (j - 1) > 1 + 3 \}$.  Therefore, we obtain that $\mathcal{X}^{[2]}$ is given by  \eqref{Ex-1-ext-HOT}.

According to \eqref{The-set-S-HOT} and recalling that $W(\mathbf{N}(\mathcal{X})) = \{(2, 1), (1,2) \}$ and $r_1 = r_2 = 3$ by \eqref{The-rk}.
In this case, it can be verified that adding the terms associated with $\mathcal{X}_*$ does not affect the $\mathcal{A}$-type nature of the characteristic direction, and therefore we remain under the scope of Theorem~\ref{t:v1}. It can also be verified that these terms do not affect the expression of the corresponding function $\mathcal{F}(\theta)$. This proves the last part of the proposition.
\end{proof}

\subsection{Proof of Proposition \ref{p:novapropTh2}}

\begin{proof}
An easy verification shows that the vector field \eqref{Toy1-XD} is the first extension of the minimal model \eqref{Ex-2}, and that the origin keeps the unique characteristic direction $\theta_*=0$ in polar coordinates if and only if the additional condition $\delta_1 := (a_1 - b_1)^2 + 4 A (D -B) < 0$ holds. If we, in addition, fix $D-B>0$ then the orbits turn around in counterclockwise direction.

The origin of $\mathcal{X}$ and $\mathcal{X}^{[1]}$ share the desingularization sequence of blow-ups. More specifically, doing the blow-up $(x, y) \mapsto (w, z_2)$ with $w = y/x$ and $z_2=x^{2}/y$ and remove the common factor by time-rescaling to obtain a polynomial differential system $\dot{z}_2 = z_2 [ (2 B-D)+o(1)]$, $\dot{w} =  w[(D-B)+ o(1)]$ having a hyperbolic saddle at the origin with hyperbolic ratio $\lambda_1 = (2 B - D)/(B - D) > 0$ on $\Lambda$.
\medskip

The blow-up $(x,y) \mapsto (w_2, x)$ with $w_2 = y/x^{2}$ and the corresponding time-rescaling we get a polynomial differential system $\dot{x} = X(x, w_2)$,
$\dot{w}_2 = W_2(x, w_2)$   such that $\dot{w}_2|_{x=0} = -C+  (2 a_2-b_2) w + (2B-D)w^2$ has no real roots under the discriminant condition $\Delta_1 < 0$.

The second extension is
$\mathcal{X}^{[2]} = \mathcal{X}^{[1]} + \mathcal{X}_*$, where $\mathcal{X}_* = \left(\sum_{(i, j) \in S} a_{i,j-1} x^i y^{j-1} \right) \partial_x + \left(\sum_{(i, j) \in S} b_{i-1,j} x^{i-1} y^{j} \right) \partial_y$ being $(i,j) \in S \subset \mathbb{N}^2$ if and only if $(i, j)$ lies in the upper half plane with respect to $\mathbf{N}(\mathcal{X})$. In other words, $S = \{ (i,j) \in \mathbb{N}^2 : i + (j - 1) > 1 + 2, i + 2 (j - 1) > 1 + 3 \}$. We emphasize that the addition of higher order terms with respect to $\mathbf{N}(\mathcal{X})$ does not affect the characteristic directions at the origin. Regarding the blow-up processes for the computation of the saddle and the regular part, we have verified that the addition of the terms defining $\mathcal{X}^{[1]}$ and $\mathcal{X}^{[2]}$ does not affect the $\mathcal{A}$-type nature of the characteristic direction.

The expression for the linear term of the return map $\Pi$ follows directly from Theorem \ref{t:v1}, taking into account $\mathcal{A}^* = \emptyset$  is empty because $\theta_*=0$ is an $\mathcal{A}$-type characteristic direction only when the weights $(p,q)=(1,2)\in W(\mathbf{N}(\mathcal{X}))$ are used, and that  for both $\mathcal{X}^{[1]}$ and $\mathcal{X}^{[2]}$ we get using polar coordinates:
$$
 \mathcal{F}(\theta)
=
\frac{ 2( b_1 +   \cot\theta (A + D +  \cot\theta (a_1 +  B \cot\theta))) \sin^2\theta }{ D-A - B  + (A - B +  D) \cos(2 \theta) + (b_1-a_1 ) \sin^2 \theta)}.
$$
\end{proof}

\subsection{Proof of Proposition \ref{p:propoexttipusb}}

\begin{proof}
Let $\mathcal{X}$ be the vector field associated to the minimal model \eqref{Toy-3}. Its first extension $\mathcal{X}^{[1]}$ is just family \eqref{Toy-3-ext}. In order to apply the type $\mathcal{B}$ desingularization scheme and to compute the linear part associated with the Poincar\'e map by using Theorem \ref{t:v2}, we  take $(p_2, q_2)=(1,3)\in W(\mathbf{N}(\mathcal{X}))$, and we consider the weighed polar coordinates $x = \rho^{p_2} \cos\theta=\rho \cos\theta$ and $y = \rho^{q_2-p_2} \sin\theta= \rho^{2} \sin\theta$. In these $(1,2)$-weighted polar coordinates there appear two characteristic directions $\Omega_{12} =\{0,\pi/2\}$ and
$$
\dot{\theta}=\frac{(D-2B)\cos^2\theta\sin^2\theta}{2-\cos^2\theta}+o(\rho)
$$
To ensure the counter-clockwise rotation sense we fix $D-2B>0$.

We claim that for $\mathcal{X}^{[1]}$, and therefore for $\mathcal{X}$, these characteristic directions are of type  $\mathcal{B}$, when we take the restricted monodromic parameter space $\hat{\Lambda}$. More precisely, to desingularize $\theta_*=0$, we perform the following blow-ups to the first extension  \eqref{Toy-3-ext}: $(x,y) \mapsto (z, w)$ with $z =x^3/y$ and $w=y/x^2$, with the appropriate time-rescaling removing the common factor $z^3 w^4$ transforming  \eqref{Toy-3-ext} into a polynomial vector field
$\dot{z} = z[3 B-D  + o(1)]$, $\dot{w} = w[D-2B  + o(1)]$,  so that $(3 B-D) (D-2B)< 0$ when we restrict to $\Lambda^\ddag$ yielding a hyperbolic saddle at the origin with hyperbolic ratio $\lambda_0=(2B-D)/(3B-D)$.

Now we do the second blow-up $(x,y) \mapsto (x, v)$ with $v = y/x^3$ and remove the common factor $x^4$ to get a polynomial differential system $\dot{x} = x(a_2+Bv+x^2X_2(x,v))$, $\dot{v} =  C +(b_2-3 a_2) v + (D- 3 B) v^2+x^2V_2(x,y)$,  where both $X_2$ and $V_2$ are polynomials. We observe that $\dot{v}|_{x=0}$
has no real roots provided that $\Delta_1 := (b_2-3 a_2)^2 + 4 C (3 B - D) < 0$. Hence $\{x=0\}$ is a regular solution. Since $D-2B>0$ we have $3B-D<0$ and therefore it implies that $C>0$.

To desingularize $\theta_* =\pi/2$ we perform the blow up $(x,y) \mapsto (z,w)$ with $z= y/x$, $w = x^2/y$ and the time-rescaling multiplying by $z^3w^2$, which gives the polynomial system $\dot{z}=z(D-B+o(1))$ and $\dot{w}=w(2B-D+o(1))$, thus having a hyperbolic saddle at the origin  if $(D-B)(2B-D)<0$, with hyperbolic ratio $\lambda_{\frac{\pi}{2}}=(D-2B)/(D-B)$.

 Finally, the blow up $(x,y) \mapsto (u,y)$ with $u=x/y$ and removing the factor $y^2$, which gives a regularized system $\dot{u}= A+(a_1-b_1)u+(B-D)u^2+y^2U_2(u,y)$, $\dot{y}=y(b_1+Du+y^2Y_2(u,y))$, i.e. such that $y=0$ is invariant and has no singular points inside when $\Delta_2=(a_1-b_1)^2+4A(D-B) < 0$. Observe that since $2B-D<0$ we have that $D-B>0$ and therefore $A<0$.

\medskip

We have seen that the characteristic directions of both $\mathcal{X}$ and $\mathcal{X}^{[1]}$ are $\mathcal{B}$-type characteristic directions, so we can apply Theorem \ref{t:v2} to compute the linear term~$\eta$. To do this, we observe that the polar coordinates associated equation is
$$
\frac{d\rho}{d\theta}= \mathcal{F}(\theta) +o(\rho) =\frac{B \cot\theta + D \tan\theta }{D-2B} +o(\rho),
$$ hence $\mathrm{PV} \int_0^{2\pi} \mathcal{F}(\theta)d\theta=0.$ We also observe that
$$
\frac{\partial}{\partial x} \left.\left(\frac{X(x,v)}{V(x,v)}\right)\right|_{x=0}=\frac{Bv+a_2}{(D-3B)v^2+(b_2-3a_2)v+C}
$$ and $\lambda_0=(2B-D)/(3B-D)$; and that
$$\frac{\partial}{\partial y}\left.\left(\frac{Y(u,y)}{U(u,y)}\right)\right|_{y=0}=\frac{Du+b_1}{(B-D)u^2+(a_1-b_1)u+A}
$$and $\lambda_{\frac{\pi}{2}}=(D-2B)/(D-B)$.
A direct application of Theorem \ref{t:v2} gives that for $\mathcal{X}^{[1]}$ the coefficient $\eta$ is given by \eqref{eta-Toy-3-ext}.

By adding arbitrary higher order terms with respect to $\mathbf{N}(\mathcal{X})$, we obtain $\mathcal{X}^{[2]} = \mathcal{X}^{[1]} + \mathcal{X}_*$ where $\mathcal{X}_* = \left(\sum_{(i, j) \in S} a_{i,j-1} x^i y^{j-1} \right) \partial_x $ $+ \left(\sum_{(i, j) \in S} b_{i-1,j} x^{i-1} y^{j} \right) \partial_y$ being $(i,j) \in S$. According with \eqref{The-set-S-HOT} and taking into account the values $(p_1, q_1)=(1,1)$, $(p_2, q_2)=(1,3)$ and $r_1=2$ and $r_2=4$ from \eqref{The-rk}, we obtain the set $S = \{ (i,j) \in \mathbb{N}^2 : i + j > 4, i + 3 j  > 8\}$. A computation shows that for $\mathcal{X}^{[2]}$, $\{0,\pi/2\}$ is also a $\mathcal{B}$-type 2-tuple of characteristic directions. It can also be verified that these terms do not affect the computation of $\eta$.
\end{proof}

\subsection{Proof of Proposition \ref{Prop-manyosa-Famili-I-ext}}

\begin{proof}
Let $\mathcal{X}$ be the vector field of the minimal model \eqref{toy-manyosa-Famili-I}. The most general extended family $\mathcal{X}^{[1]}$ with $\mathcal{X} \subset \mathcal{X}^{[1]}$, $\mathbf{N}(\mathcal{X}^{[1]}) = \mathbf{N}(\mathcal{X})$ and $\mathcal{X}^{[1]} = (\mathcal{X}^{[1]})_\Delta$ is \eqref{manyosa-ext}. The geometry of the desingularizing blow-ups of the characteristic direction $\theta_*=0$ of $\mathcal{X}$ and $\mathcal{X}^{[1]}$ coincide since \eqref{rest-exp-regularpart-ext} holds in all the monomials of \eqref{manyosa-ext}. This characteristic direction is computed using polar coordinates and we recall that $(1,1) \in W(\mathbf{N}(\mathcal{X}))$. Notice that family \eqref{manyosa-ext} also contains \eqref{ejemplo-manyosa} and they coincide when $c_1= d_1= -d_2 = 1$.
\medskip

The polar vector field associated to \eqref{manyosa-ext} has the angular component $\dot{\theta}= G_0(\theta) + O(\rho)$ with $G_0(\theta) = \sin^2\theta \mathcal{P}(\theta)$ with $\mathcal{P}(\theta) = 4 \cos^2\theta + (d_1- c_1) \sin\theta \cos\theta + \sin^2\theta$. The monodromic necessary condition $G_0(\theta) \geq 0$ holds if and only if $\delta_1(c_1, d_1) = -16 + (c_1 - d_1)^2 \leq 0$ and there is only one characteristic direction $\theta_* = 0$ when $\delta_1 < 0$, the case that we are analyzing. As we will see, $0$ is a $\mathcal{A}$-type characteristic direction for the origin of $\mathcal{X}^{[1]}$.

We perform the blow-ups associated to the characteristic direction $\theta_* = 0$ to system \eqref{manyosa-ext} as we did in the proof of Theorem \ref{t:Teo-main1} for $\mathcal{A}$-type directions.

We take $(x, y) \mapsto (z_2,w)$ with  $z_2=x^{3}/y$  and $w = y/x$ with the corresponding time-rescaling. It produces the hyperbolic saddle $\dot{z}_2 = z_2[4  + o(1)]$,
$\dot{w} = w[-4  + o(1)]$. Hence its hyperbolic ratio is $\lambda=1$. Now we perform the change $(x, y) \mapsto (x, w_2)$ with $w_2 = y/x^3$ and, denoting the resulting system $\dot{x} = X(x, w_2),$ $\dot{w}_2 = W_2(x, w_2)$, we obtain a regular solution at $\{x=0\}$. A computation gives:
$$
\int_{-\infty}^{\infty} \frac{\partial}{\partial x} \left. \left(\frac{X(x, w)}{W_2(x, w)}\right) \right|_{x=0} dw = \int_{-\infty}^{\infty} \frac{-a}{-2 + (3 a - d_2) w - 4 w^2} dw$$
$$
= 2 a \sqrt{\Delta_1^{-1}(a, d_2)} \, \pi,
$$
provided that the discriminant $\Delta_1(a, d_2) := 32 - (d_2 -3 a)^2 > 0$. Notice that we have obtained the natural extension $\Delta_1(a, -1) = \Delta(a)$ of system \eqref{ejemplo-manyosa}.

\medskip

Now we integrate the variational part associated to the regular part. First we write the polar system associated to \eqref{manyosa-ext}, that is, $\dot{\theta} = G_0(\theta) + O(\rho)$, $\dot{\rho} = R_1(\theta) \rho + O(\rho^2)$ with $R_1(\theta) = \sin^2\theta \mathcal{Q}(\theta)$ and $\mathcal{Q}(\theta) = c_1 \cos^2\theta + 3 \sin\theta \cos\theta + d_1 \sin^2\theta$. The fact that $\mathcal{F}(\theta):=R_1(\theta)/G_0(\theta)=\mathcal{Q}(\theta)/\mathcal{P}(\theta)$ where $\mathcal{P}$ has no real roots because of $\delta_1 < 0$ guarantee that the following integral is non-singular and takes the value
$$
\mathrm{PV} \int_0^{2 \pi} \mathcal{F}(\theta) d \theta = \int_0^{2 \pi} \frac{\mathcal{Q}(\theta)}{\mathcal{P}(\theta)} d \theta =  \frac{2(c_1+d_1) \pi}{\sqrt{-\delta_1}}.
$$
From these computations and using Theorem \ref{t:v1} we obtain the explicit expression \eqref{eta-Toy1-XD_ManyosaFamilyI} of $\eta$ for $\mathcal{X}^{[1]}$.
\newline

Now we add to $\mathcal{X}^{[1]}$ arbitrary higher order terms with respect to $\mathbf{N}(\mathcal{X})$, that is, we consider $\mathcal{X}^{[2]} = \mathcal{X}^{[1]} + \mathcal{X}_*$ where $\mathcal{X}_* = \left(\sum_{(i, j) \in S} a_{i,j-1} x^i y^{j-1} \right) \partial_x + \left(\sum_{(i, j) \in S} b_{i-1,j} x^{i-1} y^{j} \right) \partial_y$ being $(i,j) \in S \subset \mathbb{N}^2$ if and only if $(i, j)$ lies in the upper half plane with respect to $\mathbf{N}(\mathcal{X})$. The set $S$ is also characterized in \eqref{The-set-S-HOT} where $(p_1, q_1)=(1,1)$, $(p_2, q_2)=(1,3)$ are the weights of $W(\mathbf{N}(\mathcal{X}))$, and $r_1$ and $r_2$ are given in \eqref{The-rk}. Since the exponents of the minimal model are $(\alpha, \beta, r, s) = (1, 1, 4, 2)$ then $r_1= 2$, $r_2 = 4$ and we obtain the set $S$ defined in the statement of the proposition.

We notice that the addition of higher order terms with respect to $\mathbf{N}(\mathcal{X})$ does not affect the characteristic directions at the origin, but if we apply the same desingularization scheme of $\theta_*=0$ of  $\mathcal{X}^{[1]}$ to $\mathcal{X}^{[2]}$ (the one corresponding to a type $\mathcal{A}$ direction), we do not obtain  polynomial fields in the local charts. In summary, $\theta_*=0$ is not a type $\mathcal{A}$ characteristic direction of $\mathcal{X}^{[2]}$. Nevertheless, fortunately we have checked that the origin of $\mathcal{X}^{[2]}$ has $\{0,\pi/2\}$ as a $\mathcal{B}$-type 2-tuple of characteristic directions so that we can compute $\eta$ for $\mathcal{X}^{[2]}$ too finishing the proof. We emphasize that the value of $\eta$ computed for $\mathcal{X}^{[2]}$ using its two different desingularizations is the same as it must be.
\end{proof}



\begin{thebibliography}{99}
\bibitem{AGGMed} {\sc A. Algaba, C. Garc\' \i a, J. Gin\'e}, {\it Geometric criterium in the center problem}, Mediterr. J. Math. {\bf 13} (2016), 2593--2611.

\bibitem{AGG} {\sc A. Algaba, C. Garc\' \i a, J. Gin\'e}, {\it Center conditions to find certain degenerate centers with characteristic directions}, Math. Comput. Simulation 215 (2024) 628--638.

\bibitem{AGR} {\sc A. Algaba, C. Garc\' \i a, M. Reyes}, {\it Characterization of a monodromic singular point of a planar vector field}, Nonlinear Anal. {\bf 74} (2011), 5402--5414.

\bibitem{AGR2} {\sc A. Algaba, C. Garc\' \i a, M. Reyes}, {\it A new algorithm for determining the monodromy of a planar differential system}, Appl. Math. Comput. {\bf 237} (2014), 419--429.
    
\bibitem{Be-Me}{\sc F.S. Berezovskaya and N.B. Medvedeva}, {\it The asymptotics of the return map of a singular point with fixed Newton diagram},     Journal of Soviet Mathematics {\bf 60}(6) (1992), 1765--1781.

\bibitem{BM} {\sc M. Brunella, M. Miari}, {\it Topological equivalence of a plane vector field with its principal part defined through Newton polyhedra}, J. Differential Equations {\bf 85} (1990), 338--366.
    
\bibitem{CZZ} {\sc H. Chen, R. Zhang, X. Zhang}, {\it New criterions on stability and order of analytic nilpotent foci}, J. Differential Equations 338 (2022) 352--371.
    
\bibitem{Dum} {\sc F. Dumortier,} {\it \ Singularities of vector fields on the plane}, J. Differential Equations {\bf 23} (1977), 53--106.

\bibitem{FLLL} {\sc W.W. Farr, C. Li, I.S. Labouriau, W.F. Langford}, {\it Degenerate Hopf-bifurcation formulas and Hilbert's 16th problem},
SIAM J. Math. Anal. 20 (1989), 13--29.

\bibitem{Ga-Gi} {\sc I.A. Garc\' \i a, J. Gin\'e},  {\it Center problem with characteristic directions and inverse integrating factors}, Commun. Nonlinear Sci. Numer. Simul. {\bf 108} (2022), 14 pp.

\bibitem{GaGi2}{\sc I.A. Garc\' \i a, J. Gin\'e}, {\it The Poincar\'e map of degenerate monodromic singularities with Puiseux inverse integrating factor}, Adv. Nonlinear Anal. {\bf 12} (2023), 20220314.

\bibitem{GaGi3}{\sc I.A. Garc\' \i a, J. Gin\'e}, {\it The linear term of the Poincar\'e map at singularities of planar vector fields}, J. Differential Equations, {\bf 396} (2024), 44--67.

\bibitem{GG-complex}{\sc I.A. Garc\' \i a, J. Gin\'e}, {\it Characterization of centers by its complex separatrices}, J. Differential Equations {\bf 442} (2025), 113506.
%
\bibitem{GGGrau} {\sc I.A. Garc\' \i a, J. Gin\'e, M. Grau}, {\it A necessary condition in the monodromy problem for analytic differential equations on the plane}, J. Symbolic Comput. {\bf 41} (2006) 943--958.

\bibitem{Ga-Li-Ma-Ma} {\sc A. Gasull, J. Llibre, V. Ma\~{n}osa, F. Ma\~{n}osas,} {\it \ The focus-centre problem for a type of degenerate system}, Nonlinearity {\bf 13} (2000), 699--729.

\bibitem{Ga-Ma-Ma} {\sc A. Gasull, V. Ma\~{n}osa, F. Ma\~{n}osas,} {\it \ Monodromy and stability of a class
    of degenerate planar critical points}, J. Differential Equations {\bf 217} (2005), 363--376.
    
\bibitem{GGL} {\sc H. Giacomini, J. Gin\'e, J. Llibre}, {\it The problem of distinguishing between a center and a focus for nilpotent and degenerate analytic systems} J. Differential Equations {\bf 227} (2006), no. 2, 406--426; Corrigendum, ibid. {\bf 232} (2007) 702.
    
\bibitem{G0} {\sc J. Gin\'e}, {\it Sufficient conditions for a center at a completely degenerate critical point}, Int. J. Bifurc. Chaos Appl. Sci. Eng. 12 (7) (2002) 1659--1666.
    
\bibitem{G} {\sc J. Gin\'e}, {\it On the centers of planar analytic differential systems}, Internat. J. Bifur. Chaos Appl. Sci. Engrg. {\bf 17} (2007), no. 9, 3061--3070.

\bibitem{G1} {\sc J. Gin\'e}, {\it On the degenerate center problem}, Internat. J. Bifur. Chaos Appl. Sci. Engrg. {\bf 21} (2011), no. 5, 1383--1392. 

\bibitem{GM} {\sc J. Gin\'e, S. Maza}, {\it The reversibility and the center problem}, Nonlinear Anal {\bf 74} (2011), no. 2, 695--704.

\bibitem{Ma} {\sc V. Ma\~{n}osa}, {\it On the center problem for degenerate singular points of planar vector fields}, Internat. J. Bifur. Chaos Appl. Sci. Engrg. {\bf 12} (2002), no. 4, 687--707.

\bibitem{Med1} {\sc N.B. Medvedeva}, {\it Principal term of the monodromy transformation of a monodromic singular point is linear},
Siberian Math. J. 33 (1992), 280--288.

\bibitem{M20}{\sc N.B. Medvedeva}, {\it The problem of distinguishing between a centre and a focus in the space of vector fields with given
Newton diagram}, Sb. Math. {\bf 211} (2020), 1399--1446.

\bibitem{RS}{\sc V. G. Romanovski and D.S. Shafer}, {\it The center and cyclicity problems: a computational
    algebra approach}. Birkh\"auser Boston, Inc., Boston, MA, 2009.
    
\bibitem{TLZ}{\sc Y. Tang, W. Li and Z. Zhang}, {\it Focus-center problem of planar degenerate system}, J. Math. Anal. Appl. {\bf 345} (2008), 934--940.    
\end{thebibliography}
\end{document}